\documentclass[11pt,oneside]{amsart}

\usepackage{fullpage}

\usepackage{amsthm,amsfonts,amssymb,amsmath,amsxtra,mathtools}
\usepackage[dvipsnames]{xcolor}
\usepackage[all]{xy}
\SelectTips{cm}{}
\usepackage{tikz}
\usepackage{tikz-cd}
\usepackage{verbatim}
\usepackage{todonotes}
\usepackage{mathrsfs}
\usepackage{booktabs,makecell}
\usepackage{diagbox}

\RequirePackage{xspace}
\RequirePackage{etoolbox}
\RequirePackage{varwidth}
\RequirePackage{enumitem}
\RequirePackage{tensor}
\RequirePackage{mathtools}
\RequirePackage{longtable}
\RequirePackage{multirow}
\RequirePackage{makecell}
\RequirePackage{bigints}

\usepackage[backref=page]{hyperref} %

\DeclareMathOperator{\Hom}{Hom}
\DeclareMathOperator{\End}{End}

\let\Im\relax

\DeclareMathOperator{\Im}{Im}

\DeclareMathOperator{\Lie}{Lie}

\newcommand{\heart}{\ensuremath\heartsuit}

\newcommand{\SO}{\mathrm{SO}}

\newcommand{\isoarrow}{%
   \ifbool{@display}{\overset{\sim}{\longrightarrow}}{\xrightarrow\sim}%
   }

\newcommand{\dia}{\diamond}

\newcommand{\dR}{\mathrm{dR}}
\newcommand{\Gr}{\mathrm{Gr}}

\DeclareFontFamily{U}{matha}{\hyphenchar\font45}
\DeclareFontShape{U}{matha}{m}{n}{
      <5> <6> <7> <8> <9> <10> gen * matha
      <10.95> matha10 <12> <14.4> <17.28> <20.74> <24.88> matha12
      }{}
\DeclareSymbolFont{matha}{U}{matha}{m}{n}
\DeclareFontFamily{U}{mathx}{\hyphenchar\font45}
\DeclareFontShape{U}{mathx}{m}{n}{
      <5> <6> <7> <8> <9> <10>
      <10.95> <12> <14.4> <17.28> <20.74> <24.88>
      mathx10
      }{}
\DeclareSymbolFont{mathx}{U}{mathx}{m}{n}

\DeclareMathSymbol{\obot}         {2}{matha}{"6B}

\newcommand{\bx}{{\bf x}}

\newcommand{\bG}{\mathbb{G}}
\newcommand{\bP}{\mathbb{P}}
\newcommand{\by}{\mathbf{y}}

\newcommand{\Fil}{\mathrm{Fil}}

\newcommand{\A}{\mathbb{A}}

\newcommand{\R}{\mathbb{R}}

\newcommand{\Q}{\mathbb{Q}}
\newcommand{\Z}{\mathbb{Z}}

\newcommand{\C}{\mathbb{C}}

\renewcommand{\H}{\mathbb{H}}
\newcommand{\F}{\mathbb{F}}

\newcommand{\bfA}{\mathbf{A}}
\newcommand{\bfG}{\mathbf{G}}

\newcommand{\cC}{\mathcal{C}}
\newcommand{\Oo}{\mathcal{O}}

\newcommand{\cZ}{\mathcal{Z}}
\newcommand{\cD}{\mathcal{D}}
\newcommand{\cS}{\mathcal{S}}
\newcommand{\cX}{\mathcal{X}}

\newcommand{\cL}{\mathcal{L}}

\newcommand{\cM}{\mathcal{M}}

\newcommand{\cF}{\mathcal{F}}
\newcommand{\cB}{\mathcal{B}}
\newcommand{\Exc}{\mathrm{Exc}}

\newcommand{\tot}{\mathrm{tot}}

\newcommand{\Spec}{\mathrm{Spec}\, }

\newcommand{\Sym}{\mathrm{Sym}}

\newcommand{\SL}{\operatorname{SL}}

\newcommand{\OO}{\mathcal O}

\newcommand{\m}{\mathrm{m}}

\newcommand{\cha}{\mathrm{Char}}
\newcommand{\ff}{\mathrm{if }}

\newcommand{\kay}{\boldsymbol{k}}

\newcommand{\CH}{\operatorname{CH}}

\newcommand{\kzxz}[4]{\left(\begin{smallmatrix} #1 & #2 \\ #3 & #4\end{smallmatrix}\right) }
\newcommand{\kabcd}{\kzxz{a}{b}{c}{d}}

\newtheorem{theorem}{Theorem}[section]
\newtheorem{corollary}[theorem]{Corollary}
\newtheorem{proposition}[theorem]{Proposition}
\newtheorem{lemma}[theorem]{Lemma}
\newtheorem{remark}[theorem]{Remark}

\theoremstyle{definition}
\newtheorem{definition}[theorem]{Definition}

\numberwithin{equation}{section}

\title{Pullback of arithmetic theta series and its modularity for unitary Shimura curves}
 
\author[Qiao He]{Qiao He}
\address{Department of Mathematics, Columbia University, MC 4406, 2990 Broadway
New York, NY 10027,  USA}
\email{qh2275@columbia.edu}

\author[Yousheng Shi]{Yousheng Shi}
\address{School of Mathematical Sciences, Zhejiang University, 866 Yuhangtang Rd,
Hangzhou, Zhejiang Province 310058, China}
\email{0023140@zju.edu.cn}

\author[Tonghai Yang]{Tonghai Yang}
\address{Department of Mathematics, University of Wisconsin Madison, Van Vleck Hall,
Madison, WI 53706, USA}
\email{thyang@math.wisc.edu}


\thanks{QH was partially supported by Columbia's Faculty Research Allocation Program, YS is partially supported by Zhejiang University's start-up Fund and Qizhen Fund, and   TY was partially supported by UW-Madison Kellet Mid-career award}

\begin{document}

\makeatletter
 \providecommand\@dotsep{5}
 \def\listtodoname{List of Todos}
 \def\listoftodos{\@starttoc{tdo}\listtodoname}
 \makeatother

\maketitle 

\begin{abstract}
		This paper is a complement of the modularity result of \cite{BHKRY1} for the special case $U(1,1)$ not considered there. The main idea is to embed a $U(1, 1)$ Shimura curve to many $U(n-1, 1)$ Shimura varieties for big $n$, and prove a precise pullback formula of the generating series of arithmetic divisors.  Afterwards, we use the modularity result of \cite{BHKRY1} together with the existence of non-vanishing  classical theta series at any given point in the upper half plane to prove the modulartiy result on $U(1, 1)$ Shimura curves. 
\end{abstract}

\tableofcontents

\section{Introduction}  \label{sect:Intro}

Modularity is a beautiful way to organize a sequence of objects (say numbers) with a lot of symmetry. A classical example is the theta function 
\begin{equation}
     \theta_m(\tau) =\sum_{n=0}^\infty r_m(n) q^n, \quad q=e^{ 2 \pi i \tau},
\end{equation}
where $r_m(n)$ is the number of ways to express $n$ as sum of $m$ integers. The modularity of $\theta_m$ implies in particular that if we know a few $r_m(n)$ for ``small" $n$, then we know all $r_m(n)$.  Modularity has played an important role in number theory and other fields, see for example \cite{HZ}, \cite{KM}, \cite{KuDuke}, \cite{zhang2021AFL}, \cite{BR}, and \cite{BHKRY1}.

In \cite{BHKRY1}, the authors proved the modularity of  a generating series
of special divisors on a compactified integral model of a Shimura variety
associated to a unitary group (Kr\"amer model for a unimodular lattice) of signature $(n-1, 1)$ for $n>2$, and left the special case $n=2$ unresolved for technical reasons. This paper finishes the job for $n=2$ by the so-called embedding trick.  We now describe the main results in a little more detail. For undefined terms, see Section \ref{sect:USV} for definition. 

Let $\kay=\Q(\sqrt{-D})$ be an imaginary quadratic field with ring of integers $\Oo_{\kay}$ and odd discriminant $-D$.
Let $\mathfrak a_0$ and $\mathfrak a$ be two unimodular Hermitian $\Oo_{\kay}$-lattices of signature $(1, 0)$ and $(n-1, 1)$ respectively.  Let $\mathcal S^*$ be the associated integral model of the Shimura variety with toroidal compactification (see Section \ref{sect:USV}).  For each $m > 0$,  Kudla and Rapoport (\cite{KR2}) constructed divisors $\mathcal Z(m)$ in the open Shimura variety $\mathcal S \subset \mathcal S^*$ over $\Oo_{\kay}$.  There are two Green functions for $Z(m) = \mathcal Z(m)(\C)$, $\Gr_B(m)$ constructed by J. Bruinier in his thesis (\cite{BrThesis}, \cite{BF04}), and $\Gr_K(m, v)$ constructed by S. Kudla with an extra parameter $v >0$ (\cite{Kudla97}). By studying their behavior at the boundary, one obtains two arithmetic divisors  (\cite{BHY}, \cite{Ho1}, \cite{ES},  see Sections \ref{sec:pullback I} and \ref{sect:PullBack II})
$$
\widehat{\mathcal Z}_B^\tot(m) = (\mathcal Z_B^{\tot}(m), \Gr_B(m)),  \hbox{ and } 
\widehat{\mathcal Z}_K^\tot(m,v) = (\mathcal Z_K^{\tot}(m,v), \Gr_K(m, v)) \in \widehat{\CH}_\C^1(\mathcal S^*),
$$
with $\mathcal Z_i^{\tot}(m) = \mathcal Z^*(m) + \hbox{ boundary components}$ for $i=B, K$.  Here   $\mathcal Z^*(m)$ is the Zariski closure of $\mathcal Z(m)$ in $\mathcal S^*$, and $\widehat{\CH}_\C^1(\mathcal S^*)$ is the arithmetic Chow group with $\C$-coefficients and  log-log singularities near the boundary (see \cite{BKK} and \cite{Ho1}). Actually, $\Gr_K(m, v)$ is defined for all integers $m$   and induces  an arithmetic divisor  (for all integers $m$)
$$\widehat{\mathcal Z}_K^\tot(m,v) = (\mathcal Z_K^{\tot}(m), \Gr_K(m, v)) \in \widehat{\CH}_\C^1(\mathcal S^*).$$
One also defines arithmetic divisors  $\widehat{\mathcal Z}_B^{\tot}(0)$ and $\widehat{\mathcal Z}^{\tot}_K(0, v)$  in $\widehat{\CH}_\C^1(\mathcal S^*)$. As in \cite{BHKRY1},  We define two formal $q$-expansions with values in $\widehat{\CH}_\C^1(\mathcal S^*)$:

\begin{align}
\widehat{\Theta}_B(\tau) &= \sum_{m \ge 0}   \widehat{\mathcal Z}_B^{\tot}(m) q^m ,  \hbox{ and }
\\
\widehat{\Theta}_K(\tau) &= \sum_{m \in \Z}   \widehat{\mathcal Z}_K^{\tot}(m, v) q^m
.
\end{align}
They are usually called  Bruinier and Kudla arithmetic  theta functions  respectively. 
Ehlen and Sankran proved in  \cite[Theorem 1.4]{ES} that $\widehat{\Theta}_B(\tau) -\widehat{\Theta}_K(\tau)$ is a modular form for $\Gamma_0(D)$ of weight  $n$, character $\chi_{-D}^n$ (where $\chi_{-D}(\cdot)$ is the Kronecker symbol $\left(\frac{-D}{\cdot}\right)$) and with values in $\widehat{\CH}_\C^1(\mathcal S^*)$ (see Section \ref{sect:Modularity} for detail and some subtlety). In \cite{BHKRY1}, the authors proved the modularity of $\widehat{\Theta}_B(\tau)$ for $n \ge 3$, and thus the modularity of $\widehat{\Theta}_K(\tau)$ for $n \ge 3$.  In this paper, we prove the modularity of $\widehat{\Theta}_K(\tau)$ for $n=2$, which in turn  implies the modularity of $\widehat{\Theta}_B(\tau)$. We record it as 

\begin{theorem} \label{theo:Modularity}  When $n =2$,  $\widehat{\Theta}_B(\tau)$ and $\widehat{\Theta}_K(\tau)$ are modular forms for $\Gamma_0(D)$ of weight $2$, trivial character, and with values in 
$\widehat{\CH}_\C^1(\mathcal S^*)$.
\end{theorem}

\begin{remark}
    As $\widehat{\CH}_\C^1(\mathcal S^*)$ is infinitely dimensional, the modularity needs some explanation.  Since $\widehat{\mathcal Z}_B^{\tot}(m)$ is independent of $\tau$,  $\widehat{\Theta}_B(\tau)$ being modular means one of the two  equivalent conditions:

(1)  For every linear functional  $f$  on $\widehat{\CH}_\C^1(\mathcal S^*)$, $f(\widehat{\Theta}_B(\tau)) =  \sum f(\widehat{\mathcal Z}_B^{\tot}(m)) q^m   $ is  modular.

(2)   There are finitely many  arithmetic divisors $\widehat{Z}_i$ such that 
$$
\widehat{\Theta}_B(\tau)= \sum_i f_i(\tau) \widehat{\mathcal Z}_i 
$$
and $f_i(\tau)$ are usual scalar holomorphic  modular forms.

However, since $\widehat{\mathcal Z}_K^{\tot}(m, v)$  depends  on $v =\hbox{Im}(\tau)$,  the  above definition  does not work and the modularity is more subtle. We refer to \cite[Definition 4.1]{ES} for detail. Roughly speaking, it means that  we can write 
$$
\widehat{\Theta}_K(\tau)= \sum_i f_i(\tau) \widehat{\mathcal Z}_i + (0, f(\tau, z) )
$$
with $f_i$ and $\widehat{\mathcal Z}_i$  as above, and that $f(\tau, z)$ is a smooth scalar modular form for every $z\in S$ plus some extra technical conditions. 
\end{remark}

We remark that the modularity of $\widehat{\Theta}_B$ could be used to prove a complement of the original Gross-Zagier formula (the case when every $p|N$ is ramified in $\kay$, in comparison to that every $p|N$ is split in $\kay$).  To prove the modularity of $\widehat{\Theta}_K(\tau)$ for $n=2$, we use the following embedding trick (see, for example, \cite{li2021algebraicity}, \cite{HP}).  Let $\Lambda$ be a positive definite unimodular $\Oo_{\kay}$-lattice of rank $m$. Replacing $\mathfrak a$ by $\mathfrak a \obot \Lambda$, we obtain a Shimura variety 
$\mathcal S^{\dia, *}$   together with a canonical morphism:
\begin{equation}
\varphi_\Lambda: \mathcal S^{*} \rightarrow \mathcal S^{\dia, *}, (A_0, A, ...) \mapsto (A_0, A \times (A_0 \otimes_{\Oo_{\kay}} \Lambda), ...).
\end{equation}
It induces a homomorphism $\varphi_\Lambda^*:  \widehat{\CH}^1(\mathcal S^{\dia, *}) \rightarrow \widehat{\CH}^1(\mathcal S^{*})$. The following is  the main technical theorem of this paper, which should be of independent interest. 

\begin{theorem}\label{theo:Decomposition}
Let notations be as above. Then 
$$
\varphi_\Lambda^* (\widehat{\Theta}_K^\dia(\tau) )
 =\theta_\Lambda (\tau) \widehat{\Theta}_K(\tau).
$$
Here $\widehat{\Theta}_K^\dia(\tau) $ is the arithmetic theta function (with Kudla Green functions) associated to $\mathcal S^{\dia, *}$, and 
\begin{equation}
    \theta_\Lambda(\tau) = \sum_{x \in \Lambda} q^{(x,x)} =\sum_{n \ge 0}r_\Lambda(n) q^n
\end{equation}
is a classical holomorphic modular form for $\Gamma_0(D)$ of  weight $m$ and character $\chi_{-D}^m$ .
\end{theorem}

By an analogue of a result of Yingkun Li \cite[Lemma 3.2]{li2021algebraicity},  for every $\tau$, there is some $\Lambda$ with $\theta_\Lambda(\tau) \ne 0$. So $\widehat{\Theta}_K(\tau)$ is everywhere defined and the modularity of  $\widehat{\Theta}_K(\tau)$ follows from the modularity result of \cite{BHKRY1}. Now Theorem \ref{theo:Modularity}  for $\widehat{\Theta}_B(\tau)$ follows from  the modularity for $\widehat{\Theta}_K(\tau)$ and the main result of \cite{ES}.  Note that the analogue of Theorem \ref{theo:Decomposition} does not hold for $\widehat{\Theta}_B$.

The main work to prove Theorem \ref{theo:Decomposition} is to understand the pullback of special divisors on $\cS^{\dia,*}$ that intersect with $\cS^*$ improperly and various subtleties about the boundary components. On the generic fiber, this pullback is controlled by the line bundle  of modular form of weight $1$  (descended from  the tautological line bundle), which can be regarded as an analogue of the adjunction formula. The line bundle of modular form of weight $1$  has two well-defined integral models: $\omega$ considered in \cite{BHKRY1} and $\Omega$ considered in \cite{Ho2}. Although $\omega$ is used in the definition of $\widehat{\mathcal Z}_K^{\tot}(0, v) $ and $\widehat{\mathcal Z}_B^{\tot}(0)$ in \cite{BHKRY1}, it turns out that $\Omega$ is the one that controls the pullback of special divisors that intersect with $\cS^*$ improperly. So it is important to figure out their precise relation. We show in Section \ref{sec:two line bundles} (Theorem \ref{thm:comparision of two line bundles}):
\begin{equation} \label{eq:Relation}
    \Omega= \omega \otimes (\Oo_{\Exc})^{-1}.
\end{equation}
We believe that this relation is of independent interest too. It is discovered by \cite{Ho2} that the line bundle $\Omega$ controls the deformation theory of special divisors in the open Shimura variety $\cS$. Another main technical point of this paper is to extend this observation to the toroidal compactification $\cS^*$, see Proposition \ref{prop:deformation cS star}. In order to do this, we extend the definition of $\cZ(m)$ as a stack by \cite{KR2} to the boundary, see Definition \ref{def: cZ*}.

This paper is organized as follows. In Section \ref{sect:USV}, we review basics on integral model, toroidal compactification and special divisors which are needed in this paper. In Section \ref{sec:two line bundles}, we prove the precise relation (\ref{eq:Relation}) between the  two well-defined integral models of the line bundle of modular forms of weight $1$ (Theorem  \ref{thm:comparision of two line bundles}). In Section \ref{sec:deformation of special divisors}, we extend the controlling deformation property  of $\Omega$ of Ben Howard to boundaries (Proposition \ref{prop:deformation cS star}).  Section \ref{sec:pullback I} is dedicated to proving Theorem \ref{thm:pullback total}: Theorem \ref{theo:Decomposition} without Green functions.   In Section \ref{sect:PullBack II}, we deal with  pullbacks of Green functions and metrics on line bundles, and finish the proof of Theorem \ref{theo:Decomposition}.   In the first part of Section \ref{sect:Modularity}, we use Li's embedding trick (\cite{li2021algebraicity})  to prove the modularity of $\widehat{\Theta}_K(\tau)$. In the second part of Section \ref{sect:Modularity}, we describe two slightly different Green functions for special divisors, one defined by Bruinier in \cite{BHKRY1} and the other one defined by Ehlen and Sankaran in \cite{ES}. We show that $\widehat{\Theta}_B(\tau) $ and $\widehat{\Theta}_{ES}(\tau) $ are essentially the same (\ref{eq:Difference}), which will be used in the modularity of difference $\widehat{\Theta}_K(\tau) - \widehat{\Theta}_{B}(\tau)$.   We remark that Bruinier's Green functions are a little easier to define than those of Ehlen and Sankaran.

{\bf Acknowledgment}: We thank Jan Bruinier, Ben Howard, Yingkun Li,  Keerthi Madapusi Pera, and  Sid Sankaran their helpful discussions during the preparation of this paper. We thank the anonymous referee for carefully reading the paper and for their valuable  suggestion and comments, which makes the paper better written.

\subsection{Notations}
Let $\kay=\Q(\sqrt{-D})$ be an imaginary quadratic field with ring of integers $\Oo_{\kay}$ and odd fundamental discriminant $-D$. Let $\chi_{-D}(\cdot)$ be the Kronecker symbol $\left(\frac{-D}{\cdot}\right)$.
Fix a $\pi\in \Oo_K$ such that $\Oo_K=\Z+\pi \Z$. We also set $\delta=\sqrt{-D}$. For any $\Oo_{\kay}$-scheme $S$, define
\begin{align}\label{eq:epsilon}
    \epsilon_S=& \pi \otimes 1- 1\otimes \varphi(\bar\pi) \in \Oo_{\kay}\otimes_\Z \Oo_S,\\
    \bar\epsilon_S=& \bar\pi \otimes 1- 1\otimes \varphi(\bar\pi) \in \Oo_{\kay}\otimes_\Z \Oo_S,
\end{align}
where $\varphi:\Oo_{\kay}\rightarrow\Oo_S$ is the structure map. We denote the Galois conjugate of $\varphi$ by $\bar\varphi$.
The ideal sheaf generated by these sections are independent of the choice of $\pi$. The $\Oo_{\kay}$-scheme $S$ will be usually clear from the context and we often abbreviate $\epsilon_S$ and $\bar\epsilon_S$ to $\epsilon$ and $\bar\epsilon$ respectively.

\section{Unitary Shimura varieties and special divisors}  \label{sect:USV}
\subsection{Unitary Shimura varieties}\label{sec: unitary shimura variety} 
In this section, we review the theory of unitary Shimura variety following \cite{BHKRY1}.  Let $W_0$ and $W$ be $\kay$-vector spaces endowed with hermitian forms $H_0$ and $H$ of signatures $(1,0)$ and $(n-1,1)$, respectively. We always assume that $n \geqslant 2$. Abbreviate
$$
W(\mathbb{R})=W \otimes_{\mathbb{Q}} \mathbb{R}, \quad W(\mathbb{C})=W \otimes_{\mathbb{Q}} \mathbb{C}, \quad W\left(\mathbb{A}_f\right)=W \otimes_{\mathbb{Q}} \mathbb{A}_f,
$$
and similarly for $W_0$. In particular, $W_0(\mathbb{R})$ and $W(\mathbb{R})$ are hermitian spaces over $\mathbb{C}=\boldsymbol{k} \otimes_{\mathbb{Q}} \mathbb{R}$.

We assume the existence of $\Oo_{\kay}$-lattices $\mathfrak{a}_0 \subset W_0$ and $\mathfrak{a} \subset W$, self-dual with respect to the hermitian forms $H_0$ and $H$. As  $\delta=\sqrt{-D} \in \boldsymbol{k}$ generates the  different of $\boldsymbol{k} / \mathbb{Q}$, this is equivalent to self-duality with respect to the symplectic forms
\begin{equation}\label{eq: symp form}
 \quad \psi_0\left(w, w^{\prime}\right)=\operatorname{Tr}_{\kay / \mathbb{Q}} H_0\left(\delta^{-1} w, w^{\prime}\right), \quad \psi\left(w, w^{\prime}\right)=\operatorname{Tr}_{\boldsymbol{k} / \mathbb{Q}} H\left(\delta^{-1} w, w^{\prime}\right).
 \end{equation}
Let $G \subset \mathrm{GU}\left(W_0\right) \times \mathrm{GU}(W)$ be the subgroup of pairs for which the similitude factors are equal. We denote by $\nu: G \rightarrow \mathbb{G}_m$ the common similitude character, and note that $\nu(G(\mathbb{R})) \subset \mathbb{R}^{>0}$.
Let $\mathcal{D}\left(W_0\right)=\left\{y_0\right\}$ be a one-point set, and
\begin{equation}
  \mathcal{D}(W)=\{\text{negative definite }  \C\text{-lines }  y \subset W(\mathbb{R})\},  
\end{equation}
so that $G(\mathbb{R})$ acts on the connected hermitian domain
$$
\mathcal{D}=\mathcal{D}\left(W_0\right) \times \mathcal{D}(W) .
$$

The symplectic forms \eqref{eq: symp form} determine a $\boldsymbol{k}$-conjugate-linear isomorphism
\begin{equation}
\operatorname{Hom}_k\left(W_0, W\right) \stackrel{\sim}{\longrightarrow} \operatorname{Hom}_k\left(W, W_0\right),\   x \mapsto x^{\vee},
\end{equation}
characterized by $\psi\left(x w_0, w\right)=\psi_0\left(w_0, x^{\vee} w\right)$. The $\boldsymbol{k}$-vector space
$$
V=\operatorname{Hom}_{\boldsymbol{k}}\left(W_0, W\right)
$$
carries a hermitian form of signature $(n-1,1)$ defined by
\begin{equation}\label{eq: herm form rhs}
    ( x_1, x_2)=x_2^{\vee} \circ x_1 \in \operatorname{End}_{\boldsymbol{k}}\left(W_0\right) \cong \boldsymbol{k} .
\end{equation}
Let 
\begin{equation}\label{eq:L}
   L\coloneqq \Hom_{\Oo_k}(\mathfrak{a}_0,\mathfrak{a}),
\end{equation}
which is a unimodular hermitian $\Oo_{\kay}$-lattice of signature $(n-1,1)$.
The group $G$ acts on $V$ in a natural way, defining an exact sequence
\begin{equation}
    1\rightarrow \mathrm{Res}_{\kay/\Q}\mathbb{G}_m\rightarrow G\rightarrow \mathrm{U}(V)\rightarrow 1.
\end{equation}

The lattices $\mathfrak{a}_0$ and $\mathfrak{a}$ determine a compact open subgroup
\begin{align}\label{eq: K}
K=\left\{g \in G\left(\mathbb{A}_f\right): g \widehat{\mathfrak{a}}_0=\widehat{\mathfrak{a}}_0 \text{ and } g \widehat{\mathfrak{a}}=\widehat{\mathfrak{a}}\right\} \subset G\left(\mathbb{A}_f\right),
\end{align}
and the orbifold quotient
$$
\operatorname{Sh}(G, \mathcal{D})(\mathbb{C})=G(\mathbb{Q}) \backslash \mathcal{D} \times G\left(\mathbb{A}_f\right) / K
$$
is the set of complex points of a smooth $\boldsymbol{k}$-stack of dimension $n-1$, denoted by $\operatorname{Sh}(G, \mathcal{D})$.

\subsection{Integral model}\label{subsec:integral model}
We recall the integral model defined by \cite{BHKRY1} originating in the work of Kr\"amer \cite{Kr}. For $(r,s)=(n-1,1)$ or $(n,0)$, let
$$
\mathcal{M}_{(r,s)} \rightarrow \operatorname{Spec}(\Oo_{\kay} ) 
$$
 be the stack so that for an $\Oo_{\kay}$-scheme $S$,  $\mathcal{M}_{(r,s)}(S)$ is the
groupoid of quadruples $\left(A, \iota, \lambda, \mathcal{F}_A\right)$  where
\begin{enumerate}
    \item  $A \rightarrow S$ is an abelian scheme of relative dimension $n$,
    \item $\iota: \Oo_{\kay} \rightarrow \operatorname{End}(A)$ is an action of $\Oo_{\kay}$,
    \item $\lambda: A \rightarrow A^{\vee}$ is a principal polarization whose induced Rosati involution $\dagger$ on $\End^0(A)$ satisfies $\iota(\alpha)^{\dagger}=\iota(\bar{\alpha})$ for all $\alpha \in \Oo_{\kay}$,
    \item\label{item:Kramer} $\mathcal{F}_A \subset \operatorname{Lie}(A)$ is an $\Oo_{\kay}$-stable $\Oo_S$-module local direct summand of rank $r$ satisfying Krämer's condition: $\Oo_{\kay}$ acts on $\mathcal{F}_A$ via the structure map $\Oo_{\kay} \rightarrow  \Oo_S$, and acts on quotient $\operatorname{Lie}(A) / \mathcal{F}_A$ via the complex conjugate of the structure map.
\end{enumerate}
In particular, when $(r,s)=(n,0)$, the $\cF_A$ in the condition $(4)$ above is simply $\Lie(A)$. In this case, $\cM(n,0)\rightarrow \Spec \Oo_{\kay}$ is proper and smooth of relative dimension $0$ by \cite[Proposition 2.1.2]{Ho1}. For $\mathcal{M}_{(n-1,1)}$, we have the following theorem.
\begin{theorem}\cite[Theorem 2.3.3]{BHKRY1} 
The $\mathcal{O}_{\boldsymbol{k}}$-stack $\mathcal{M}_{(n-1,1)}$ is regular and flat with reduced fibers. 
\end{theorem}

Finally, for the genus class $[[L]]$ of $L$, we define 
\[\cS\subset \cM_{(1,0)}\times \cM_{(n-1,1)} \]
to be the open and closed substack such that $\cS(S)$ is the groupoid of tuples
$$\left(A_0,\iota_0,\lambda_0, A,\iota,\lambda\right) \in \mathcal{M}_{(1,0)}(S) \times \mathcal{M}_{(n-1,1)}(S)$$
such that 
 at every geometric point $s \rightarrow S$, there exists  an isomorphism of hermitian $\mathcal{O}_{\boldsymbol{k}, \ell}$-modules
\begin{align}\label{eq: iso tate and L}
\operatorname{Hom}_{\Oo_{\kay}}\left(T_{\ell} A_{0, s}, T_{\ell} A_s\right) \cong \operatorname{Hom}_{\Oo_{\kay}}\left(\mathfrak{a}_0, \mathfrak{a}\right) \otimes \mathbb{Z}_{\ell}=L_\ell,
\end{align}
for every finite prime $\ell \neq \mathrm{char}(k(s))$. Here the hermitian form on the right hand side is defined as in \eqref{eq: herm form rhs}. For the hermitian form on the left hand side, we define it similarly by replacing the symplectic forms \eqref{eq: symp form} on $W_0$ and $W$ with the Weil pairings on the Tate modules $T_{\ell} A_{0, s}$ and $T_{\ell} A_s$ ($A_s$ being the pullback of the universal object $\mathbf{A}$ over $\mathcal S$ to $s$) induced by their polarizations. The generic fiber $\mathcal{S}\times_{\Spec \Oo_{\kay}} \Spec\kay$ is the Shimura variety $\operatorname{Sh}(G, \mathcal{D})$.

Now assume that $p\in \Z$ is a prime ramified in $\kay$ (i.e. dividing $D$). Following \cite[Appendix A]{Ho2} we define the exceptional divisor $\Exc_p$ to be the locus of $\mathcal{S}\times_{\Spec\Oo_{\kay,p}} \Spec\F_p$ where a geometric point $s\in \mathcal{S}(\bar\F_p)$ is in $\Exc_p$ if the action $\Oo_{\kay}\rightarrow \Lie A_s$ factor through the reduction homomorphism $\Oo_{\kay}\rightarrow \F_p$. Then $\Exc_p$ can be given the structure of a reduced substack of $\mathcal{S}$ and is in fact a Cartier divisor consisting of disjoint unions of $\bP^{n-1}$ over $\F_p$ (\cite[Theorem 2.3.4]{BHKRY1}, \cite[Proposition A.2]{Ho2}). Finally, we define 
\[\Exc=\bigsqcup_{p|D} \Exc_p. \]

\subsection{Special divisors}
For a connected $\mathcal{O}_{k}$-scheme $S$ and 
$$
\left(A_0,\iota_0,\lambda_0, A,\iota,\lambda\right) \in \mathcal{S}(S),
$$
we can define a positive definite hermitian form on $\operatorname{Hom}_{\Oo_{\kay}}\left(A_0, A\right)$  by
\begin{equation}\label{eq:hermitian form on Hom}
 \left( x_1, x_2\right)=\iota_0^{-1}(\lambda_0^{-1} \circ x_2^\vee \circ \lambda \circ x_1) \in \Oo_{\kay}
\end{equation}
where $x_2^{\vee}: A^\vee\rightarrow A_0^\vee$ is the dual homomorphism of $x_2$. By \cite[Lemma 2.7]{KR2}, the form $\left( ,\right)$ is positive-definite.

\begin{definition}(\cite[Definition 2.8]{KR2})
For any  $m \in \mathbb{Z}_{>0}$, define  $\mathcal{Z}(m)$  to be the moduli stack assigning to a connected $\Oo_{\kay}$-scheme $S$ the groupoid of tuples $\left(A_0,\iota_0,\lambda_0, A,\iota,\lambda,x\right)$, where
\begin{enumerate}
    \item $\left(A_0,\iota_0,\lambda_0, A,\iota,\lambda\right) \in \mathcal{S}(S)$,
    \item $x \in \operatorname{Hom}_{\Oo_{\kay}}\left(A_0, A\right)$ satisfies $( x, x)=m$.
\end{enumerate}    
\end{definition}
According to the discussion in \cite[\S 2.5]{BHKRY1}, we may regard $\cZ(m)$ as a Cartier divisor on $\cS$.  

In the rest of the section, we recall the construction of the toroidal compactification of $\mathcal S$ following \cite[\S 3]{BHKRY1} and \cite[\S 2]{Ho1}.

\subsection{Cusp labels}\label{subsec:cusp label}
A proper cusp label is an isomorphism class of pairs $\Phi=(\mathfrak{n}, L_{\Phi})$ in which $\mathfrak{n}$ is a projective $\mathcal{O}_{\kay}$-module of rank one, and $L_{\Phi}$ is a unimodular Hermitian lattice of signature $(n-2,0)$. Consider a pair $\mathfrak{m} \subset M$  where  $M$ is a  unimodular hermitian lattice of signature $(n-1, 1)$ and $\mathfrak{m}$ is an isotropic direct summand of rank one. A normal decomposition of $\mathfrak{m} \subset M$ is an $\mathcal{O}_{\boldsymbol{k}}$-module direct sum decomposition
\begin{align}\label{eq: normal decom}
 M=(\mathfrak{m}\oplus\mathfrak{n} ) \obot L_{\Phi}
\end{align}
where $L_{\Phi}=(\mathfrak{m} \oplus \mathfrak{n})^{\perp}$ and $\mathfrak{n}$ is an isotropic direct summand of rank one and can be identified as $\Hom_{\Oo_{\kay}}(\mathfrak m, \Oo_{\kay})$. The hermitian form on  $M$  makes $L_{\Phi} \cong$ $\mathfrak{m}^{\perp} / \mathfrak{m}$ into a unimodular hermitian lattice of signature $(n-2,0)$.
\begin{lemma}\label{lem: eq cusp}\cite[Proposition 2.6.3]{Ho1}
Every pair $\mathfrak{m} \subset M$ as above admits a normal decomposition. The rule
$$
\mathfrak{m} \subset M \mapsto\left( M / \mathfrak{m}^{\perp}, \mathfrak{m}^{\perp} / \mathfrak{m}\right) \cong(\mathfrak{n}, L_{\Phi})
$$
establishes a bijection between the isomorphism classes of pairs $\mathfrak{m} \subset M$ as above, and the set of cusp labels.
\end{lemma}
Recall that $L$ is the unimodular lattice defined in \ref{eq:L}.
\begin{definition}\label{def: cusp label}
We define $\mathrm{Cusp}(M)$ to be the set of proper cusp labels  $\Phi=(\mathfrak{n},L_{\Phi})$ such that there exists a  $\mathfrak{m}$  with $M=\mathfrak{m}\oplus\mathfrak{n}  \obot L_{\Phi}$.  Moreover, we denote $\mathrm{Cusp}([[L]])=\coprod_{M\in [[L]]}\mathrm{Cusp}(M)$. 
\end{definition}

Because of Lemma \ref{lem: eq cusp}, we also denote a proper cusp label as $\Phi=(\mathfrak m\subset M) $.
The definition of proper cusp labels in \cite{BHKRY1} is different with Definition \ref{def: cusp label}. However, 
according to \cite[Lemma 3.1.4]{BHKRY1} and Lemma \ref{lem: eq cusp}, there is a natural bijection between the set of equivalence classes of proper cusp labels defined in \cite{BHKRY1} and the one defined in Definition \ref{def: cusp label}. 
Indeed,  assume $(P,g)$ is as in  \cite[Definition 3.1.1]{BHKRY1}, then $gL=M$ is a self-dual $\OO_{\kay}$-lattice and $P$ determines an isotropic line $J\subset W$. Then $\mathrm{Hom}_{\kay}(W_0,J)\cap L$ determines an isotropic dirrect summand of rank one. Then \cite[Lemma 3.1.4]{BHKRY1} shows that this map is an injection. Now given $(\mathfrak{m},M)$, we can choose $g$ such that $gL=M$. Then we can choose $J\subset W$ such that $\mathrm{Hom}_{\kay}(W_0,J)\cap L=\mathfrak{m}$ and set $g=\mathrm{Stab}_G(J)$.

\subsection{Degenerating abelian schemes}\label{subsec:degenerating abelian schemes}
We review the theory of degenerating abelian scheme following \cite[\S 2.3]{Ho1} and \cite[\S 5.1]{Lan}.
For a projective $\Oo_{\kay}$-module $\mathfrak{p}$ of rank $1$, let $\underline{\mathfrak{p}}$ be the associated constant $\Oo_{\kay}$-module scheme over $\Spec \Z$. Let $X$ be an $\Oo_{\kay}$-stack, $Z\rightarrow X$ be a closed substack, and $U\subset X\setminus Z$ be a dense open substack. We remark that we can also take $U$ to be the generic point of $X$ when $X$ is an irreducible scheme and the following discussion will be the same. Let $\Phi=(\mathfrak{n},L_\Phi)$ or $(\mathfrak{m}\subset M)$ be a cusp label as in \S \ref{subsec:cusp label}.
Let $\Lambda_\Phi=\Hom_{\Oo_{\kay}}(L_\Phi,\Oo_{\kay})$ be the hermitian dual of $L_\Phi$.

\begin{definition}
    A semiabelian scheme over $X$ is a smooth commutative group scheme $G\rightarrow X$, such that for every geometric point $z\rightarrow X$ the fiber $G_z$ is an extension 
    \[0\rightarrow T\rightarrow G_z\rightarrow B\rightarrow 0\]
    of an abelian variety by a torus.
\end{definition}

\begin{definition}\label{def:degenerating abelian scheme}
    A degenerating abelian scheme of type $\Phi$ relative to $(X,Z,U)$ is a triple $(G,\iota,\lambda)$ such that 
    \begin{itemize}
        \item $G$ is a semi-abelian scheme over $X$ such that $G_U$ is an abelian scheme;
        \item $\iota:\Oo_{\kay}\rightarrow \End(G_U)$ is an action of $\Oo_{\kay}$ on $G_U$;
        \item $\lambda: G_U \rightarrow G_U^{\vee}$ is a principal polarization whose induced Rosati involution $\dagger$ on $\End^0(G_U)$ satisfies $\iota(\alpha)^{\dagger}=\iota(\bar{\alpha})$ for all $\alpha \in \Oo_{\kay}$;
        \item there is an abelian scheme $B_Z$ over $Z$ equipped with an $\Oo_{\kay}$-action, and an $\Oo_{\kay}$-linear exact sequence
        \[0\rightarrow \mathfrak{m}\otimes_\Z \mathbb{G}_m\rightarrow G_Z\rightarrow B_Z\rightarrow 0. \]
    \end{itemize}
    If in addition 
    \begin{itemize}
        \item there is an $\Oo_{\kay}$-stable $\Oo_U$-module local direct summand $\mathcal{F} \subset \operatorname{Lie}(G_U)$ of rank $n-1$ satisfying Krämer's condition as in Condition \eqref{item:Kramer} in the definition of $\cM_{(r,s)}$;
        \item $(A_0,\iota_0,\lambda_0)\in \cM_{(1,0)}(X)$;
        \item there is an isomorphism of étale sheaves of hermitian $\mathcal{O}_{\boldsymbol{k}}$-modules over $Z$:
        \[\underline{\Lambda_\Phi}\cong \underline{\operatorname{Hom}}_{\Oo_{\kay}}\left(B_Z, A_0\right);\]
    \end{itemize}
    then we say $(A_0,\iota_0,\lambda_0,G,\iota,\lambda,\cF)$ is a degenerating abelian scheme of type $\Phi$ and signature $(n-1,1)$ relative to $(X,Z,U)$. We denote the category of degenerating abelian scheme of type $\Phi$ (and signature $(n-1,1)$ resp.) relative to $(X,Z,U)$ as $\mathrm{DEG}^\Phi(X,Z,U)$  ($\mathrm{DEG}^\Phi_{(n-1,1)}(X,Z,U)$ resp.), with isomorphisms in the obvious sense being morphisms.
\end{definition}

\begin{definition}\label{def:degeneration data}
    Degeneration data of type $\Phi$ relative to $(X,Z,U)$ consist of  tuples $(B, \kappa, \psi, c, c^\vee, \tau)$
    such that
    \begin{itemize}
        \item $B\rightarrow X$ is an abelian scheme;
        \item $\kappa:\Oo_{\kay}\rightarrow \End(B)$ is an action of $\Oo_{\kay}$ on $B$;
        \item $\psi: B \rightarrow B^{\vee}$ is a principal polarization whose induced Rosati involution $\dagger$ on $\End^0(B)$ satisfies $\iota(\alpha)^{\dagger}=\iota(\bar{\alpha})$ for all $\alpha \in \Oo_{\kay}$;
        \item $c:\underline{\mathfrak{n}}_{/X}\rightarrow B^\vee$ and $c^\vee:\underline{\mathfrak{n}}_{/X}\rightarrow B$ are $\Oo_{\kay}$-module maps satisfying $c=\psi\circ c^\vee$;
        \item $\tau$ is a  positive, symmetric, and $\Oo_{\kay}$-linear isomorphism
        \begin{equation}\label{eq:tau}
    \tau: 1_{(\underline{\mathfrak{n}}\times \underline{\mathfrak{n}})|_{U}}\rightarrow (c^\vee\times c)^* (\mathcal{P}^{-1})|_{(\underline{\mathfrak{n}}\times \underline{\mathfrak{n}})|_{U}} 
\end{equation}
        of $\bG_m$-biextensions of $(\underline{\mathfrak{n}}\times \underline{\mathfrak{n}})|_{U}$ (see below). Here $\mathcal{P}$ is the Poincar\'e sheaf on $B\times B^\vee$.
    \end{itemize}
    If in addition 
    \begin{itemize}
        \item  $(A_0,\iota_0,\lambda_0)\in \cM_{(1,0)}(X)$;
        \item  $B\in \cM_{(n-2,0)}(X)$ and there is an isomorphism of étale sheaves of hermitian $\mathcal{O}_{\boldsymbol{k}}$-modules over $X$:
        \[\underline{\Lambda_\Phi}\cong \underline{\operatorname{Hom}}_{\Oo_{\kay}}\left(B, A_0\right);\]
    \end{itemize}
    then we say $(A_0,\iota_0,\lambda_0,B, \kappa, \psi, c, c^\vee, \tau)$ is degeneration data of type $\Phi$ and signature $(n-1,1)$ relative to $(X,Z,U)$.   We denote the category of degeneration data of type $\Phi$ (and signature $(n-1,1)$ resp.) relative to $(X,Z,U)$ as $\mathrm{DD}^\Phi(X,Z,U)$  ($\mathrm{DD}^\Phi_{(n-1,1)}(X,Z,U)$ resp.), with isomorphisms in the obvious sense being morphisms.
\end{definition}
We explain the meaning of $\tau$ in more detail. To give a $\mathbb{G}_m$-biextension on $\underline{\mathfrak{n}}\times \underline{\mathfrak{n}}|_X$ is equivalent to giving a collection of invertible sheaves $\mathcal{E}(\mu,\nu)_{(\mu,\nu)\in \mathfrak{n}\times \mathfrak{n}}$  on $X$, together with isomorphisms
\[\mathcal{E}(\mu_1+\mu_2,\nu)\cong \mathcal{E}(\mu_1,\nu)\otimes \mathcal{E}(\mu_2,\nu) \]
and 
\[\mathcal{E}(\mu,\nu_1+\nu_2)\cong \mathcal{E}(\mu,\nu_1)\otimes \mathcal{E}(\mu,\nu_2) \]
satisfying certain partial group axioms. Denote by $\cL(\mu,\nu)$ the pullback of the Poincar\'e line bundle under the morphism $c^\vee(\mu)\times c(\nu):X\rightarrow B\times B^\vee$.
It follows from the standard bilinear properties of Poincar\'e bundles that $\cL(\mu,\nu)|_{(\mu,\nu)\in \mathfrak{n}\times \mathfrak{n}}$ determines a $\mathbb{G}_m$-biextension of $\underline{\mathfrak{n}}\times \underline{\mathfrak{n}}$ over $X$. Moreover, the $\Oo_{\kay}$-linearality of the polarization of $B$ guarantees that  $\cL(\mu,\nu)$, up to canonical isomorphism, only depends on the image of $\mu,\nu$ in 
\[\Sym_\Phi=\Sym^2_\Z(\mathfrak{n})/\langle (x\mu)\otimes \nu-\mu\otimes (\bar x \nu):x\in \Oo_{\kay}, \mu,\nu \in \mathfrak{n} \rangle. \]
Thus for each $\chi\in \Sym_\Phi$, we may associate a line bundle $\cL(\chi)$ on $\cB_\Phi$, such that there are canonical isomorphisms 
\[\cL(\chi)\otimes \cL(\chi')\cong \cL(\chi+\chi').\]
Our assumption that $D$ is odd implies that $\Sym_\Phi$ is a free $\Z$-module of rank $1$. There is a positive cone in $\Sym_\Phi\otimes_\Z \R$ uniquely determined by the condition $\mu\otimes\mu \geq 0$ for all $\mu\in \mathfrak n$. Thus all the line bundles $\cL(\chi)$ are powers of the distinguished line bundle
\begin{equation}\label{eq:cL_Phi}
    \cL_\Phi=\cL(\chi_0),
\end{equation}
determined by the unique positive generator $\chi_0\in \Sym_\Phi$.
Let $ 1_{(\underline{\mathfrak{n}}\times \underline{\mathfrak{n}})|_{U}}$ be the constant collection of invertible sheaves $\mathcal{O}_{\cC_\Phi}$. {\it The  positivity condition} means   that for every $\mu\in \mathfrak{n}$, the isomorphism $\tau(\mu,\mu)$ extends (ncecessarily uniquely) to a homomorphism
\[\tau(\mu,\mu):\Oo_{X}\rightarrow (c(\mu)^\vee\times c(\mu))^* (\mathcal{P}^{-1})\]
and if $\mu\neq 0$, the homomorphism becomes trivial after restricting to $Z$.

There is a functor $M^\Phi(X,Z,U):\mathrm{DD}^\Phi(X,Z,U)\rightarrow \mathrm{DEG}^\Phi(X,Z,U)$. We briefly recall its construction. Suppose $(B, \kappa, \psi, c, c^\vee, \tau)\in \mathrm{DD}^\Phi(X,Z,U)$. In particular we get a homomorphism of fppf sheaves $c^\vee \in \underline{\Hom}_{\Oo_{\kay}}(\mathfrak{n},B)$ over $X$. Since 
\[\underline{\Hom}_{\Oo_{\kay}}(\mathfrak{n},B)\cong \underline{\mathrm{Ext}}^1_{\Oo_{\kay}}(B^\vee,\mathfrak{n}^\vee\otimes_\Z \bG_m)\] 
(see for example \cite[Proposition 3.1.5.1]{Lan}), $c^\vee$ determines a semi-abelian scheme $(G^\sharp)^\vee$ over $X$, such that there is an exact sequence of fppf sheaves of $\Oo_{\kay}$-modules 
\begin{equation} \label{eq:pi-prime}
0\rightarrow \mathfrak{m}\otimes_\Z \bG_m \rightarrow (G^\sharp)^\vee \xrightarrow{\pi'} B^\vee\rightarrow 0.
\end{equation}
Similarly $c$ determines a semi-abelian scheme $G^\sharp$ over $X$, such that there is an exact sequence of fppf sheaves of $\Oo_{\kay}$-modules
\begin{equation}\label{eq:pi}  
0\rightarrow \mathfrak{m}\otimes_\Z \bG_m \rightarrow G^\sharp \xrightarrow{\pi} B\rightarrow 0.
\end{equation}
By \cite[Lemma 3.4.2]{Lan}, the condition $c=\psi\circ c^\vee$ guarantees that there is an $\Oo_{\kay}$-linear isomorphism $\lambda^\sharp:G^\sharp \rightarrow (G^\sharp)^\vee$.
By \cite[\S 4.2]{Lan}, the datum $\tau$ in \eqref{eq:tau} gives us $1$-motives $M=[\underline{\mathfrak{n}}\xrightarrow{u} G^\sharp]$, and $M^\vee=[\underline{\mathfrak{n}}\xrightarrow{v} (G^\sharp)^\vee]$ over $X$, where $u,v$ are morphisms of fppf sheaves of $\Oo_{\kay}$-modules such that
\[\pi\circ u=c^\vee,\ \pi'\circ v=c. \]
By the proof of \cite[Proposition 3.3.3]{BHKRY1}, the morphism $M\rightarrow M^\vee$ induced by the identity map of $\mathfrak n$ and  $\lambda^\sharp$ is a principal polarization of $M$ in the sense of \cite[\S 10.2.11]{deligne1974}, which is compatible with the given polarization $\psi:B\rightarrow B^\vee$ and with the isomorphism $\mathfrak m \cong \mathfrak{n}^\vee$.

From now on we assume $R$ is normal and complete with respect to $I$, $X=\Spec R$, $Z$ is the closed subscheme $\Spec R/I$ of $X$, and $U$ is the open subscheme $X\setminus Z$ of $X$. 
Let us recall Mumford's construction (see \cite[\S 4.5]{Lan}). Let ${}^\heart G$ be the (analytic) quotient of $G^\sharp$ by the image of the period map $\underline{\mathfrak{n}}\xrightarrow{u} G^\sharp$. Then the $\Oo_{\kay}$-action $\kappa$ descends to an $\Oo_{\kay}$-action ${}^\heart \iota:\Oo_{\kay}\rightarrow \End({}^\heart G)$, and the principal polarization of $M$ gives us the principal polarization of ${}^\heart G$. The positivity of $\tau$ guarantees that we get a degenerating abelian scheme $({}^\heart G,{}^\heart \iota,{}^\heart \lambda)$ relative to $(X,Z,U)$.

\begin{theorem}\label{thm:DD DEG equivalence}
Assume that  $R$ is a Noetherian domain complete with respect to an ideal $I$ satisfying $\mathrm{rad}(I)=I$, and 
\[(X,Z,U)=(\Spec R,\Spec (R/I),\eta) \]
where $\eta$ is the generic point of $\Spec R$. Then we have an equivalence of categories
    \begin{equation}\label{eq:DD DEG equivalence}
        M^\Phi(X,Z,U):\mathrm{DD}^\Phi(X,Z,U)\rightarrow \mathrm{DEG}^\Phi(X,Z,U)
    \end{equation}
Moreover, this restricts to an equivalence of categories 
 \[\mathrm{DD}^\Phi_{(n-1,1)}(X,Z,U)\rightarrow \mathrm{DEG}^\Phi_{(n-1,1)}(X,Z,U). \]
\end{theorem}
\begin{proof}
The fact that $ M^\Phi(X,Z,U)$ is an equivalence of category is a special case of \cite[Theorem 5.1.1.4]{Lan}. It remains to show that $\mathrm{DEG}^\Phi_{(n-1,1)}(X,Z,U)$ is the essential image of the  functor
\[ \mathrm{Id}\times M^\Phi(X,Z,U):\cM_{(1,0)}(X)\times \mathrm{DD}^\Phi_{(n-1,1)}(X,Z,U)\rightarrow \cM_{(1,0)}(X)\times \mathrm{DEG}^\Phi_{(n-1,1)}(X,Z,U) \]
when restricted on $\mathrm{DD}^\Phi_{(n-1,1)}(X,Z,U)$. This is true by \cite[Lemma 2.3.5]{Ho1} and \cite[Lemma 2.3.6]{Ho1}.
\end{proof}

\subsection{Formal boundary charts}\label{subsec:boundary chart}
We describe the boundary more explicitly following \cite[\S 3.3]{BHKRY1}. 
For $\left(A_0, B,\ldots\right) \in \mathcal{M}_{(1,0)}(S) \times_{\Oo_{\kay}} \mathcal{M}_{(n-2,0)}(S)$, the \'etale sheaf $\underline{\Hom}_{\Oo_{\kay}}\left(B, A_0\right)$ is locally constant by \cite[Theorem 5.1]{BHY}.
For a fixed cusp label $\Phi=(\mathfrak{n},L_\Phi)$, let $\Lambda_\Phi=\Hom_{\Oo_{\kay}}(L_\Phi,\Oo_{\kay})$ be the hermitian dual of $L_\Phi$.
Let $\mathcal{A}_{\Phi}$ be the moduli space of triples $\left(A_0, B, \varrho\right)$ over $\mathcal{O}_{\boldsymbol{k}}$-schemes $S$, where 
$$\left(A_0,\ldots, B,\ldots\right) \in \mathcal{M}_{(1,0)}(S) \times_{\Oo_{\kay}} \mathcal{M}_{(n-2,0)}(S),$$
and 
$$\varrho: \underline{\Lambda_\Phi}\cong \underline{\operatorname{Hom}}_{\Oo_{\kay}}\left(B, A_0\right)
$$
is an isomorphism of étale sheaves of hermitian $\mathcal{O}_{\boldsymbol{k}}$-modules. Then $\mathcal{A}_\Phi\rightarrow \Spec \Oo_{\kay}$ is smooth of relative dimension $0$.
Now we define $\mathcal{B}_{\Phi}$ to be the moduli space of quadruples $\left(A_0, B, \varrho, c^\vee\right)$, where for an $\mathcal{O}_{\boldsymbol{k}}$-schemes $S$ we have $\left(A_0, B, \varrho\right)\in\mathcal{A}_{\Phi}(S)$ and $c^\vee: \mathfrak{n} \rightarrow B$ is an $\mathcal{O}_{\boldsymbol{k}}$-linear homomorphism of group schemes over $S$. In other words, 
\begin{align}\label{eq: Boundary}
\mathcal{B}_{\Phi}=\underline{\operatorname{Hom}}_{\Oo_{\kay}}(\mathfrak{n}, B)
\end{align}
where $\left(A_0, B, \varrho\right)$ is the universal object over $\mathcal{A}_{\Phi}$. More explicitly, according to \cite[Proposition 3.4.4]{BHKRY1}, we have
\begin{equation}\label{eq:B_Phi}
    \mathcal{B}_\Phi\cong E\otimes_{\Oo_{\kay}} L_\Phi,
\end{equation}
where $E=\underline{\operatorname{Hom}}_{\Oo_{\kay}}(\mathfrak n, A_0) \in \mathcal{M}_{(1,0)}(S)$ and $\otimes$ is Serre's tensor construction.
The forgetful morphism $\mathcal{B}_\Phi\rightarrow \mathcal{A}_\Phi$ is smooth of relative dimension $n-2$.

Now define $\cB_\Phi$-stacks
\[\cC_\Phi=\underline{\mathrm{Iso}}(\cL_\Phi,\Oo_{\cB_\Phi}),\ \cC_\Phi^*=\underline{\mathrm{Hom}}(\cL_\Phi,\Oo_{\cB_\Phi}),\]
where $\cL_\Phi$ is as in \eqref{eq:cL_Phi}.
In other words, $\cC_\Phi^*$ is the total space of the line bundle $\cL_\Phi^{-1}$, and $\cC_\Phi$ is the complement of the zero section $\cB_\Phi\hookrightarrow \cC_\Phi^*$.  Relative to $(\cC_\Phi^*,\cB_\Phi,\cC_\Phi)$ there is a tautological degeneration data $(B, \kappa, \psi, c, c^\vee, \tau)$ of type $\Phi$ and signature $(n-1,1)$ where
\begin{itemize}
    \item $(B, \kappa, \psi)\in \cM(n-2,0)(\cC_\Phi^*)$ is as above,
    \item $c^\vee: \mathfrak{n} \rightarrow B$ is as above and $c=\psi\circ c^\vee$,
    \item $\tau$ is determined by $\cC_\Phi^*$.
\end{itemize}

Finally define $\Delta_{\Phi}$ to be the finite group 
$$\Delta_{\Phi} = \mathrm{U}\left(\Lambda_\Phi\right) \times \mathrm{GL}_{\Oo_{\kay}}(\mathfrak{n}).$$
The group $\Delta_{\Phi}$ acts on $\mathcal{B}_{\Phi}$ by (see \cite[Remark 3.3.2]{BHKRY1})
\[ (u,t)\cdot (A_0, B, \varrho,c)=(A_0, B, \varrho\circ u^{-1} ,c\circ t^{-1}),\ (u,t)\in \mathrm{U}\left(\Lambda_\Phi\right) \times \mathrm{GL}_{\Oo_{\kay}}(\mathfrak{n}).\]
The line bundle $\cL_\Phi$ is invariant under $\Delta_\Phi$, hence the action lifts to both $\cC_\Phi$ and $\cC_\Phi^*$.

\subsection{Toroidal compactification}\label{subsec:toroidal compactification}
For each geometric point $z$ of $\cB_\Phi$ (viewed as a geometric point of $\cC_\Phi^*$ via the zero section $\cB_\Phi \hookrightarrow \cC_\Phi^*$), let $R_z$ be the \'etale local ring of $\cC_\Phi^*$ at $z$, and $I_z$ be the ideal defined by the divisor $\cB_\Phi\hookrightarrow \cC_\Phi^*$. Let $\hat{R}_z$ be the completion of $R_z$ with respect to $I_z$, and let $\hat{\eta}_z$ be the generic point of $\hat{R}_z$. As $\cC_\Phi^*$ is smooth over $\Oo_{\kay}$, both $R_z$ and $\hat{R}_z$ are Noetherian normal domains. By applying Theorem  \ref{thm:DD DEG equivalence} to the pullback of the tautological degeneration data relative to $(\cC_\Phi^*,\cB_\Phi,\cC_\Phi)$, we get a degenerating abelian scheme $({}^\heart G_z,{}^\heart \iota_z,{}^\heart \lambda_z,{}^\heart \cF_z)$ of type $\Phi$ and signature $(n-1,1)$ relative to $(\Spec \hat{R_z},\Spec (\hat{R}_z/I_z),\hat{\eta}_z)$. For every \'etale neighborhood $X^{(z)}\rightarrow \cC_\Phi^*$ of a geometric point $z$, define a closed substack of $X^{(z)}$ by 
\[ Z^{(z)}=\cB_\Phi \times_{\cC_\Phi^*} X^{(z)}, \]
and an open substack 
\[ U^{(z)}=\cC_\Phi \times_{\cC_\Phi^*} X^{(z)}. \]
\begin{proposition}\label{prop:good algebraic neighborhood}
    For every geometric point $z$ of $\cB_\Phi$ there is an \'etale neighborhood $X^{(z)}\rightarrow \cC_\Phi^*$ of $z$ and a degenerating abelian scheme $(G^{(z)},\iota^{(z)},\lambda^{(z)}, \cF^{(z)})$ of type $\Phi$ and signature $(n-1,1)$ relative to $(X^{(z)},Z^{(z)},U^{(z)})$ with the following properties.
    \begin{itemize}
        \item There exists an automorphism of $\Oo_{\kay}$-scheme $\gamma:\Spec\hat{R}_z\rightarrow \Spec\hat{R}_z$ inducing the identity on $\Spec\hat{R}_z/I_z$ such that 
        \[(G^{(z)},\iota^{(z)},\lambda^{(z)}, \cF^{(z)})|_{\Spec \hat{R}_z} \cong \gamma^*({}^\heart G_z,{}^\heart \iota_z,{}^\heart \lambda_z,{}^\heart \cF_z),\]
        where the left hand side is the pullback of $(G^{(z)},\iota^{(z)},\lambda^{(z)}, \cF^{(z)})$ via the canonical map $\Spec (\hat{R}_z)\rightarrow X^{(z)}$.
        \item The tuple $(G^{(z)},\iota^{(z)},\lambda^{(z)}, \cF^{(z)})|_{U^{(z)}}$ defines an \'etale morphism 
        \[U^{(z)}\rightarrow \cS. \]
        \item We have $\cF^{(z)}=\mathrm{ker}(\bar\epsilon:\Lie G^{(z)}\rightarrow \Lie G^{(z)})$.
    \end{itemize}
\end{proposition}
\begin{proof}
    This is essentially \cite[Proposition 2.5.1]{Ho1} which ultimately depends on results from \cite{Lan} or \cite{FC}.
\end{proof}

By the quasi-compactness of $\cB_\Phi$, we may choose finitely many geometric points $z$ so that the union of the images of $X^{(z)}\rightarrow \cC_\Phi^*$ as in Proposition \ref{prop:good algebraic neighborhood} covers $\cB_\Phi$. Let $\Phi$ vary over all cusp labels in $\mathrm{Cusp}([[L]])$, let $\cX$ be the disjoint union of the finitely many $X^{(z)}$'s so constructed, and let $\mathcal{U}$ be the disjoint union of the finitely many $U^{(z)}$'s. The obvious map defined by identifying the abelian scheme over $\mathcal{U}$
\[\cS\sqcup \mathcal{U}\rightarrow \cS \]
is an \'etale surjection, and realizes $\cS$ as the quotient of $\cS\sqcup \mathcal{U}$ by an \'etale equivalence relation 
\[\mathcal{R}_0\rightarrow (\cS\sqcup \mathcal{U})\times_{\Spec \Oo_{\kay}} (\cS\sqcup \mathcal{U}). \]
The normalization of $\mathcal{R}_0\rightarrow (\cS\sqcup \mathcal{X})\times_{\Spec \Oo_{\kay}} (\cS\sqcup \mathcal{X})$ defines a new stack $\mathcal{R}$ sitting in a commutative diagram 
\begin{equation}\label{eq:glue}
  \begin{tikzcd}
        \mathcal{R}_0 \arrow[r] \arrow[d] & \mathcal{R} \arrow[d,"r"] \\
        (\cS\sqcup \mathcal{U})\times_{\Spec \Oo_{\kay}} (\cS\sqcup \mathcal{U}) \arrow[r] & (\cS\sqcup \mathcal{X}) \times_{\Spec \Oo_{\kay}} (\cS\sqcup \mathcal{X}).
\end{tikzcd}  
\end{equation}
Exactly as in \cite[Proposition 6.3.3.13]{Lan}, the morphism $r$ is an \'etale equivalence relation. Let $\cS^*$ be the quotient of $\cS\sqcup \mathcal{X}$ by $r$. 
The following theorem cites results that we need from \cite[Theorem 3.7.1]{BHKRY1}.
\begin{theorem}\label{thm:toroidal compact}
 There is a canonical toroidal compactification $\mathcal{S} \hookrightarrow \mathcal{S}^*$ such that $\cS^*$ is flat over $\Oo_{\kay}$ of relative dimension $n-1$. It admits a stratification
 \[ \mathcal{S}^*=\cS \bigsqcup_{\Phi \in \mathrm{Cusp}([[L]])} \cS^*(\Phi)\]
as a disjoint union of locally closed substacks. 
\begin{enumerate}
    \item The $\mathcal{O}_{\boldsymbol{k}}$-stack $\mathcal{S}^*$ is regular.
    \item  The boundary divisor
$$
\partial \mathcal{S}^*=\bigsqcup_{\Phi \in \mathrm{Cusp}([[L]])} \mathcal{S}^*(\Phi)
$$
is a smooth divisor, flat over $\Oo_{\kay}$. 
    \item For each $\Phi\in \mathrm{Cusp}([[L]])$ the stratum $\mathcal{S}^*(\Phi)$ is closed. All components of $\mathcal{S}^*(\Phi)_{/ \mathbb{C}}$ are defined over the Hilbert class field $\boldsymbol{k}^{\text {Hilb }}$, and they are
    permuted simply transitively by $\operatorname{Gal}\left(\boldsymbol{k}^{\mathrm{Hilb}} / \boldsymbol{k}\right)$.  Moreover there is a canonical identification of $\Oo_{\kay}$-stacks
    \[\begin{tikzcd}
        \Delta_\Phi \backslash \cB_\Phi \arrow[r,"\pi_\Phi"] \arrow[d,hook] & \cS^*(\Phi) \arrow[d,hook] \\
        \Delta_\Phi \backslash \cC_\Phi^* & \cS^*
    \end{tikzcd}\]
    such that $\pi_\Phi$ is an isomorphism, and the two stacks in the bottom row become isomorphic after completion along their common closed substack in the top row. In other words, we have
    \[\Delta_\Phi \backslash (\cC_\Phi^*)_{\cB_\Phi}^\wedge \cong (\cS^*)^\wedge_{\cS^*(\Phi)}. \]
    \item The boundary divisor $\partial \cS^*$ does not intersect with the exceptional divisor $\Exc$.
\end{enumerate}
\end{theorem}

The universal object over $\cS$ extends to a semi-abelian scheme over $\cS^*$. The following is due to \cite[Theorem 2.5.2]{Ho1}.
\begin{proposition}\label{prop:extension of universal object}
The universal abelian scheme $\mathbf A$ over $\cS$ extends to a semi-abelian scheme $\mathbf G$ over $\cS^*$ with $\Oo_{\kay}$-action such that $\mathbf{G}|_{\cS}=\mathbf A$. At a geometric point $z=\Spec \mathbb{F}\in \cS^*(\Phi)$ where $\Phi=(\mathfrak{m}\subset M)\in \mathrm{Cusp}([[L]])$ , the semi-abelian scheme $\mathbf{G}_z$ is an extension 
\begin{equation}\label{eq:semi abelian exact sequence}
  0\rightarrow\mathfrak{m}\otimes_\Z \mathbb{G}_m \rightarrow \mathbf{G}_z \rightarrow B\rightarrow 0, 
\end{equation}
where $B$ is an abelian variety which sits in a triple $(B,\iota_B,\lambda_B)\in \cM_{(n-2,0)}(\mathbb{F})$. The flag of bundles $\cF_{\mathbf A}\subset \Lie \mathbf A$ over $\cS$ has a canonical extension to $\cF_{\mathbf G}\subset \Lie \mathbf G$ over $\cS^*$ which satisfies Kr\"amer's condition. On the complement of $\Exc$, we have
\begin{equation}
    \cF_{\mathbf G}=\mathrm{ker}(\bar\epsilon:\Lie {\mathbf G}\rightarrow \Lie {\mathbf G}).
\end{equation}
\end{proposition}

\section{Comparison of two line bundles}\label{sec:two line bundles}
The goal of this section is to compare two line bundles $\boldsymbol{\omega}$ and $\Omega$ defined respectively in \cite{BHY} (or \cite{BHKRY1}) and \cite{Ho2}, and reinterpret the constant term of the generating series of special divisors (see \eqref{eq: Z0} below) in \cite{BHKRY1}.

We define the line bundle of modular form $\boldsymbol{\omega}$ on $\cS$ following \cite[\S 2.4]{BHKRY1}. Let $\left(\mathbf{A}_0, \mathbf{A}\right)$ be the pair of universal abelian schemes over $\mathcal{S}$, let $\mathcal{F}_{\mathbf A} \subset \operatorname{Lie}(\mathbf A)$ be the universal subsheaf of Krämer's moduli problem. Recall that for any abelian scheme $A\rightarrow S$, we have the following exact sequence of locally free $\Oo_{\mathcal S}$ sheaves
\begin{equation}
0\rightarrow \Fil (A) \rightarrow H_1^{\dR} (A)\rightarrow \Lie A\rightarrow 0,    
\end{equation}
where $H_1^{\dR} (A)$ can be defined as the Lie algebra of the universal vector extension of $A$, and 
$\operatorname{Fil}(A)$ is canonically isomorphic to the $\mathcal{O}_S$-dual of $\Lie A^{\vee}$. 
The principal polarization of $\mathbf A$ induces a non-degenerate alternating form $\langle,\rangle$ on $H_1^{\dR} (\mathbf A)$, satisfying
\[ \langle \iota(a)x,y \rangle=\langle x, \iota(\bar{a})y\rangle, \forall a\in \Oo_{\kay}, x,y \in H_1^{\dR} (\mathbf A).  \]
Moreover $\Fil(\mathbf A)$ is totally isotropic with respect to $\langle,\rangle$, hence $\langle,\rangle$ induces a perfect pairing
\[\Lie \mathbf{A} \times \Fil({\mathbf A})\rightarrow \Oo_S. \]
Let
$$
\mathcal{F}_{\mathbf A}^{\perp} \subset \Fil(\mathbf A)
$$
be the orthogonal to $\mathcal{F}_{\mathbf A}$ under the pairing $\langle,\rangle$. It is a rank one $\Oo_{\mathcal{S}}$-module local direct summand on which $\mathcal{O}_{\boldsymbol{k}}$ acts through the structure morphism $\mathcal{O}_{\boldsymbol{k}} \rightarrow \mathcal{O}_{\mathcal{S}}$. Define the line bundle $\boldsymbol{\omega}$ on $\mathcal{S}$ by
$$
\boldsymbol{\omega}\coloneqq \Hom\left(\Lie\left(\mathbf A_0\right), \mathcal{F}_{\mathbf A}^{\perp}\right).
$$
Equivalently, we have 
\begin{equation}
    \boldsymbol{\omega}^{-1}=\Hom(\Fil(\mathbf{A}_0),\Lie(\mathbf A) / \mathcal{F}_{\mathbf A}).
\end{equation}

We introduce another line bundle $\Omega$ which controls the deformation theory of special divisors 
following \cite{Ho2}. The following is essentially \cite[Proposition 3.3]{Ho2}. Although \cite[Proposition 3.3]{Ho2} is about $p$-divisible groups, the same proof carries over to abelian varieties.

\begin{proposition}\label{prop:L_A} 
There are inclusions of $\mathcal{O}_{\cS}$-module local direct summands $F_{\mathbf A}^{\perp} \subset \epsilon H_1^\dR(\mathbf A) \subset H_1^\dR(\mathbf A)$. The morphism $\epsilon: H_1^\dR(\mathbf A) \rightarrow \epsilon H_1^\dR(\mathbf A)$ (see \eqref{eq:epsilon}) descends to a surjection
$$
\Lie \mathbf{A} \stackrel{\epsilon}{\rightarrow} \epsilon H_1^\dR(\mathbf A) / \mathcal{F}_{\mathbf{A}}^{\perp}
$$
whose kernel $L_{\mathbf A} \subset \Lie \mathbf{A}$ is an $\mathcal{O}_\cS$-module local direct summand of rank one. It is stable under $\mathcal{O}_{\kay}$, which acts on $\Lie \mathbf{A} / L_{\mathbf{A}}$ and $L_{\mathbf{A}}$ via $\varphi$ and $\bar{\varphi}$ respectively, where $\varphi:\OO_{\kay}\rightarrow \Oo_{\mathcal S}$ is the structure map, and $\bar\varphi$ is its Galois conjugate.
\end{proposition}

We define the line bundle $\Omega$ on $\cS$ by
\begin{equation}
  \Omega^{-1}\coloneqq\Hom\left(\Fil(\mathbf A_0), L_{\mathbf{A}}\right) .  
\end{equation}
The goal of this section is to prove the following comparison theorem. 
\begin{theorem}\label{thm:comparision of two line bundles}
We have the following equation in  $\mathrm{Pic}_\Q(\mathcal{S}^*)$.
    \begin{equation}
        \Omega=\boldsymbol{\omega}\otimes \Oo(\Exc)^{-1},
    \end{equation}
where $\Oo(\Exc)$ is the line bundle over $\cS$ associated to the effective Cartier divisor $\Exc$ defined in Section \ref{subsec:integral model}.
\end{theorem}
\begin{proof}
The theorem follows directly from Corollary \ref{cor:zero locus of s} and equation \eqref{eq:cZ_p* and Exc_p} below.
\end{proof}

Assume that $p$ is an odd prime of $\Q$ ramified in $\kay$. Following \cite[Theorem 4.5, Step 2]{Kr}, we define a substack $\cZ_p$ of the special fiber $\mathcal{S}_p$ of $\mathcal{S}$ over $p$ as follows. Let $(\mathbf{A}_{0,p}, \mathbf{A}_{p})$ be the universal abelian schemes over $\mathcal{S}_p$. Set \begin{align}\label{eq:varepsilon}
    \varepsilon_S=& \delta \otimes 1- 1\otimes \varphi(\bar\delta) \in \Oo_{\kay}\otimes_\Z \Oo_S,\\
    \bar\varepsilon_S=& \bar\delta \otimes 1- 1\otimes \varphi(\bar\delta) \in \Oo_{\kay}\otimes_\Z \Oo_S.
\end{align}
Notice that $\varepsilon$ induces a morphism from $\varepsilon: H_1^{\dR}(\mathbf{A}_p)\to  H_1^{\dR}(\mathbf{A}_p)$ such that the kernel and image of $\varepsilon$ are equal to each other. This is because   $H_1^{\dR}(\mathbf{A}_p)$ is locally free over $\Oo_{\kay}\otimes \mathcal{S}_p$ and the similar property of $\varepsilon$ acting on  $\Oo_{\kay}\otimes \mathcal{S}_p$ holds (on $\mathcal{S}_p$ we have $\varepsilon=\delta \otimes 1$). Note that now we have $\varepsilon=-\bar{\varepsilon}$.
Define on $\varepsilon H_1^{\dR}(\mathbf{A}_p)$  an $\Oo_{\mathcal{S}_p}$-bilinear form $\{,\}$ by
\begin{equation}\label{eq:definition of {,}}
    \{\varepsilon x, \varepsilon y\}=\langle \varepsilon x,y\rangle.
\end{equation} 
As $\varepsilon H_1^{\dR}(\mathbf{A}_p)$ is isotropic with respect to $\langle,\rangle$, the form $\{,\}$ is well-defined.
Moreover $\{,\}$ is symmetric:
\[\{\varepsilon x, \varepsilon y\}=\langle (\pi \otimes 1 ) x,y\rangle=
-\langle y, (\pi \otimes 1 ) x\rangle=\langle (\pi \otimes 1 )y,x\rangle=\{\varepsilon y, \varepsilon x\}.
\]

\begin{definition}
Let $\cZ_p$ be the substack of $\mathcal{S}_p$ such that for any $\Oo_{\kay}/(\mathfrak p)$-scheme $S$, $\cZ_p(S)$ is the groupoid of isomorphism classes $(A_0,\iota_0,\lambda_0,A,\iota,\lambda , \mathcal{F})\in \mathcal{S}_p(S)$ such that $\mathcal{F}_{A}^\bot$ is isotropic with respect to $\{,\}$. Here $\mathfrak{p}$ is the ideal of $\mathcal{O}_{\boldsymbol{k}}$ such that $\mathfrak p^2=p\mathcal{O}_{\boldsymbol{k}}$.
\end{definition}

\begin{lemma}\label{lem:cZ_p}
    $\cZ_p$ is a regular scheme and a Cartier divisor of $\mathcal{S}$. Moreover we have the following equation of Cartier divisors
    \begin{equation}\label{eq:special fiber}
        \mathcal{S}_p=\cZ_p+\Exc_p.
    \end{equation}
\end{lemma}
\begin{proof}
First we show that
\[\mathcal{S}_p^{\mathrm{red}}=\cZ_p^{\mathrm{red}}\cup \Exc_p .\]
Let $z$ be a geometric point in $\mathcal{S}_p\setminus \cZ_p$.
By the same argument of \cite[Theorem 4.5, Step 3]{Kr}, we know that both $\Fil(\mathbf{A}_z)$ and $\varepsilon H_1^{\dR}(\mathbf{A}_z)$ are equal to 
\[ \mathcal{F}_{\mathbf{A}_z}^\bot \oplus (\varepsilon^{-1} \mathcal{F}_{\mathbf{A}_z}^\bot)^\bot. \]
Here $(\varepsilon^{-1} \mathcal{F}_{\mathbf{A}_z}^\bot)^\bot$ is the perpendicular complement of $\varepsilon^{-1} \mathcal{F}_{\mathbf{A}_z}^\bot$ in $ H_1^{dR}(\mathbf{A}_z)$ with respect to $\langle,\rangle$. This implies that $\varepsilon=\pi\otimes 1$ acts trivially on $\Lie \mathbf{A}_z$, so $z$ lies in $\Exc_p$ by definition.

By the same calculation as in \cite[Theorem 4.5, Step 4]{Kr}, we can conclude the fact that $\cZ_p$ is a regular Cartier divisor and \eqref{eq:special fiber}. 
\end{proof}

\begin{proposition}\label{prop:zero locus of natural map}
The natural map $L_{\mathbf A}\rightarrow \Lie \mathbf{A}/\mathcal{F}_{\mathbf{A}}$ defines a section $s\in \Hom( \Omega^{-1}, \boldsymbol{\omega}^{-1})$. We have the following equation for the zero locus of $s$.
    \[(s)=\sum_{p\mid D} \cZ_p, \]
    where the summation is over all finite primes ramified in $\kay$.
\end{proposition}
\begin{proof}
Step 1.
Let $z$ be a geometric point of $(s)$ with residue field $\kappa$. If the characteristic of $\kappa$ does not divide the discriminant $D$ of $\kay$, then the structural morphism $\varphi:\Oo_{\kay}\rightarrow \kappa$ and its conjugate $\bar\varphi$ are not the same. Since $\Oo_{\kay}$ acts on  $\cF_{\mathbf{A}_z}$ via $\varphi$, and on $L_{\mathbf{A}_z}$ via $\bar\varphi$, $L_{\mathbf{A}_z}\not\subset \cF_{\mathbf{A}_z}$, which means the map $s_z$ has to be nonzero, a contradiction. Hence $z$ is in $\mathcal{S}_p$ for some ramified $p$. 

Step 2. Assume $z$ is a geometric point of $\mathcal{S}_p$ for a ramified $p$. As $\cF_{\mathbf{A}_z}^\bot \subset \varepsilon H_1^{\dR}(\mathbf{A}_z)$, we can assume $\cF_{\mathbf{A}_z}^\bot$ is spanned by $\varepsilon x$ for some $x\in H_1^{\dR}(\mathbf{A}_z)$. Since the kernel and image of $\varepsilon$ acting on $H_1^{\dR}(\mathbf{A}_z)$ are equal to each other, we have   
\[\varepsilon^{-1} \cF_{\mathbf{A}_z}^\bot=\mathrm{Span}\{x\}\oplus\varepsilon H_1^{\dR}(\mathbf{A}_z).\]
The line $L_{\mathbf{A}_z}$ by definition is $\varepsilon^{-1} \cF_{\mathbf{A}_z}^\bot/ \Fil(\mathbf{A}_z)\subset \Lie \mathbf{A}_z$.  Hence
\[L_{\mathbf{A}_z}\subset \cF_{\mathbf{A}_z}\Leftrightarrow \langle x, \cF_{\mathbf{A}_z}^\bot\rangle=0 \text{ and } \langle \varepsilon H_1^{\dR}(\mathbf{A}_z), \cF_{\mathbf{A}_z}^\bot\rangle=0. \]
The last condition is automatic as $\cF_{\mathbf{A}_z}^\bot\subset \varepsilon H_1^{\dR}(\mathbf{A}_z)$ by Proposition \ref{prop:L_A} and $\varepsilon H_1^{\dR}(\mathbf{A}_z)$ is isotropic. Hence
\[L_{\mathbf{A}_z}\subset \cF_{\mathbf{A}_z}\Leftrightarrow \langle x, \cF_{\mathbf{A}_z}^\bot\rangle=0 \Leftrightarrow \langle x, \varepsilon x\rangle=0.\]
By the definition of $\{,\}$, the last condition is true if and only if $\{x,x\}=0$, i.e., $\cF_{\mathbf{A}_z}^\bot$ is isotropic with respect to $\{,\}$. Hence
\[z\in (s)\cap \mathcal{S}_p\Leftrightarrow z\in \cZ_p. \]
In fact in the above argument we can replace $z$ by any Artinian scheme over $\mathcal{S}_p$ and conclude by Nakayama's Lemma that
\[(s)|_{\mathcal{S}_p}=\cZ_p.\]

Step 3: We have shown that $(s)$ is supported on the primes ramified in $\kay$ and its special fiber over such a prime $p$ is $\cZ_p$. Since $\cZ_p$ is regular, by \cite[Lemma 10.3]{RTZ}, in order to prove the proposition it remains to show that $(s)$ has no $\Oo_{\kay,p}/(\pi^2)$-point for any ramified prime $p$. Assume $\tilde z$ is such a point. Then on one hand $\pi$ acts on $\cF_{\mathbf{A}_{\tilde z}}$ by $\varphi(\pi)$. On the other hand $\pi$ acts on $L_{\mathbf{A}_{\tilde z}}\subset \cF_{\mathbf{A}_{\tilde z}}$ by $\varphi(\bar\pi)$. Since $p\neq 2$, we know that $\pi\neq \bar\pi$ in $\Oo_{\kay}/(\pi^2)$. This is a contradiction. This finishes the proof of the proposition.
\end{proof}

\subsection{Extensions of the line bundles to the boundary}
Assume $p$ is a prime divisor of $D$.
Let $\cZ_p^*$ be the Zariski closure of $\cZ_p$ in $\cS_p^*$, then $\cZ_p^*+ \hbox{Exc}_p =\cS_p^*$  as Cartier divisors on $\cS^*$.
Since $\mathcal{S}_p^*$ is the Cartier divisor defined by an ideal $\mathfrak{p}$ such that $\mathfrak{p}^2=(p)$, we have the following equation in $\mathrm{CH}_\Q^1(\mathcal S^*)$.
\begin{equation}\label{eq:cZ_p* and Exc_p}
    \cZ_p^*+\Exc_p=0.
\end{equation}

By \cite[Theorem 3.7.1(6)]{BHKRY1}, the line bundle $\boldsymbol{\omega}$ admits a canonical extension to the compactification $\cS^*$ which is still denoted by $\boldsymbol{\omega}$ such that  
\begin{equation}
    \boldsymbol{\omega}^{-1}=\Hom(\Fil(\mathbf{A}_0),\Lie(\mathbf G) / \mathcal{F}_{\mathbf G}).
\end{equation}
Here we recall from Proposition \ref{prop:extension of universal object} that $\mathbf G$ is the extension of the universal abelian scheme $\mathbf A$ over $\cS$ to $\cS^*$. We have a similar result for $\Omega$.

\begin{lemma}\label{lem:equality of bundles outside Exc}
    We have the equality of vector bundles over $\cS\setminus \Exc$.
    \[ L_{\mathbf A}=\bar{\epsilon} \Lie {\mathbf A}.\]
\end{lemma}
\begin{proof}
    By the fact that $\epsilon \bar \epsilon=0$, we clearly have 
    \begin{equation}\label{eq:bar epsilon Lie A in L_A}
        \bar{\epsilon} \Lie {\mathbf A}\subset L_{\mathbf A}.
    \end{equation}
    Now let $z$ be any geometric closed point of $\cS\setminus \Exc$. 
    Since $L_{\mathbf A}$ has rank $1$, if the inclusion  $\bar{\epsilon} \Lie {\mathbf{A}_z}\subset L_{\mathbf{A}_z}$ is strict, we must have $\bar{\epsilon} \Lie {\mathbf{A}_z}=0$. In other words $\bar\epsilon \Lie \mathbf{A}_z=\{0\}$, which forces $z$ to be a point in $\Exc$ by the definition of $\Exc$. This is a contradiction. So we know that the lemma is true on the level of geometric closed points. Since $\bar\epsilon:\Lie \mathbf{A}\rightarrow \Lie \mathbf{A}$ is a section of the coherent sheaf $\End(\Lie \mathbf{A})$, the lemma is true by Nakayama's lemma.
\end{proof}

\begin{corollary}\label{cor:L_G on boundary}
    The line bundle $L_{\mathbf A}$ can be extended to a line bundle over $\cS^*$  denoted by $L_{\mathbf G}$ such that for a geometric point $z=\Spec \mathbb{F}\in \partial \cS^*$ on the boundary,
    \[L_{\mathbf{G}_z}=\bar{\epsilon} \Lie \mathbf{G}_z \subset \Lie \mathbf{G}_z.   \]
\end{corollary}
\begin{proof}
We first claim that $\bar\epsilon \Lie {\mathbf G}$ is a line bundle over $\cS^*\setminus\Exc$. By Nakayama's lemma, it suffices to check all geometric closed points. For points in $\cS\setminus\Exc$, the claim follows from Lemma \ref{lem:equality of bundles outside Exc}. For a geometric point $z$ on $\cB_\Phi$, we have the exact sequence of Lie algebras
\[0\rightarrow\mathfrak{m}\otimes_\Z \Lie \mathbb{G}_m \rightarrow \Lie \mathbf{G}_z \rightarrow \Lie B\rightarrow 0,\]
from the exact sequence \eqref{eq:semi abelian exact sequence}. Now since $(B,\iota_B,\lambda_B)\in \cM_{(n-2,0)}(\mathbb{F})$, we have $\bar\epsilon \Lie B=0$. So $\bar{\epsilon} \Lie \mathbf{G}_z=\bar{\epsilon} \Lie(\mathfrak{m}\otimes_\Z \mathbb{G}_m)$ is of dimension $1$.

Now by Lemma \ref{lem:equality of bundles outside Exc} and the (5) of Theorem \ref{thm:toroidal compact}, over $\cS^*\setminus \Exc$, we can define $L_{\mathbf G}$ to be $\bar{\epsilon} \Lie \mathbf{G}$. The corollary follows.
\end{proof}

As a consequence, $\Omega$ can be canonically extended to a line bundle over $\cS^*$ which is defined by 
\begin{equation}
  \Omega^{-1}\coloneqq\Hom\left(\Fil(\mathbf A_0), L_{\mathbf{G}}\right) .  
\end{equation}
The natural map $L_{\mathbf A}\rightarrow \Lie \mathbf{A}/\mathcal{F}_{\mathbf{A}}$ also has a canonical extension to $L_{\mathbf G}\rightarrow \Lie \mathbf{G}/\mathcal{F}_{\mathbf{G}}$ which defines a section in $\Hom( \Omega^{-1}, \boldsymbol{\omega}^{-1})$ still denoted by $s$. 
\begin{corollary}\label{cor:zero locus of s}
We have the following equation on the zero locus of $s$ on $\cS^*$.
    \[(s)=\sum_{p} \cZ^*_p, \]
    where the summation is over all finite primes ramified in $\kay$.
\end{corollary}
\begin{proof}
The flatness of the boundary divisor $\partial \cS^*$ implies that
every irreducible component of $\partial \cS^*$ has some closed point of characteristic prime
to $D$, and Step 1 of the proof of Proposition  \ref{prop:zero locus of natural map} implies that such a point cannot lie in $(s)$. Hence none of the boundary component belongs to $(s)$. Since both sides of the equation are Cartier divisors, the corollary follows from Proposition \ref{prop:zero locus of natural map}.
\end{proof}

\subsection{Constant term of the generating series}\label{subsec:constant term}
For $m\leq 0$, using the correspondence between line bundles and Cartier divisors, we define the the class $\cZ^*(m)\in \mathrm{Pic}(\cS^*)\cong \mathrm{CH}^1(\cS^*)$  by
\begin{equation}\label{eq: Z0}
    \cZ^*(m)\coloneqq 
    \begin{cases}
    \Omega^{-1}&\text{if } m=0,\\ 
    \Oo_{\cS} &\text{if } m<0.
    \end{cases}
\end{equation}
 By Theorem \ref{thm:comparision of two line bundles}, the definition of $\cZ^*(0)$ (the constant term in the generating series of special divisors)
agrees with that of \cite[Equation (1.1.4)]{BHKRY1}. We will see in the next section that as $\Omega$ controls the deformation of special divisors, this definition of the constant term is also compatible with pullbacks, see \S \ref{sec:pullback I}.

\section{Deformation of special divisors near the boundary}\label{sec:deformation of special divisors}
It is discovered by \cite{Ho2} that the line bundle $\Omega$ controls the deformation theory of special divisors on the open Shimura variety $\cS$, see Proposition \ref{prop:deformation of Z x interior} below. The main goal of this section is to extend this observation to $\cS^*$ (Proposition \ref{prop:deformation cS star}). In order to do this we define $\cZ^*(m)$ as a Deligne-Mumford stack (see Definition \ref{def: cZ*} and Proposition \ref{prop: Z*m representability}), which should be thought of as the ``toroidal compactification" of the stack $\cZ(m)$. By  Proposition \ref{prop:compare Z* and Z*Phi} and \cite[Theorem 3.7.1(4)]{BHKRY1}, our definition of $\cZ^*(m)$ agrees with that of \cite{BHKRY1} which is the Zariski closure of $\cZ(m)$.

First we review some deformation theory following \cite[\S 4]{Ho2}. 
Throughout this section we will repeatedly encounter the following setting. Let $S$ be a stack and $Z$ be another stack with a finite unramified morphism $Z\rightarrow S$. So there is an \'etale morphism $U\rightarrow S$ such that $U$ is a scheme and $Z|_U\rightarrow S|_U$ is a closed immersion of schemes on every connected component of its domain. Let $\cZ$ be such a connected component and $I_{\cZ}$ be its ideal sheaf in $S|_U$. Then $I_{\cZ}^2$ defines a larger closed subscheme 
\begin{align*}
    \cZ\hookrightarrow \tilde{\cZ} \hookrightarrow S|_U.
\end{align*}
The image of $I_{\cZ}$ in the $\Oo_{\tilde \cZ}$ can be equipped with the trivial P.D. structure. The stack $\tilde{\cZ}$ is called the first order thickening of $\cZ$ in $S|_U$.

Now we specialize to the case $S=\cS$ and $Z=\cZ(m)$.
By \cite[Proposition 2.9]{KR2}, for any geometric point $z \in \cS$, there is an \'etale neighborhood $U$ of $z$ in $\cS$ such that $\cZ(m)|_U\rightarrow \cS|_U$ is a closed immersion of schemes on every connected component of its domain. 
Let $\cZ$ be such a connected component as in the previous paragraph.
The universal morphism $x:\mathbf{A}_0|_{\cZ}\rightarrow \mathbf{A}|_{\cZ}$ induces a morphism of vector bundles over $\cZ$.
\[H_1^\dR(\mathbf A_0|_{\cZ})\xrightarrow{x} H_1^\dR(\mathbf A|_{\cZ}), \]
which maps $\Fil(\mathbf A_0|_{\cZ})$ to $\Fil(\mathbf A|_{\cZ})$. By Grothendieck-Messing theory (\cite{Messing}), this morphism admits a canonical extension to a morphism between vector bundles over $\tilde{\cZ}$,
\[H_1^\dR(\mathbf A_0|_{\tilde{\cZ}})\xrightarrow{\tilde{x}} H_1^\dR(\mathbf A|_{\tilde{\cZ}}), \]
which determines a morphism (still denoted by $\tilde{x}$)
\begin{equation}\label{eq:obstruction map Lie A}
    \Fil(\mathbf A_0|_{\tilde{\cZ}})\xrightarrow{\tilde{x}} \Lie(\mathbf A|_{\tilde{\cZ}}).
\end{equation}
The following is an analogue of \cite[Proposition 4.1]{Ho2}.

\begin{proposition}\label{prop:deformation of Z x interior}
The morphism \eqref{eq:obstruction map Lie A} takes values in the rank one local direct summand
\[
 L_{\mathbf{A}} |_{\tilde{\cZ}} \subset \Lie \mathbf{A}|_{\tilde{\cZ}},
\]
and so can be viewed as a morphism of line bundles
\begin{equation}\label{eq:obstruction map L_A}
    \operatorname{Fil}\left(\mathbf{A}_0\right)|_{\tilde{\cZ}} \stackrel{\tilde{x}}{\rightarrow} L_{\mathbf{A}} |_{\tilde{\cZ}}.
\end{equation}
The Kudla-Rapoport divisor $\cZ$ is the largest closed subscheme of $\tilde{\cZ}$ over which \eqref{eq:obstruction map L_A} is trivial.
\end{proposition}
\begin{proof}
The proof is identical to that of \cite[Proposition 4.1]{Ho2} if we replace $p$-divisible groups by abelian schemes.   
\end{proof}

\subsection{Special divisors on formal boundary charts}\label{subsec:special divisors in boundary chart}
We would like to study special divisors on the boundary chart $\cC_\Phi^*$.  Let  $\Phi=(\mathfrak{n},L_\Phi)$ and $\mathfrak{m}$ be as  in \S \ref{subsec:toroidal compactification}. We are in the setting of \S\ref{subsec:degenerating abelian schemes} with 
\[(X,Z,U)=(\cC_\Phi^*,\cB_\Phi,\cC_\Phi).\]
Recall from \S\ref{subsec:degenerating abelian schemes} and \S\ref{subsec:boundary chart} that there is tautological degeneration data $(A_0,\iota_0,\lambda_0, B, \kappa, \psi, c, c^\vee, \tau)$ of type $\Phi$ and signature $(n-1,1)$ over $(\cC_\Phi^*,\cB_\Phi,\cC_\Phi)$, and using this we can construct semi-abelian schemes $G^\sharp$ and $(G^\sharp)^\vee$ and $1$-motives $M$ over $M^\vee$ over $\cC_\Phi^*$. In fact, $G^\sharp$ and $(G^\sharp)^\vee$ are defined over $\cB_\Phi$ and can be viewed as objects over $\cC_\Phi^*$ through the canonical projection $\cC_\Phi^*\rightarrow \cB_\Phi$.

Suppose we have a morphism of group schemes $x:A_0\rightarrow G^\sharp$ over $\cB_\Phi$. Its dual morphism $x^\vee:(G^\sharp)^\vee\rightarrow A_0^\vee$ is defined as follows. By the definition of the dual $1$-motive $M^\vee$, we have
\[ (G^\sharp)^\vee= \underline{\mathrm{Ext}}^1([\underline{\mathfrak n}\xrightarrow{c^\vee}B],\bG_m).\]
For any $\cB_\Phi$-scheme $S$, an $S$-point $\xi$ of $(G^\sharp)^\vee$ is a commutative diagram of group schemes over $S$
\[   \begin{tikzcd}
 &0  \arrow[r] \arrow[d]
& \mathfrak n \arrow[d] \arrow[r] & \mathfrak{n} \arrow[r] \arrow[d,"c^\vee"] & 0  \\
  0 \arrow[r] &\bG_m \arrow[r] 
& \tilde{B} \arrow[r] & B \arrow[r] & 0
\end{tikzcd}
\]
where the rows are exact. Then $x^\vee(\xi)\in A_0^\vee(S)=\underline{\mathrm{Ext}}^1(A_0,\bG_m)(S)$ is the extension  
$\tilde{A_0}\rightarrow A_0$ which pulls back the extension $\tilde B\rightarrow B$ along $\pi\circ x$. Namely we have the Cartesion diagram 
\[\begin{tikzcd}
    \tilde{A}_0 \arrow[r] \arrow[d] \arrow[dr, phantom, "\square"] & A_0 \arrow[d,"\pi\circ x"]\\
    \tilde{B} \arrow[r] & B.
\end{tikzcd}
\]
If we view $B^\vee$ as $\underline{\mathrm{Ext}}^1(B,\bG_m)$, we immediately see that 
\begin{equation}
    x^\vee=(\pi\circ x)^\vee\circ \pi'    
\end{equation}
where $\pi$ and $\pi'$ are defined in \eqref{eq:pi} and \eqref{eq:pi-prime}.
Now we are ready to define special divisors on $\cB_\Phi$, $\cC_\Phi$ and $\cC_\Phi^*$. For any morphism of $\Oo_{\kay}$-schemes $S\rightarrow \cB_\Phi$, we define a hermitian form on $\Hom_{\Oo_{\kay}}(A_{0,S}\ , G^\sharp_S)$ in the same way as in 
\eqref{eq:hermitian form on Hom}:
\begin{equation}
        \left( x_1, x_2\right)=\iota_0^{-1}(\lambda_0^{-1} \circ x_2^\vee \circ \lambda^\sharp \circ x_1) \in \Oo_{\kay}
\end{equation}

\begin{lemma}\label{lem:decsending lambda sharp}
For any   $x_1,x_2\in \Hom_{\Oo_{\kay}}(A_{0,S}\ , G^\sharp_S)$, we have
\[\left( x_1, x_2\right)=\left( \pi\circ x_1, \pi\circ x_2\right)_\psi, \]
where $\pi:G^\sharp \rightarrow B$ is defined in \eqref{eq:pi} and $(,)_\psi$ is the hermitian form on $\Hom_{\Oo_{\kay}}(A_{0,S}\ , B_S)$ defined in the same way as \eqref{eq:hermitian form on Hom} using the polarization $\psi:B\rightarrow B^\vee$. In particular $\left(,\right)$ is positive definite.
\end{lemma}
\begin{proof}
We have
\begin{align*}
    \left( x_1, x_2\right)=&\iota_0^{-1}(\lambda_0^{-1} \circ x_2^\vee \circ \lambda^\sharp \circ x_1)\\
    =&\iota_0^{-1}(\lambda_0^{-1} \circ (\pi\circ x_2)^\vee\circ \pi' \circ \lambda^\sharp \circ x_1)\\
    =&\iota_0^{-1}(\lambda_0^{-1} \circ (\pi\circ x_2)^\vee\circ \psi\circ \pi \circ x_1)\\
    =&\left( \pi\circ x_1, \pi\circ x_2\right)_\psi.
\end{align*}
The third equality above uses the compatibility of $\psi$ and $\lambda^\sharp$: $\pi'\circ \lambda^\sharp=\psi\circ \pi$ by the definition of polarization of $1$-motives.

If $\pi\circ x_1=0$, then $x_1$ factors through $\mathfrak{m}\otimes\mathbb{G}_m$ which is affine. This forces $x_1$ to be $0$.
The positivity of $\left(,\right)$ now follows from that of $\left(,\right)_\psi$ by \cite[Lemma 2.7]{KR2}.
\end{proof}

The following definition is due to \cite[Remark 3.6.3]{BHKRY1}.
\begin{definition}
For $m\in \Z_{> 0}$, define $\cZ_{\cB_\Phi}(m)$ to be the stack assigning for any morphism of $\Oo_{\kay}$-schemes $S\rightarrow \cB_\Phi$, the set
\[\cZ_{\cB_\Phi}(m)(S)=\{x\in \Hom_{\Oo_{\kay}}(A_{0,S}\ , G^\sharp_S) \mid ( x,x)=m\}.\]
   We also  define 
\begin{equation}\label{eq:factor through B_Phi}
   \cZ_\Phi(m)=\cZ_{\cB_\Phi}(m)\times_{\cB_\Phi} \cC_\Phi,\  \cZ_\Phi^*(m)=\cZ_{\cB_\Phi}(m)\times_{\cB_\Phi} \cC_\Phi^*. 
\end{equation}
\end{definition} 

By \cite[Proposition 3.6.2]{BHKRY1}, the natural morphism $\cZ_{\cB_\Phi}(m)\rightarrow \cB_\Phi$ is finite and unramified as $\cZ_{\cB_\Phi}(m)$ is a finite disjoint union of stacks, each of which maps to $\cB_\Phi$ as a Cartier divisor. We can view $\cZ_{\cB_\Phi}(m)$ (or rather its image under the morphism to $\cB_\Phi$) as an effective Cartier divisor of $\cB_\Phi$ that is flat over $\mathcal{A}_\Phi$ and $\Spec\Oo_{\kay}$. As $\cC_\Phi$ and $\cC_\Phi^*$ are smooth over $\cB_\Phi$, $\cZ_\Phi(m)$ and $\cZ_\Phi^*(m)$ are smooth over $\cZ_{\cB_\Phi}(m)$, and are effective Cartier divisors of $\cC_\Phi$ and $\cC_\Phi^*$ respectively, flat over $\mathcal{A}_\Phi$ and $\Spec\Oo_{\kay}$.

\subsection{Deformation of special divisors on formal boundary charts}
Now consider $G^\sharp$ as the $1$-motive $[0\rightarrow G^\sharp]$ over $\cB_\Phi$. Let $H_1^\dR(G^\sharp)$ be the de Rham realization of $G^\sharp$ (see \cite[\S 10.1.7]{deligne1974}) defined by the Lie algebra of its universal vector extension, which is a rank $2n-2$ vector bundle over $\cB_\Phi$, locally free of rank $n-1$ over $\Oo_{\kay}\otimes_\Z  \OO_{\cB_\Phi}$.  Then $\Fil(G^\sharp)=F^0 H_1^\dR(G^\sharp)$ is a locally free direct summand of  $H_1^\dR(G^\sharp)$ of $\Oo_{\cB_\Phi}$-rank $n-2$. We have the following exact sequence of Lie algebras
\begin{equation}\label{eq:exact sequence Lie G sharp}
    0\rightarrow\mathfrak{m}\otimes_\Z \Oo_{\cB_\Phi} \rightarrow \Lie G^\sharp \rightarrow \Lie B\rightarrow 0.
\end{equation}
Define
\begin{equation}\label{eq:LGsharp}
  L_{G^\sharp}=\bar\epsilon \Lie G^\sharp.
\end{equation} 
Then by the exact sequence \eqref{eq:exact sequence Lie G sharp} and an argument similar to that of \cite[Lemma 2.3.6]{Ho1}, we can easily see that $L_{G^\sharp}$ is an $\Oo_{\cB_\Phi}$-module local direct summand of rank one. It is stable under $\mathcal{O}_{\kay}$, which acts on $\Lie G^\sharp / L_{G^\sharp}$ and $L_{G^\sharp}$ via $\varphi, \bar{\varphi}$ : $\mathcal{O}_{\kay} \rightarrow \mathcal{O}_{\cB_\Phi}$, respectively.  Pulling back to $\cC_\Phi^*$, we get vector bundles $H_1^\dR(G^\sharp)$, $\Fil(G^\sharp)$, $\Lie G^\sharp$, and $L_{G^\sharp}$ over $\cC_\Phi^*$ denoted by the same notation.

We now revisit the setting at the beginning of this section with $S=\cC_\Phi^*$ and $Z=\cZ_\Phi^*(m)$.
For any geometric point $z \in \cC_\Phi$,  there is an \'etale neighborhood $U$ of $z$ in $\cC_\Phi^*$ such that $\cZ_\Phi^*(m)|_U\rightarrow \cC_\Phi^*|_U$ is a closed immersion of schemes on every connected component of its domain. Let $\cZ$ be such a connected component.
The universal morphism $x:A_0|_{\cZ}\rightarrow G^\sharp|_{\cZ}$ induces a morphism of vector bundles over $\cZ$.
\[H_1^\dR(A_0|_{\cZ})\xrightarrow{x} H_1^\dR(G^\sharp|_{\cZ}), \]
which maps $\Fil(A_0|_{\cZ})$ to $\Fil(G^\sharp|_{\cZ})$.
By Grothendieck-Messing theory for $1$-motives (see for example \cite{AB11}), this morphism admits a canonical extension to a morphism between vector bundles over $\tilde{\cZ}$,
\[H_1^\dR(A_0|_{\tilde{\cZ}})\xrightarrow{\tilde{x}} H_1^\dR(G^\sharp|_{\tilde{\cZ}}), \]
which determines a homorphism
\begin{equation}\label{eq:obstruction map Lie G sharp}
    \Fil(A_0|_{\tilde{\cZ}})\xrightarrow{\tilde{x}} \Lie(G^\sharp|_{\tilde{\cZ}}).
\end{equation}
The following is an analogue of Proposition \ref{prop:deformation of Z x interior}.
\begin{proposition}\label{prop:deformation C Phi star}
The homorphism \eqref{eq:obstruction map Lie G sharp} takes values in the rank one local direct summand
\[
 L_{G^\sharp} |_{\tilde{\cZ}} \subset \Lie G^\sharp|_{\tilde{\cZ}},
\]
  and so can be viewed as a morphism of line bundles
\begin{equation}\label{eq:obstruction map L_G sharp}
    \operatorname{Fil}\left(A_0\right)|_{\tilde{\cZ}} \stackrel{\tilde{x}}{\rightarrow} L_{G^\sharp}|_{\tilde{\cZ}}.
\end{equation}
The divisor $\cZ$ is the largest closed formal subscheme of $\tilde{\cZ}$ over which \eqref{eq:obstruction map L_G sharp} is trivial.
\end{proposition}
\begin{proof}
By the signature condition of $A_0$, we know that $\bar\epsilon H_1^{\dR}(A_0) \subset \Fil (A_0)$. Since both of them are rank 1 local direct summands of
the vector bundle $H_1^{\dR}(A_0)$, we know they are equal. Hence \eqref{eq:obstruction map Lie G sharp} takes values in $\bar\epsilon \Lie G^\sharp|_{\tilde{\cZ}}$ which is exactly $L_{G^\sharp}|_{\tilde{\cZ}}$ by definition. This proves the first claim of the proposition. The second claim follows from Grothendieck-Messing theory for $1$-motives, see for example \cite[Theorem 2.1 (iv), Remark 2.2 (c)]{AB11}.
\end{proof}

\subsection{Special divisors on the toroidal compactification}\label{subsec:extended special divisors}
The tautological degeneration data $(B, \kappa, \psi, c, c^\vee)$ of type $\Phi$ and signature $(n-1,1)$ over $\cS^*(\Phi)$ determines a morphism of semi-abelian schemes $\lambda_{\mathbf{G}}:\mathbf{G}\rightarrow \mathbf{G}^\vee$ which extends the principal polarization $\lambda:\mathbf{A}\rightarrow\mathbf{A}^\vee$. Using 
$\lambda_{\mathbf{G}}$ we can define a hermitian form $(,)$ on $\operatorname{Hom}_{\Oo_{\kay}}\left(A_0, G\right)$ in the same way as in \eqref{eq:hermitian form on Hom}. By Proposition \ref{prop:good algebraic neighborhood} and Lemma \ref{lem:decsending lambda sharp}, the form is positive definite.

\begin{definition}\label{def: cZ*}
For any  $m \in \mathbb{Z}_{>0}$, define  $\mathcal{Z}^*(m)$  to be the stack assigning to each $\Oo_{\kay}$-scheme $S$ with an $\Oo_{\kay}$-morphism $S\rightarrow \cS^*$ the set 
\[ \mathcal{Z}^*(m)(S)=\{x \in \operatorname{Hom}_{\Oo_{\kay}}\left(A_0, G\right)\mid  ( x, x)=m\}, \]
where $\left(A_0,\iota_0,\lambda_0, G,\iota,\lambda,\cF\right)$ is the pullback of the universal object over $\mathcal{S}^*$ to $S$.   
\end{definition}
Recall that we have the following rigidity lemma of semi-abelian schemes.
\begin{lemma}\label{lem:rigidity}
Let $S$ be a noetherian scheme and $S_0$ be its reduced locus. Assume $G$ and $H$ are semi-abelian schemes over $S$. Then the reduction map induces a canonical inclusion 
\[\Hom(G,H)\hookrightarrow \Hom(G|_{S_0}, H|_{S_0}), \]
where we take homomorphisms between group schemes.
\end{lemma}
\begin{proof}
This follows easily from the rigidity lemma of abelian schemes and the fact that there are no non-constant homomorphisms from tori to abelian schemes.  
\end{proof}

\begin{proposition}\label{prop: Z*m representability}
For each integer $m>0$,  $\cZ^*(m)$ is a Deligne-Mumford stack containing $\cZ(m)$ as an open substack. The forgetful morphism $\cZ^*(m)\rightarrow \cS^*$ is finite and unramified.
\end{proposition}
\begin{proof}
Given an $S$-valued point $\left(A_0,\iota_0,\lambda_0, G,\iota,\lambda,\cF\right)$ of $\cS^*$, the functor $\underline{\Hom}_{\Oo_{\kay}}\left(A_0, G\right)$ on $(\mathrm{Sch}/S)$ defined by
\[S'\mapsto \operatorname{Hom}_{\Oo_{\kay}}\left(A_0\times_S S', G\times_S S'\right) \]
is representable by a scheme by Grothendieck's representability theorem on $\Hom$ sheaves. Given this, since $\cS^*$ is a Deligne-Mumford stack, so is $\cZ^*(m)$.

Unramifiedness of the forgetful morphism $\cZ^*(m)\rightarrow \cS^*$ follows from the infinitesimal criterion for unramifiedness and Lemma \ref{lem:rigidity}. Quasi-finiteness  of the morphism $\cZ^*(m)\rightarrow \cS^*$ follows from the positivity of the hermitian form $(,)$. To prove the properness of the morphism $\cZ^*(m)\rightarrow \cS^*$, we use the valuation criterion. Assume that we have a commutative diagram of morphisms
\[\begin{tikzcd}
    \Spec K \arrow[r] \arrow[d]  & \cZ^*(m) \arrow[d]\\
    \Spec V \arrow[r]  & \cS^*.
\end{tikzcd}
\]
where $V$ is a DVR with fractional field $K$, we would like to show that there exists a dashed arrow $\Spec V \dasharrow \cZ^*(m)$ after adding which the diagram still commutes. Let $\left(A_0,\iota_0,\lambda_0, G,\iota,\lambda,\cF\right)$ be the pullback of the universal object over $\cS^*$ to $\Spec V$. By assumption there exists $x|_K\in \Hom_{\Oo_{\kay}}(A_0|_{\Spec K}, G|_{\Spec K})$ such that $(x|_K,x|_K)=m$. By \cite[Chapter I, Proposition 2.7]{FC}, $x|_K$ extends to a homomorphism $x\in\Hom(A_0, G)$. As commuting with $\Oo_{\kay}$-action is a closed condition, we have $x\in \Hom_{\Oo_{\kay}}(A_0, G)$. By similar reasoning, we have $(x,x)=m$. This shows that the dashed arrow does exist.
\end{proof}

By Proposition \ref{prop: Z*m representability}, for any geometric point $z \in \partial\cS^*$, there is an \'etale neighborhood $U$ of $z$ in $\cS^*$ such that $\cZ^*(m)|_U\rightarrow \cS^*|_U$ is a closed immersion of schemes on every connected component of its domain. 
Let $\hat{U}$ be the completion of $U$ along $\partial S^*$.  Let $\cZ$ be a connected component of $\cZ^*(m)|_{\hat U}$ and $\tilde{\cZ}$ be its first order thickening.
The universal morphism $x:\mathbf{A}_0|_{\cZ}\rightarrow \mathbf{G}|_{\cZ}$ induces a morphism of vector bundles over $\cZ$
\[H_1^\dR(\mathbf{A}_0|_{\cZ})\xrightarrow{x} H_1^\dR(\mathbf G|_{\cZ}), \]
which maps $\Fil(\mathbf{A}_0|_{\cZ})$ to $\Fil(\mathbf G|_{\cZ})$. By Grothendieck-Messing theory for $1$-motives, this morphism admits a canonical extension to a morphism between vector bundles over $\tilde{\cZ}$,
\[H_1^\dR(\mathbf{A}_0|_{\tilde{\cZ}})\xrightarrow{\tilde{x}} H_1^\dR(\mathbf G|_{\tilde{\cZ}}), \]
which determines a morphism (still denoted by $\tilde{x}$)
\begin{equation}\label{}
    \Fil(\mathbf{A}_0|_{\tilde{\cZ}})\xrightarrow{\tilde{x}} \Lie(\mathbf G|_{\tilde{\cZ}}).
\end{equation}

\begin{proposition}\label{prop:deformation cS star}
The morphism \eqref{eq:obstruction map L_A} can be extended to a morphism  between  line bundles:
    \begin{equation}\label{eq:obstruction map L_G app}
 \operatorname{Fil}\left(\mathbf{A}_0\right)|_{\tilde{\cZ}} \stackrel{\tilde{x}}{\rightarrow} L_{\mathbf{G}} |_{\tilde{\cZ}}.
    \end{equation}
The divisor $\cZ$ is the largest closed formal subscheme of $\tilde{\cZ}$ over which \eqref{eq:obstruction map L_G app} is trivial.
\end{proposition}
\begin{proof}
The proposition can be proved similarly as Proposition \ref{prop:deformation C Phi star}.    
\end{proof}

Now we consider a geometric point $z$ on $\cB_\Phi$. We can choose an \'etale neighborhood $X^{(z)}$ of $z$ in $\cC^*_\Phi$ together with an automorphism $\gamma:\Spec \hat{R}_z\rightarrow \Spec \hat{R}_z$ where $\hat{R}_z$ is the completed local ring of $z$ in $\cC^*_\Phi$ such that Proposition \ref{prop:good algebraic neighborhood} holds. In the following discussion, the morphism $\Spec \hat{R}_z \rightarrow \cS^*$ is meant to be the composition $\Spec \hat{R}_z \rightarrow X^{(z)}\rightarrow \cS^*$ in the context of \S\ref{subsec:toroidal compactification}.
\begin{proposition}\label{prop:compare Z* and Z*Phi}
There is an isomorphism $ \cZ^*(m)|_{\Spec \hat{R}_z}\xrightarrow{\sim} \cZ^*_\Phi(m)|_{\Spec \hat{R}_z}$ such that the following diagram is Cartesian.
\begin{equation}\label{eq:compare Z* and Z*Phi}
   \begin{tikzcd}
    \cZ^*(m)|_{\Spec \hat{R}_z} \arrow[r] \arrow[d]  & \Spec \hat{R}_z \ar[d] \arrow[d,"\gamma"]\\
    \cZ^*_\Phi(m)|_{\Spec \hat{R}_z} \arrow[r] & \Spec \hat{R}_z.
\end{tikzcd} 
\end{equation}    
\end{proposition}
\begin{proof}
Since both horizontal morphisms in \eqref{eq:compare Z* and Z*Phi} are finite and unramified, after passing to a further \'etale cover if necessary, we can assume the horizontal morphisms in \eqref{eq:compare Z* and Z*Phi} are closed immersion when restricted to each connected component, and the conclusion of Proposition \ref{prop:good algebraic neighborhood} still holds.

Recall that $\gamma$ induces the identity on $\Spec \hat{R}_z/I_z$ (see Proposition \ref{prop:good algebraic neighborhood}) and $G^\sharp|_{\Spec \hat{R}_z/I_z}=\bfG_{\Spec \hat{R}_z/I_z}$ where $G^\sharp$ is the tautological semi-abelian scheme over $\cB^\Phi$ and $\bfG$ is the universal semi-abelian scheme over $\cS^*$. Hence 
\[\cZ^*(m)|_{\Spec \hat{R}_z/I_z}=\cZ^*_\Phi(m)|_{\Spec \hat{R}_z/I_z} \]
by the definition of both sides. Suppose $x \in \Hom(\bfA_0|_z, G^\sharp|_z)$ such that $(x,x)=m$. Define a functor $\cZ^*_\Phi(x)$ over $\Spec \hat{R}_z$ such that for any $\Spec \hat{R}_z$-scheme $S$, $\cZ^*_\Phi(x)(S)$ is the isomorphism classes of tuples 
\[\left(\bfA_{0,S},\iota_{0,S},\lambda_{0,S},  G_S^\sharp, \iota_S^\sharp, \lambda_S^\sharp, \cF_S^\sharp,\bx \right)\]
where $(\bfA_{0,S},\ldots,\cF_S^\sharp)$ is the pullback of the tautological object over $\cC^*_\Phi$ via $S\rightarrow \Spec \hat{R}_z \rightarrow \cC^*_\Phi$ while $\bx\in \Hom (\bfA_{0,S},G_S^\sharp)$ and restricts to $x \in \Hom(\bfA_0|_z, G^\sharp|_z)$. By Grothendieck-Messing theory,  $\cZ^*_\Phi(x)$ is a subscheme of $\Spec \hat{R}_z$.
By Lemma \ref{lem:rigidity}, we have the following decomposition of $\cZ^*_\Phi(m)|_{\Spec \hat{R}_z}$
\[\cZ^*_\Phi(m)|_{\Spec \hat{R}_z}=\bigsqcup_{\substack{x\in \Hom(\bfA_0|_z, G^\sharp|_z)\\ (x,x)=m }} \cZ^*_\Phi(x).\]
Similarly define a functor $\cZ^*(x)$ over $\Spec \hat{R}_z$ such that for any $\Spec \hat{R}_z$-scheme $S$, $\cZ^*(x)(S)$ is the isomorphism classes of tuples 
\[\left(\bfA_{0,S},\iota_{0,S},\lambda_{0,S},  \bfG_S, \iota_S, \lambda_S, \cF_S,\by \right)\]
where $(\bfA_{0,S},\ldots,\cF_S)$ is the pullback of the universal object over $\cS^*$ via $S\rightarrow \Spec \hat{R}_z \rightarrow \cS^*$ while $\by\in \Hom (\bfA_{0,S},\bfG_S)$ and restricts to $x \in \Hom(\bfA_0|_z, G^\sharp|_z)$.  Then $\cZ^*(x)$ is a subscheme of $\Spec \hat{R}_z$, and we have the following decomposition of $\cZ^*(m)|_{\Spec \hat{R}_z}$
\[\cZ^*(m)|_{\Spec \hat{R}_z}=\bigsqcup_{\substack{x\in \Hom(\bfA_0|_z, G^\sharp|_z)\\ (x,x)=m }} \cZ^*(x).\]

Take $S=\cZ^*_\Phi(x)$.
Recall that by taking the quotient of $(G_S^\sharp, \iota_S^\sharp, \lambda_S^\sharp, \cF_S^\sharp)$ by the image of the period map $\underline{\mathfrak n}\xrightarrow[]{u} G_S^\sharp$, we get a degenerating abelian scheme
$({}^\heart G_S,{}^\heart \iota_S,{}^\heart \lambda_S,{}^\heart \cF_S)$  of type $\Phi$ relative to $(S,S\cap \cB_\Phi,\eta_S)$. Let $\by$ be the composition of $\bx$ with quotient map $G_S^\sharp \rightarrow {}^\heart G_S$, then  $\by\in \operatorname{Hom}_{\Oo_{\kay}}\left(\bfA_{0,S}, {}^\heart G_S\right)$ and $( \by, \by)=m$.
Combining with Proposition \ref{prop:good algebraic neighborhood},
we have a tuple $\gamma^*\left(\bfA_{0,S},\iota_{0,S},\lambda_{0,S}, {}^\heart G_S,{}^\heart \iota_S,{}^\heart \lambda_S,{}^\heart \cF_S,\by\right)$ over $\gamma^*(\cZ_\Phi(x))$ which satisfies the definition of $\cZ^*(x)$.
This shows that $\gamma^*(\cZ^*_\Phi(x))$ is a substack of $\cZ^*(x)$. 

Let $I$ (resp. $J$) be the defining ideal of  $\cZ^*(x)$ (resp.  $\gamma^*(\cZ_\Phi(x))$) in $\Spec \hat{R}_z$. Then $I\subseteq J$. Since $\dim G^\sharp_S=\dim {}^\heart G_S$, we have a canonical isomorphism $\Lie G^\sharp_S\cong \Lie {}^\heart G_S$. 
By Proposition \ref{prop:good algebraic neighborhood}, $\gamma^*(\Lie {}^\heart G_S)\cong\Lie \bfG_S$. Combining with Corollary \ref{cor:L_G on boundary} and \eqref{eq:LGsharp}, we have
\[L_{\bfG}|_{\gamma^*(S)}=\bar \epsilon \Lie \bfG_{\gamma^*(S)} \cong  \gamma^*(\bar\epsilon\Lie {}^\heart G_S)\cong\gamma^*(\bar\epsilon\Lie  G^\sharp_S)=\gamma^*(L_{G^\sharp}|_S). \]
Moreover by the definition of $\by$, the following diagram commutes
\[
  \begin{tikzcd}
    H_1^\dR({\bf A}_0|_{\gamma^*(S)}) \arrow[rd,"\gamma^*(\by)"] \arrow[r,"\gamma^*(\bx)"] & \gamma^*(\Lie  G^\sharp_S)  \arrow[d,"\cong"]\\
     & \Lie \bfG_{\gamma^*(S)}.
\end{tikzcd} 
\]
By Proposition \ref{prop:deformation C Phi star}, the above diagram factors through 
\[
  \begin{tikzcd}
    H_1^\dR({\bf A}_0|_{\gamma^*(S)}) \arrow[rd,"\gamma^*(\by)"] \arrow[r,"\gamma^*(\bx)"] & \gamma^*(L_{G^\sharp}|_S)  \arrow[d,"\cong"]\\
     & L_{\bfG}|_{\gamma^*(S)}.
\end{tikzcd} 
\]
Let $\widetilde{\cZ^*(x)}$ be the first order thickening of $\cZ^*(x)$ in $\Spec \hat{R}_z$. By Grothendieck-Messing theory, the right-downward arrow in the above diagram extends over $\widetilde{\cZ^*(x)}$, hence the whole diagram extends to
\begin{equation}\label{eq:gamma* diagram}
   \begin{tikzcd}
    H_1^\dR({\bf A}_0|_{\widetilde{\cZ^*(x)}}) \arrow[rd,"\gamma^*(\by)"] \arrow[r,"\gamma^*(\bx)"] & \gamma^*(L_{G^\sharp})|_{\widetilde{\cZ^*(x)}}  \arrow[d,"\cong"]\\
     & L_{\bfG}|_{\widetilde{\cZ^*(x)}}.
\end{tikzcd} 
\end{equation}
By Proposition \ref{prop:deformation cS star} and Proposition \ref{prop:deformation C Phi star}, 
$\cZ^*(x)$ (resp. $\cZ^*_\Phi(x)$) is the vanishing locus of $\gamma^*(\by)$ (resp. $\gamma^*(\bx)$) in diagram  \eqref{eq:gamma* diagram}. This shows that
\[\widetilde{\cZ^*(x)}\cap \cZ^*_\Phi(x)=\cZ^*(x),  \]
in other words, $J\equiv I \pmod{I^2}$. In particular $J\subseteq I$. Hence $I=J$, and $\gamma^*(\cZ^*_\Phi(x))=\cZ^*(x)$. The proposition is proved.
\end{proof}

\section{Pullback of algebraic special divisors}\label{sec:pullback I}

\subsection{Morphisms betweeen unitary Shimura varieties}
In this section,  we fix  the hermitian lattice $L$ of signature $(n-1,1)$ as in \eqref{eq:L},  and  a unimodular hermitian $\Oo_{\kay}$-lattice $\Lambda$ with signature $(m,0)$. Clearly, we have a natural map  
\begin{equation} \label{eq:Lattice}
  [[L]] \rightarrow [[L^\dia]] , \quad  M\mapsto M^\dia=M\obot \Lambda
\end{equation}. 
 
We denote by $\cS^\dia$ the integral model of the Shimura variety determined by $L^\dia$ and by $\cS^{\dia,*}$ its toroidal compactification as defined in Section \ref{subsec:toroidal compactification}. We often add the superscript $\dia$ to already defined notations to mark objects associated to $\cS^\dia$. For example, we use $\cZ^{\dia,*}(m)$ (resp. $\cZ^\dia(m)$) to denote special divisors on $\cS^{\dia,*}$ (resp. $\cS^\dia$).

Let $S$ be an $\Oo_{\kay}$-scheme. Consider $(A_0,\iota_0,\lambda_0,A,\iota,\lambda,\cF)\in \cS(S)$. Then  Serre's tensor construction carried out in \cite[Theorem A]{AK} induces a morphism 
\begin{equation}
T_{\Lambda}:     \cM_{(1,0)}\rightarrow \cM_{(m,0)}, \quad 
  (A_0,\iota_0,\lambda_0)      \mapsto (A_0\otimes_{\Oo_{\kay}}\Lambda,\iota_0\otimes_{\Oo_{\kay}}\Lambda,\lambda_0\otimes_{\Oo_{\kay}}\Lambda).
\end{equation}
This induces a morphism
\begin{equation}\label{eq: embed map}
 \varphi_{\Lambda}:    \cS \rightarrow \cS^\dia, 
(A_0,\iota_0,\lambda_0,A,\iota,\lambda,\cF)\mapsto(A_0,\iota_0,\lambda_0,(A,\iota,\lambda)\times T_{\Lambda}(A_0,\iota_0,\lambda_0),\cF \times (\Lie A_0 \otimes_{\Oo_{\kay}} \Lambda)). 
\end{equation}    

\begin{lemma}
    The morphism $\varphi_\Lambda$ is finite and unramified.
\end{lemma}
\begin{proof}
    As in the proof of Proposition \ref{prop: Z*m representability}, the properness of $\varphi_\Lambda$ follows from \cite[Chapter I, Proposition 2.7]{FC} and the valuation criterion, the unramifiedness of $\varphi_\Lambda$ follows from the infinitesimal criterion and Lemma \ref{lem:rigidity}. The quasi-finiteness of $\varphi_\Lambda|_{\cS}$ comes from Faltings and Tate's isogeny theorem. Indeed,   $A\times T_\Lambda(A_0)\cong A'\times T_\Lambda(A_0)$ implies that $A$ and $A'$ are isogenous by looking at the Tate modules and Dieudonn\'e modules. 
\end{proof}

\begin{lemma}  Let the notation be as above.
\begin{enumerate}
\item  The map (\ref{eq:Lattice}) gives a bijection between the genus of $L$  and the genus of  $L^\dia$.

\item  The map  $\varphi_\Lambda $  induces a bijection between $\pi_0(\mathcal S(\C))$ (the set or connected components of $\mathcal S(\C)$ ) and $\pi_0(\mathcal S^\dia(\C))$. 
\end{enumerate}

\end{lemma}
\begin{proof} Let $V=L \otimes_\Z \Q$, $H=U(V)$ and $H_0=\hbox{SU}(V)$. Then  $H_0$ is semi-simple and simply connected, and has thus strong apparoximation property. In  particular, let 
$$K=U(\hat L) =\{  h \in H(\A_f):\,   h L =L\} $$
and  $K_0 = K \cap H_0(\A_f)$. Then 
 $H_0(\A_f)= H_0(\A) K_0$. The exact sequence 
 $$
 1 \rightarrow H_0   \rightarrow  H \rightarrow  \kay^1  \rightarrow 1 
 $$
implies that
$$
H(\Q) \backslash H(\A_f)/ K \cong  \kay^1 \backslash \kay_f^1/\det K.
$$
Notice that $\det K = \hat{\OO}_{\kay}^\times\cap \kay_f^1 = \hat{\OO}_{\kay}^1$ does not depend on $L$. 
So we have 
$$
 \begin{tikzcd}
   H(\Q) \backslash H(\A_f)/ K  \arrow[r] \arrow[r, "\cong"] \arrow[d]  & \kay^1 \backslash \kay_f^1/\hat{\OO}_{\kay}^1 \ar[d] \arrow[d,"="]\\
     H^\dia(\Q) \backslash H^\dia(\A_f)/ K^\dia   \arrow[r] \arrow[r, "\cong"]  & \kay^1 \backslash \kay_f^1/\hat{\OO}_{\kay}^1.
\end{tikzcd} 
$$
So the left hand side arrow is also a  bijection.
Since
$$
H(\Q) \backslash H(\A_f)/ K \cong [[L]],\quad   h \mapsto hL 
$$
the above bijection  proves (1). 

Claim  (2) follows from  (1) easily.  Indeed, recall from  Section \ref{sec: unitary shimura variety}  that $\mathcal S$  is the integral model of the Shimura variety associated to 
$$
G =\{ (g_0, g) \in \hbox{GU}(W_0) \times \hbox{GU}(W) :\, \det g_0 = \nu(g)\} 
  \cong  \hbox{GU}(W_0) \times   U(V) ,
  $$
where  $W_0 =\mathfrak a_0 \otimes_\Z \Q$, $W= \mathfrak a \otimes_\Z \Q$, and $V=L \otimes_\Z \Q$ with $L =\hbox{Hom}(\mathfrak a_0, \mathfrak a)$.  The isomorphism map on  $U(V)$ is give by 
$(g_0, g) (f) (w_0) = g f(g_0^{-1} w_0)$. 
So  
$$\pi_0(\mathcal S(\C)) = \hbox{CL}(\kay) \times \pi_0(\hbox{Sh}(L))$$
where $\hbox{Sh}(L) $ is the  Shimura variety associated to $(H, K)$.  It is easy to see 
$$
\hbox{Sh}(L)(\C) =H(\Q) \backslash \mathcal D \times H(\A_f) /K   \cong \cup_{M \in [[L]]}  \Gamma_M \backslash \mathcal D
$$
with  
$$
\Gamma_M = \{ h\in H(\Q):\,  h M =M\}
=  (g  K g^{-1}) \cap H(\Q)
$$
where $g \in H(\A_f)$ and  with $g L = M$.  In  particular,  $\pi_0(\hbox{Sh}(L)) = [[L]]$. Now (2) is clear.
\end{proof}

\begin{remark} The above lemma does not imply `cancellation law' on unitary Hermitian lattices in general as we require the starting lattices  are in the same genus in the lemma.  Indeed it is not hard to find  counter example to the cancellation law. Let $L_1$ and $L_2$ be two unimodular Hermitian  lattices with gram matrices 
$$
A_1= \begin{pmatrix}
    0  &1\\ 1 &0
\end{pmatrix}
\hbox{ and }
A_1= \begin{pmatrix}
    1  &1\\ 1 &0
\end{pmatrix}
$$
 It is easy to check that $B_1 =\hbox{Diag}(A_1, 1)$ and 
 $B_2= \hbox{Diag}(A_2, 1)$  are equaivalent (i.e.,  $L_1 \obot \Lambda  \cong  L_2 \obot \Lambda$):
 $
 {}^t P B_1 \bar  P= B_2
 $
 with 
 $$
 P=\begin{pmatrix}
     0 &1 &1
     \\
     1 &0 &0
     \\
     1 &0 &-1
 \end{pmatrix}
 $$
 On the other hand,  $L_1$ and $L_2$ are not (Hermitian) equivalent over $\Z_2$ if $\kay/\Q$ is ramified  at $2$, and of course are  not equivalently over $\Z$. By the way, when  $\kay/Q$ is unramified at $2$,  $L_1$ and $L_2$ are  equivalent.
\end{remark}

\subsection{Extension of the morphism to the boundary}
We first recall an extension criterion from \cite{Lan}. Let $X$ be a noetherian normal $\Oo_{\kay}$-scheme and $Z$ is a Cartier divisor of $X$. Let $U=X\setminus Z$. Suppose $(G,\iota,\lambda,\cF)$ is a degenerating abelian scheme of signature $(n-1,1)$ relative to $(X,Z,U)$ in the sense of Definition \ref{def:degenerating abelian scheme} such that there is an $\Oo_{\kay}$-morphism $\varphi:U\rightarrow \cS$ and $(A_0,\iota_0,\lambda_0,G,\iota,\lambda,\cF)|_U$ is the pullback of the universal family over $\cS$. Let $s$ be a geometric point of $Z$, and an $\Oo_{\kay}$-morphism $\phi:\Spec V\rightarrow X$ centered at $s$ where $V$ is a complete discrete valuation ring with fraction field $K$ such that $\phi|_{\Spec K}$ factor through $U$. Then $(A_0,\iota_0,\lambda_0,G,\iota,\lambda,\cF)|_{\Spec V}$ determines an element $\xi\in\mathrm{DEG}^\Phi_{(n-1,1)}(\Spec V,s,\Spec K)$ for a cusp label $\Phi=(\mathfrak{m}\subset M)\in \mathrm{Cusp}([[L]])$ determined by setting $\mathfrak m$ to be the toric part of $G_s$. Define $\mathfrak n=\mathfrak m^\vee$ and
\begin{equation}\label{eq:definition of S_Phi}
\Sym_\Phi=\Sym^2_\Z(\mathfrak{n})/\langle (x\mu)\otimes \nu-\mu\otimes (\bar x \nu):x\in \Oo_{\kay}, \mu,\nu \in \mathfrak{n} \rangle.  
\end{equation}
By Theorem \ref{thm:DD DEG equivalence}, the element $\xi\in\mathrm{DEG}^\Phi_{(n-1,1)}(\Spec V,s,\Spec K)$ determines a degeneration data in
\[\mathrm{DD}^\Phi_{(n-1,1)}(\Spec V,s, \Spec K),\]
which in turn determines a map $B_\Phi:\Sym_\Phi\rightarrow K^*$, see \cite[\S 6.3.1.1]{Lan}. Let $v$ be the canonical discrete valuation of $K$. 
\begin{lemma}\label{lem:extension criterion}
    The morphism $\varphi:U\rightarrow \cS$ extends to a (necessarily unique) morphism $\varphi:X\rightarrow \cS^*$ such that $(A_0,\iota_0,\lambda_0,G,\iota,\lambda,\cF)$ is the pullback of the universal family over $\cS^*$ under $\varphi$ if and only if for any geometric point $s$ of $Z$ and every $\Oo_{\kay}$-morphism $\phi:\Spec V\rightarrow X$ centered at $s$ as above, $v\circ B_\Phi$ is contained in the unique positive semi-definite cone of $\Sym_\Phi^\vee\otimes_\Z \R$.
\end{lemma}
\begin{proof}
The lemma follows the same proof as \cite[Theorem 6.4.1.1(6)]{Lan}.
\end{proof}

Given a proper cusp label $\Phi=(\mathfrak{m}\subset M)\in \mathrm{Cusp}([[L]])$, we may produce a new cusp label 
\[\Phi^\dia\coloneqq(\mathfrak{m},M\obot \Lambda) \in \mathrm{Cusp}(M\obot \Lambda)\subset\mathrm{Cusp}([[L^\dia]]).\]

\begin{proposition}\label{prop:pullback boundary}
The morphism $\varphi_{\Lambda}:\cS \rightarrow \cS^\dia$ extends to a morphism  $\varphi_{\Lambda}:\cS^* \rightarrow \cS^{\dia,*}$ (still denoted by the same notation) which is finite and unramified. It has the following properties. 
\begin{enumerate}
    \item The  pullback of the universal family over $\cS^{\dia,*}$ under $\varphi_\Lambda$ is
    \begin{equation}\label{eq:Serre tensor semi abelian scheme}
 (\mathbf{A}_0,\iota_0,\lambda_0,(\mathbf{G},\iota,\lambda)\times T_{\Lambda}(\mathbf{A}_0,\iota_0,\lambda_0),\cF \times (\Lie \mathbf{A}_0 \otimes_{\Oo_{\kay}} \Lambda)) 
\end{equation}
    \item We have the equality of line bundles over $\cS^*$
\begin{equation}\label{eq: pullback of Omega sec 5}
\varphi_\Lambda^*(\Omega^\dia) \cong \Omega.
\end{equation}
 \item For a cusp label $\Phi'$ of $\cS^\dia$, we have the following equation of Cartier divisors.
 \begin{equation}\label{eq:pullback of boundary divisor}
     \varphi_{\Lambda}^*(\cS^{\dia,*}(\Phi'))=
     \begin{cases}
    \cS^*(\Phi) & \text{ if $\Phi'=\Phi^\dia$ for some cusp label $\Phi$ of $\cS^*$},\\
     0 & \text{ otherwise}.
     \end{cases}
  \end{equation}
     \end{enumerate}
\end{proposition}
\begin{proof}
To prove the proposition, we proceed in several steps.

\noindent {\bf Step 1}: Construct an analogue of $\varphi_\Lambda$ for the formal boundary charts $\cC^*_\Phi$. Let $L_\Phi= \mathfrak{m}^\bot/\mathfrak{m}$ be as in  \eqref{eq: normal decom}, then $L_{\Phi^\dia}=L_\Phi\oplus\Lambda$. Recall from Definition \ref{def:degeneration data} that we have the tautological degeneration data $(B, \kappa, \psi, c, c^\vee, \tau)$ (resp. $(B^\dia, \kappa^\dia, \psi^\dia, c^\dia, c^{\dia,\vee}, \tau^\dia)$) over $(\cC^*_\Phi,\cB_\Phi,\cC_\Phi)$ (resp. $(\cC^*_{\Phi^\dia},\cB_{\Phi^\dia},\cC_{\Phi^\dia})$). Then $B^\dia=B\times (A_0\otimes_{\Oo_{\kay}}\Lambda)$.
By \eqref{eq:B_Phi}, we have 
\[\cB_\Phi\cong E\otimes_{\Oo_{\kay}} L_\Phi,\  \cB_{\Phi^\dia}\cong E\otimes_{\Oo_{\kay}} L_\Phi^\dia, \]
where $E=\underline{\Hom}_{\mathcal O_{\kay}}(\mathfrak n,\mathbf{A}_0)$, $\mathbf{A}_0$ is the universal elliptic curve over $\cM_{(1,0)}$, and $\otimes$ means Serre tensor. We then define a closed immersion
\[\cB_\Phi\hookrightarrow \cB_{\Phi^\dia}, \quad  \sum z_i\otimes s_i \mapsto  \sum z_i\otimes (s_i,0),\]
where $z_i\in E(S)$ for any $\cB_\Phi$-scheme $S$. From the definition of the Poincar\'e bundle $P$ (resp. $P^\dia$) over $B\times B^\vee$ (resp. $B^\dia\times B^{\dia,\vee}$), we have
\begin{equation}\label{eq:pullback Poincare bundle}
   P^\dia|_{B\times B^\vee}=P. 
\end{equation}
Recall from \S \ref{subsec:toroidal compactification} that $\cC_{\Phi^\dia}^*$ (resp. $\cC_{\Phi}^*$) is the total space of the line bundle $\cL_{\Phi^\dia}^{-1}$ (resp. $\cL_\Phi^{-1}$).
By \eqref{eq:pullback Poincare bundle}, we have
\[ \cL_{\Phi^\dia}|_{\cB_\Phi}=\cL_\Phi. \]
Hence we get a natural closed immersion
\begin{equation}
\varphi_{\Lambda,\Phi}:\cC_{\Phi}^*\hookrightarrow \cC_{\Phi^\dia}^*.   
\end{equation}
The closed immersions $\cB_\Phi\hookrightarrow \cC_\Phi^*$ and $\cB_{\Phi^\dia}\hookrightarrow \cC_{\Phi^\dia}^*$ defined by zero sections make the following diagram Cartesian.
\begin{equation}\label{eq:Cartesian diagram}
    \begin{tikzcd}
\cB_\Phi \arrow[r] \arrow[d]  \arrow[dr, phantom, "\square"]
& \cB_{\Phi^\dia} \arrow[d] \\
\cC_\Phi^* \arrow[r]
& \cC_{\Phi^\dia}^*
\end{tikzcd}
\end{equation}
In particular we have the following equality of Cartier divisors.
\begin{equation}\label{eq:pullback cB}
    \varphi_{\Lambda,\Phi}^*(\cB_{\Phi^\dia})=\cB_\Phi. 
\end{equation}
By the above construction, $c^\dia$  when restricted on $\cC^*_\Phi$ is the composition $\mathfrak{n}\xrightarrow[]{c} B \hookrightarrow B\times (\bfA_0\otimes_{\Oo_{\kay}} \Lambda)$ where $ B \hookrightarrow B\times (\bfA_0\otimes_{\Oo_{\kay}} \Lambda)$ is the obvious closed immersion $z\mapsto (z,0)$. Similar statement is true for $c^{\dia,\vee}$. Hence we have
\begin{equation}\label{eq:restriction of c dia P}
    ((c^{\dia,\vee}\times c^\dia )^*(\mathcal{P}^\dia))|_{B\times B^\vee}=(c^\vee\times c )^*(\mathcal{P}).
\end{equation}

\noindent {\bf Step 2}: Construct an analogue of $\varphi_\Lambda$ for the \'etale neighborhood $X^{(z)}$ for a geometric point $z$ of $\cC^*_\Phi$ (hence also an \'etale neighborhood of $\cS^*$) as in Proposition \ref{prop:good algebraic neighborhood}, and glue them to get the extension we want to construct.
Let 
\begin{equation}\label{eq:tuple before serre tensor}(\bfA_0,\iota_0,\lambda_0,\mathbf{G}^{(z)},\iota^{(z)},\lambda^{(z)},\cF^{(z)})
\end{equation}
be the universal semi-abelian schemes over $X^{(z)}$. Consider the tuple 
\begin{equation}\label{eq:serre tuple}
    (\mathbf{A}_0,\iota_0,\lambda_0,(\mathbf{G}^{(z)},\iota^{(z)},\lambda^{(z)})\times T_{\Lambda}(\bfA_0,\iota_0,\lambda_0),\cF \times (\Lie \bfA_0 \otimes_{\Oo_{\kay}} \Lambda)).
\end{equation}
Restricting the tuple to the interior part $U^{(z)}$ of $X^{(z)}$ defines a morphism $\varphi_{\Lambda}|_{ U^{(z)}}:U^{(z)}\rightarrow \cS^{\dia}$ which is the composition $U^{(z)}\rightarrow \cS\xrightarrow{\varphi_\Lambda} \cS^\dia$.

Let $s$ be any geometric point of $X^{(z)}\setminus U^{(z)}$.
Let $V$ be a complete discrete valuation ring with fraction field $K$ and $\phi:\Spec V \rightarrow X^{(z)}$ be a morphism centered at $s$ such that $\phi|_{\Spec K}$ factor through $\cS$. By \eqref{eq:definition of S_Phi}, we have
\begin{equation}\label{eq:identify SPhi and SPhidia}
\Sym_\Phi=\Sym_{\Phi^\dia}=\Sym^2_\Z(\mathfrak{n})/\langle (x\mu)\otimes \nu-\mu\otimes (\bar x \nu):x\in \Oo_{\kay}, \mu,\nu \in \mathfrak{n} \rangle.  
\end{equation}
The pullback of the universal family \eqref{eq:tuple before serre tensor} to $\Spec V$ defines a degenerating abelian scheme, hence a degeneration data (by Theorem \ref{thm:DD DEG equivalence}) in
\[\mathrm{DD}^\Phi_{(n-1,1)}(\Spec V,s, \Spec K), \]
hence defines a map $B_\Phi:\Sym_\Phi\rightarrow K^*$.
Let $v$ be the canonical discrete valuation of $K$, then Lemma \ref{lem:extension criterion} implies that $v\circ B_\Phi$ is contained in the unique positive semi-definite cone of $\Sym_\Phi^\vee\otimes_\Z \R$.  Similarly  the pullback of \eqref{eq:serre tuple} to $\Spec V$ defines a degeneration data in
\[\mathrm{DD}^{\Phi^\dia}_{(n+m-1,1)}(\Spec V,s, \Spec K), \]
hence also defines a map $B_{\Phi^\dia}:\Sym_{\Phi^\dia}\rightarrow K^*$. Unraveling the construction of $B_\Phi$ and $B_{\Phi^\dia}$ in \cite{Lan} shows that the two agree under the identification \eqref{eq:identify SPhi and SPhidia}. Applying Lemma \ref{lem:extension criterion} again, we know that the  morphism $\varphi_{\Lambda}|_{U^{(z)}}:U^{(z)} \rightarrow \cS^\dia$ extends to a morphism  $\varphi_{\Lambda}|_{X^{(z)}}:X^{(z)} \rightarrow \cS^{\dia,*}$ and \eqref{eq:serre tuple} is the pullback of the universal family over $\cS^{\dia,*}$ under $\varphi_{\Lambda}|_{X^{(z)}}$. We claim that these $\varphi_{\Lambda}|_{X^(z)}$ together with $\varphi_\Lambda:\cS\rightarrow\cS^{\dia}$ glue to a morphism from $\varphi_\Lambda:\cS^*\rightarrow\cS^{\dia,*}$. This is true because $\cS^*$ is the quotient of an \'etale equivalence relation \eqref{eq:glue} which ``identifies" isomorphic semi-abelian schemes and if two tuples $\xi=(A_0,\ldots,G,\ldots)\in X^{(z)}(S)$ and $\xi'=(A'_0,\ldots,G',\ldots)\in X^{(z')}(S)$ are isomorphic, then so are $\varphi_{\Lambda}|_{X^{(z)}}(\xi)$ and $\varphi_{\Lambda}|_{X^{(z')}}(\xi')$. The extension $\varphi_\Lambda$ thus defined is independent of the choice of $X^{(z)}$s used in the gluing process.
Obviously the we have $\varphi_\Lambda(\cS^*(\Phi))\subset \cS^{\dia,*}(\Phi^\dia)$.

\noindent {\bf Step 3}: Compare the morphisms in Step 1 and 2, and conclude the proof.
Let $z$ be a geometric point of $\cB_\Phi$ (also considered as a geometric point of $\cC^*_\Phi$ and $\cS^*$). Let $z^\dia$ be a geometric point of $\cB_{\Phi^\dia}$ whose image under the morphism $\cB_{\Phi^\dia}\rightarrow \cS^{\dia,*}$ is $\varphi_\Lambda(z)$. Choose \'etale neighborhoods $X^{(z^\dia)}\rightarrow\cS^{\dia,*}$ of $z^\dia$ and $X^{(z)}\rightarrow\cS^*$ of $z$ respectively as in Proposition \ref{prop:good algebraic neighborhood}. We claim that 
\begin{equation}\label{eq:varphi z z^dia}
    \varphi_{\Lambda,\Phi}(z)=z^\dia.
\end{equation}
Assuming the claim and using the notations of Proposition \ref{prop:good algebraic neighborhood}, we can define a morphism $\varphi_\Lambda^{(z)}:\Spec \hat{R}_z \rightarrow \hat{R}_{z^\dia}$ by
\[\varphi_\Lambda^{(z)}=(\gamma^\dia)^{-1}\circ \varphi_{\Lambda,\Phi}|_{\Spec \hat R_z} \circ \gamma, \]
where $\varphi_{\Lambda,\Phi}$ is the morphism defined in Step 1. We claim that the following diagram commutes.
\begin{equation}\label{eq:approximation commutes}
    \begin{tikzcd}
\Spec \hat{R}_z \arrow[r,"\varphi_\Lambda^{(z)}"] \arrow[d,"i"]
& \Spec \hat{R}_{z^\dia} \arrow[d,"j"] \\
 \cS^* \arrow[r,"\varphi_{\Lambda}"]
& \cS^{\dia,*}.
\end{tikzcd}
\end{equation}
It follows that $\varphi_\Lambda$ is finite and unramified when restricted to $(\cS^*)^\wedge_{\cS^*(\Phi)}$ as $\varphi_{\Lambda,\Phi}$ is.

To prove the commutativity of \eqref{eq:approximation commutes}, let $({}^\heart G,{}^\heart \iota,{}^\heart \lambda,{}^\heart\cF)$ (resp. $({}^\heart G^\dia,{}^\heart \iota^\dia,{}^\heart \lambda^\dia,{}^\heart\cF^\dia)$) be the Mumford family over $\Spec \hat{R}_z$ (resp. $\Spec \hat{R}_{z^\dia}$) associated to the tautological data $(B, \kappa, \psi, c, c^\vee, \tau)$ (resp. $(B^\dia, \kappa^\dia, \psi^\dia, c^\dia, c^{\dia,\vee}, \tau^\dia)$). By \eqref{eq:restriction of c dia P}, we have 
\begin{equation}\label{eq:pullback of Mumford famliy}
    (\varphi_{\Lambda,\Phi})^* ({}^\heart G^\dia,{}^\heart \iota^\dia,{}^\heart \lambda^\dia,{}^\heart\cF^\dia)=({}^\heart G,{}^\heart \iota,{}^\heart \lambda,{}^\heart\cF)\times T_\Lambda(\bfA_0,\iota_0,\lambda_0),
\end{equation}
By \eqref{eq:pullback of Mumford famliy} and Proposition \ref{prop:good algebraic neighborhood}, we know
\[ (j\circ \varphi_{\Lambda}^{(z)})^* (\bfG^\dia,\iota^\dia, \lambda^\dia,\cF^\dia)=(G^{(z)},\iota^{(z)},\lambda^{(z)}, \cF^{(z)})\times T_\Lambda(\bfA_0,\iota_0,\lambda_0). \]
where $(\bfG^\dia,\iota^\dia, \lambda^\dia,\cF^\dia)$ is the universal object over $\cS^{\dia,*}$.
Here we use the fact that $T_\Lambda(\bfA_0,\iota_0,\lambda_0)$ is defined over $\Oo_{\kay}$, so is invariant under $\gamma$. On the other hand, using the definition of $\varphi_\Lambda$, we know
\[ (\varphi_\Lambda\circ i)^* ({}^\heart G^\dia,{}^\heart \iota^\dia,{}^\heart \lambda^\dia,{}^\heart\cF^\dia)=(G^{(z)},\iota^{(z)},\lambda^{(z)}, \cF^{(z)})\times T_\Lambda(\bfA_0,\iota_0,\lambda_0). \]
By the universal property of $\widehat{\cC}^*_\Phi$ for degenerating abelian schemes in the category
\[\mathrm{DEG}_{(n-1,1)}^\Phi(\Spec \hat{R_z},\Spec (\hat{R}_z/I_z),\hat{\eta}_z) \]
relative to $(\Spec \hat{R_z},\Spec (\hat{R}_z/I_z),\hat{\eta}_z)$, diagram \eqref{eq:approximation commutes} commutes. The same argument proves claim \eqref{eq:varphi z z^dia}.

Equation \eqref{eq: pullback of Omega sec 5} is obvious from the definition of $\Omega$.
By \eqref{eq:pullback cB}, \eqref{eq:approximation commutes} and the fact that $\gamma$ and $\gamma^\dia$ are identities when restricted to boundaries, 
we have equation \eqref{eq:pullback of boundary divisor}.
The proposition is proved.
\end{proof}

\subsection{Pullback of line bundles}
In the following discussion we use the correspondence between line bundles and Cartier divisors. For a Cartier divisor $D$ on $\cS^{\dia,*}$, we use $D|_{\cS^*}$ to denote $\varphi_\Lambda^*(D)$, the pullback of the corresponding line bundle via $\varphi_\Lambda$. 
\begin{proposition}\label{prop:pullback inside}
Assume $m\in \Z$. Then in  $\hbox{Pic}(\mathcal S^*)$
\begin{equation}\label{eq: pullback}
    \cZ^{\dia,*}(m)|_{\cS^*}\cong \displaystyle (\bigotimes_{\substack{m_1+m_2=m\\ m_1\neq 0}}\cZ^*(m_1)^{\otimes r_{\Lambda}(m_2)}) \bigotimes (\Omega^{-1})^{\otimes r_{\Lambda}(m)},
\end{equation}
where  $r_{\Lambda}(m)$ is the cardinality of
\begin{align*}
    R_{\Lambda}(m)\coloneqq \{\lambda\in \Lambda\mid (\lambda,\lambda)=m\}.
\end{align*}
\end{proposition}
\begin{proof}
The proof is similar to that of \cite[Proposition 6.6.3]{HP}.
If $m<0$, then both sides of \eqref{eq: pullback} are trivial line bundles. If $m=0$, then   \eqref{eq: pullback} is simply \eqref{eq: pullback of Omega} by  the definition of $\cZ^*(0)$ in \S \ref{subsec:constant term}. For any scheme $S\rightarrow \cS$,  let $\mathbf G_S$ be the pullback of the universal semi-abelian scheme over $\cS$, and $\mathbf G_S^\dia$ be the pullback of the universal semi-abelian scheme over $\cS^\dia$ via the morphism $S\rightarrow \cS\xrightarrow[]{\varphi_\Lambda} \cS^\dia$. Define the following $\Oo_{\kay}$-modules of special homomorphisms 
\[V(\mathbf G_S)=\Hom_{\Oo_{\kay}}(\mathbf{A}_{0,S},\mathbf G_S), \quad V(\mathbf G_S^\dia)=\Hom_{\Oo_{\kay}}(\mathbf{A}_{0,S},\mathbf G_S^\dia).\]
By the definition of $\varphi_\Lambda$, we know $\mathbf G_S^\dia=\mathbf G_S\oplus (\mathbf{A}_0\otimes_{\Oo_{\kay}}\Lambda)_S$. By \cite[Proposition 2]{AK}, we know that $\Hom_{\Oo_{\kay}}(\mathbf{A}_0,\mathbf{A}_0\otimes_{\Oo_{\kay}} \Lambda)\cong \Lambda$. Hence we have
\begin{equation}\label{eq:decomposition of speical homomorphisms}
    V(\mathbf G_S^\dia)\cong V(\mathbf G_S)\oplus \Lambda.
\end{equation}
Equation \eqref{eq:decomposition of speical homomorphisms} induces the following isomorphism of $\cS^*$-stacks.
\begin{equation}\label{eq:decomposition of restriction of Z dia}
    \cZ^{\dia,*}(m)|_{\cS^*}\cong \bigsqcup_{\substack{m_1+m_2=m\\ m_1\neq 0\\ \lambda\in R_{\Lambda}(m_2) }} \cZ^*(m_1) \sqcup \bigsqcup_{\lambda\in R_\Lambda(m)} \cS^*.
\end{equation}
There is a corresponding canonical decomposition of $\cS^{\dia,*}$-stack
\begin{equation}
    \cZ^{\dia,*}(m)=\cZ^\dia_{\mathrm{prop}}\sqcup \cZ^\dia_{\mathrm{im}},
\end{equation}
such that each connected component of $\cZ^\dia_{\mathrm{prop}}$ intersects the image of $\varphi_\Lambda$ properly and contributes to the first term on the right hand side of \eqref{eq:decomposition of restriction of Z dia}, while each connected component of $\cZ^\dia_{\mathrm{im}}$ contains (intersects ``improperly" with) a component of the image of $\varphi_\Lambda$ and contributes to the second term on the right hand side of \eqref{eq:decomposition of restriction of Z dia}. Notice that the first term on the right hand side of \eqref{eq:decomposition of restriction of Z dia} corresponds to exactly the first term of on the right hand side of \eqref{eq: pullback}. To prove the proposition, it suffices to prove the following lemma.
\end{proof}

\begin{lemma}\label{lem:restriction degenerate bundle}
Let $\cZ^\dia$ be a connected component of $\cZ^\dia_{\mathrm{im}}$ and $U^\dia$  a connected \'etale neighborhood of $\cS^{\dia,*}$ such that the morphism $\cZ^\dia|_{U^\dia}\rightarrow U^\dia$ is a closed immersion when restricted on each connected component of $\cZ^\dia|_{U^\dia}$.
Let $U$ be a connected \'etale neighborhood of $\cS^*$ whose image under $\varphi_\Lambda$ is contained in $\cZ^\dia|_{U^\dia}$. 
Then we have canonical isomorphism of line bundles
\[\cZ^\dia_{\mathrm{prop}}=\displaystyle\bigotimes_{\substack{m_1+m_2=m\\ m_1\neq 0}}\cZ^*(m_1)|_U^{\otimes r_{\Lambda}(m_2)},\]
and 
    \begin{align}\label{eq: adjunction}
    \Oo(\cZ^\dia)|_U\cong \Omega^{-1}|_U.
    \end{align}
\end{lemma}
\begin{proof}
The proof is identical to that of \cite[Theorem 7.10]{BHY} or \cite[Lemma 6.6.4]{HP}, given Proposition \ref{prop:deformation of Z x interior} and Proposition \ref{prop:deformation cS star}. We leave it to the reader.
\end{proof}

For $v>0$ and $m\geq 0$, define
\begin{equation}\label{eq:Bmv}
    \mathcal{B}(m,v)\coloneqq\frac{1}{4\pi v}\sum_{\Phi\in \mathrm{Cusp}([[L]])}\#\left\{x \in L_\Phi:( x, x)=m\right\} \cdot \mathcal{S}^*(\Phi)\in \mathrm{CH}^1_\C(\mathcal{S}^*),
\end{equation}
where $L_\Phi$ is the positive definite lattice in \eqref{eq: normal decom}. Then we define the total special divisor to be the following element in $\mathrm{CH}^1_\C(\mathcal{S}^*)$. 
\begin{equation}\label{eq:Z tot}
    \cZ^{\tot}(m,v)\coloneqq \begin{cases}
        \cZ^*(m)+\mathcal{B}(m,v)  & \text{ if } m>0\\
       \Omega^{-1}+\mathcal{B}(0,v) & \text{ if } m=0\\
       0 & \text{ if } m<0.
    \end{cases}
\end{equation}

\begin{theorem}\label{thm:pullback total}
Assume $m\in \Z$ and $v>0$. Then we have the following identity in $\mathrm{CH}^1(\mathcal{S}^*)$.
\begin{align}\label{eq: pullback II}
    \varphi_{\Lambda}^*(\cZ^{\mathrm{tot},\dia}(m,v))=\sum_{m_1+m_2=m} r_{\Lambda}(m_2)\cdot \cZ^{\mathrm{tot}}(m_1,v).
\end{align}
Notice here that the sum is a finite sum and is over $m_1, m_2 \ge 0$ by (\ref{eq:Z tot}).
\end{theorem}
\begin{proof}
By Proposition \ref{prop:pullback boundary}, we have 
\begin{align*}
   &\varphi_{\Lambda}^*(\mathcal{B}^\dia(m,v)) \\
   =& \frac{1}{4\pi v} \sum_{\Phi^\dia\in \mathrm{Cusp}([[L^\dia]])}  \#\left\{x \in L_{\Phi^\dia}:( x, x)=m\right\} \cdot  \varphi_{\Lambda}^*(\mathcal{S}^{\dia,*}(\Phi^\dia))\\
   =&  \frac{1}{4\pi v} \sum_{\Phi\in \mathrm{Cusp}([[L]])}  \#\left\{x \in L_\Phi\obot \Lambda:( x, x)=m\right\} \cdot  \mathcal{S}^*(\Phi) \\
   = &\frac{1}{4\pi v} \sum_{\Phi\in \mathrm{Cusp}([[L]])} 
 \sum_{m_1+m_2=m}  \#\left\{x \in  \Lambda:( x, x)=m_2\right\} \cdot \#\left\{x \in  L_\Phi:( x, x)=m_1\right\} \cdot\mathcal{S}^*(\Phi)\\
 =&\sum_{m_1+m_2=m}  \#\left\{x \in  \Lambda:( x, x)=m_2\right\} \cdot  \frac{1}{4\pi v} \sum_{\Phi\in \mathrm{Cusp}([[L]])} \#\left\{x \in  L_\Phi:( x, x)=m_1\right\} \cdot\mathcal{S}^*(\Phi)\\
 =& \sum_{m_1+m_2=m} r_{\Lambda}(m_2)\cdot \cB(m_1,v).
\end{align*}
Combining this with Proposition \ref{prop:pullback inside}, the theorem is proved.
\end{proof}

\section{Pullback of arithmetic special divisors} \label{sect:PullBack II}

\subsection{Arithmetic Chow groups}
In this subsection, we review the theory of arithmetic Chow groups formed from cycles with Green currents with certain log-log singularities along a fixed normal crossing divisor developed by \cite{BKK}. We follow \cite[\S 3.1]{Ho1} closely.

\begin{definition}
Let $M^*$ be a compact complex manifold of dimension $n-1, \partial M^* \subset$ $M^*$ be a smooth codimension one submanifold, $M=M^* \backslash \partial M^*$, and $z_0 \in \partial M^*$. Suppose that  on some open neighborhood $V \subset M^*$ of $z_0$, there are coordinates $q, u_1, \ldots, u_{n-2}$ such that $\partial M^*$ is the vanishing locus of $q$. After possibily shrinking $V$, we may always assume that $\log \left|q^{-1}\right|>1$ on $V$. Then we call the open set $V$ and its coordinates   \emph{adapted} to $\partial M^*$.
\end{definition}

\begin{definition}
    Suppose $f$ is a $C^{\infty}$ function on an open subset $U \subset M$. We say that $f$ has $\log$-log growth along $\partial M^*$ if around any point of $\partial M^*$ there is an open neighborhood $V \subset M^*$ and coordinates $q, u_1, \ldots, u_{n-2}$ adapted to $\partial M^*$ such that
\begin{align}\label{eq: log}
    f=O(\log \log \left|q^{-1}\right|)
\end{align}
on $U \cap V$. A smooth differential form $\omega$ on $U$ has log-log growth along $\partial M^*$ if around any point of $\partial M^*$ there is an open neighborhood $V \subset M^*$ and coordinates $q, u_1, \ldots, u_{n-2}$ adapted to $\partial M^*$ such that $\left.\omega\right|_{U \cap V}$ lies in the subring (of the ring of all smooth forms on $U \cap V$ ) generated by
$$
\frac{d q}{q \log |q|}, \frac{d \bar{q}}{\bar{q} \log |q|}, d u_1, \ldots, d u_{n-2}, d \bar{u}_1, \ldots, d \bar{u}_{n-2},
$$
and functions satisfying \eqref{eq: log}.
\end{definition}
The work of \cite{BKK} about arithmetic Chow group with log-log singularity along a normal crossing divisor is for a flat, regular and proper $\Oo_{\kay}$-scheme of finite type. Since $\cS^*$ is in fact a Deligne-Mumford stack, we need to adapt this theory to Deligne-Mumford stack accordingly following \cite{Ho1}.

In particular, we extend the notion of log-log growth to the orbifold fibers of $\mathcal{S}^*$ in the following way.
We can write  $\mathcal{S}^*(\mathbb{C})$ as the quotient of a complex manifold $M^*$ by the action of a finite group $H$. As a result, we can regard
$$\partial \mathcal{S}^*(\mathbb{C}) \rightarrow \mathcal{S}^*(\mathbb{C}) \leftarrow \mathcal{S}(\mathbb{C})$$
 as the quotients of $H$-invariant morphisms of complex manifolds
$$
\partial M^* \rightarrow M^* \leftarrow M
$$
with $H$-actions. Then we say a smooth form on $\cS^*(\mathbb{C})$ has log-log growth along the boundary $\partial \cS^*(\mathbb{C})$ if the corresponding $H$-invariant form on $M$ obtained via pullback has $\log$-log growth along $\partial M^*$.

Let   $\mathcal{Z}=\sum m_i \mathcal{Z}_i$ be a finite $\mathbb{C}$-linear combination of pairwise distinct irreducible closed substacks of codimension one where $m_i\in \C$, which is a divisor on $\cS^*$ with complex coefficients.
\begin{definition}\label{def: green function} 
A Green function for $\mathcal{Z}$ consists of a smooth  function $\operatorname{Gr}(\mathcal{Z}, \cdot)$ on $\mathcal{S}(\mathbb{C}) \backslash \mathcal{Z}(\mathbb{C})$ satisfying the following properties:
\begin{enumerate}
    \item For every point of $\cS^*(\mathbb{C})$, there is an open neighborhood $V$ and local equations $\psi_i(z)=0$ for the divisors $\mathcal{Z}_i(\mathbb{C})$ such that the function
$$
\mathcal{E}(z)=\operatorname{Gr}(\mathcal{Z}, z)+\sum_i m_i \log ||\psi_i(z)||^2
$$
on $V \cap(\mathcal{S}(\mathbb{C}) \backslash \mathcal{Z}(\mathbb{C}))$ extends smoothly to $V \cap \mathcal{S}(\mathbb{C})$.
\item The forms $\mathcal{E}, \partial \mathcal{E}, \bar{\partial} \mathcal{E}$, and $\partial \bar{\partial} \mathcal{E}$ on $V \cap \mathcal{S}(\mathbb{C})$ have log-log growth along $\partial \cS^*(\mathbb{C})$.
\end{enumerate}
\end{definition}

\begin{definition}
We define an arithmetic divisor on $\mathcal{S}^*$ to be a pair $(\mathcal{Z}, \operatorname{Gr}(\mathcal{Z}, \cdot))$, where $\mathcal{Z}$ is a divisor  on $\mathcal{S}^*$ with complex coefficients, and $\operatorname{Gr}(\mathcal{Z}, \cdot)$ is a Green function for $\mathcal{Z}$. 
 We define a principal arithmetic divisor to be
$$
\widehat{\operatorname{div}}(f)=\left(\operatorname{div}(f),-\log ||f||^2\right)
$$
for some rational  function $f$ on $\mathcal{S}^*$.
\end{definition}

Given the definition of arithmetic divisor and principal arithmetic divisor, we can define the first arithmetic Chow group as follows.
\begin{definition}
   We recall  the first arithmetic Chow group 
   \begin{align*}
 \widehat{\mathrm{CH}}^1\left(\mathcal{S}^*\right)=      \mathrm{Span}\{ \text{arithmetic divisors}\}/  \mathrm{Span}\{ \text{principal arithmetic divisors}\}.
   \end{align*}
\end{definition}
Accordingly, we define 
$\widehat{\mathrm{Pic}}(\cS^*)$ to be the isomorphism classes of metrized line bundles  $\widehat{\cL}=(\cL, ||\cdot ||)$,  where $\mathcal  L$ is a  line bundle on $\mathcal S^*$, and $\| \cdot \|$ is a Hermitian metric on $\mathcal L$ such that if we choose a rational section  $s$ of $\mathcal L$ and set $g  = - \log \| s \|^2$, then     $\widehat{\mathrm{Div} } (s) = ( \mathrm{Div}(s),  -\log \| s \|^2)\in  \widehat{\mathrm{CH}}^{1}(\cS^*)$.  
 This gives a natural isomorphism
\begin{equation}\label{eq:iso pic and chow}
   i: \widehat{\mathrm{Pic}}(\cS^*) \rightarrow  \widehat{\mathrm{CH}}^{1}(\cS^*).
\end{equation}
  The preimage of the class of $(\cZ,g)$ under $i$  is represented by 
\[(\Oo(\cZ),||\cdot ||), \]
where $-\log ||{\bf 1} ||^2=g$ with ${\bf 1}$ the canonical section of $\Oo(\cZ)$.  
For our purpose, we extend this isomorphism $\R$-linearly to 
\begin{align}\label{eq: pic ch}
    \widehat{\mathrm{Pic}}_\R(\cS^*) \cong  \widehat{\mathrm{CH}}_\R^{1}(\cS^*).
\end{align} 
On the Picard group side, we allow formally real power of a line bundle $\mathcal L^r$   with metric being a $r$-power of a metric on  $\mathcal L$. On the arithmetic side, we allow the $r$-multiple of an honest divisor. We need this to accommodate the arithmetic divisors $\widehat{\mathcal Z}^{\tot}(m, v)$.

Note that according to \cite[Proposition 7.5]{BKK}, we can define a pullback $\varphi_{\Lambda}^* : \widehat{\mathrm{Pic}}_\R(\cS^{\dia,*})\to \widehat{\mathrm{Pic}}_\R(\cS^{*})$ since $\varphi_{\Lambda}^{-1}(\partial \cS^{\dia,*})\subset \varphi_{\Lambda}^{-1}(\partial \cS^{*})$.

\subsection{Arithmetic special divisors}
The complex fiber of  $\mathcal{S}$ is of the form 
\begin{equation}\label{eq:complex uniformization}
    \mathcal{S}(\C)=\bigsqcup_{\mathrm{Cl}(\kay)}\bigsqcup_{L_j\in [[L]]} \Gamma_j \backslash \mathcal{D}, 
\end{equation}
where $\Gamma_j$ is the automorphism group of $L_j$. Each connected component $\Gamma_j \backslash \mathcal{D}$ is a complex orbifold.
For $z\in \mathcal D$ and $x\in V_\R$, define
\begin{equation}
    R(x,z)=-(\mathrm{pr}_{z}(x), \mathrm{pr}_{z}(x))\geq 0.
\end{equation}
Then the majorant hermitian form 
\[(x,x)_z:=(x,x)+2R(x,z) = (\mathrm{pr}_{z^\perp} (x) , \mathrm{pr}_{z^\perp} (x)) - (\mathrm{pr}_{z}(x), \mathrm{pr}_{z}(x)) \]
is positive definite.
Define
\begin{equation}\label{eq:D x}
  \mathcal{D}(x)=\{z\in \mathcal{D} \mid (z,x)=0\}=\{z\in \mathcal{D} \mid R(x,z)=0\}.  
\end{equation}
Then $\mathcal{D}(x)$ is nonempty if and only if $(x,x)>0$. 
Let $Z(x)_j$ to be the image of $\mathcal D(x)$ in  $\Gamma_j \backslash \mathcal{D}$.

Let $Z(m)=\cZ(m)(\C)$.
Then we know that (\cite[\S 3]{KR2}) for $m\in \Z_{>0}$, we have
\[Z(m)=\bigsqcup_{\mathrm{Cl}(\kay)}\bigsqcup_{L_j\in [[L]]} \bigsqcup_{ \substack{x\in L_j  \mod \Gamma_j \\ (x,x)=m }} Z(x)_j.\]
Define 
\[\beta_1(r)=\int_{1}^\infty e^{-rt} \frac{dt}{t}.\]
Then there is a power series expansion  
\begin{equation}\label{eq:Taylor expansion}
    \beta_1(r)+\log r=-\gamma-\sum_{k=1}^\infty \frac{(-1)^{k} }{k} \frac{r^k}{ k!},
\end{equation}
where $\gamma$ is the Euler's constant. In particular,  $\beta_1(r)$ has log singularity at $r=0$.  For $0 \ne x\in L_j  $ and $v>0$, we define the Kudla's Green function on $\Gamma_j\backslash D$
\begin{equation}
    \mathrm{Gr}(x,v)_j(z)= \beta_1(4 \pi v \cdot R(x,z)).
\end{equation}
Then  $\mathrm{Gr}(x,v)_j$ is a Green function for $Z(x)_j$ (c.f. \cite[\S 3]{howard2012complex}).
For $v>0$ and $m\in \Z$, define Kudla's Green function $\mathrm{Gr}(m,v)(z)$ such that on the component $\Gamma_j \backslash \mathcal D$ it is of the form
\begin{equation}
    \mathrm{Gr}(m,v)(z)|_{\Gamma_j \backslash \mathcal D}=\sum_{ \substack{x\in L_j\setminus \{0\} \\ (x,x)=m}}  \mathrm{Gr}(x,v)_j(z).
\end{equation}
Then $\mathrm{Gr}(m,v)$ is a Green function for $Z(m)$.
Note that $Z(m) =0$ and $\mathrm{Gr}(m, v)$ is smooth on $\mathcal D$ for $ m\le 0$. 
 The following theorem states that  $\Gr(m,v)$ is a Green function  of $\cZ^{\tot}(m).$
\begin{theorem}\cite[Theorem 3.7.4]{Ho1}\label{thm: boundary behavior of Gr}
 Let $m$ be a non-zero integer, and  $p$    be a complex point of  $\cS^*(\Phi)$ for some cusp label $\Phi$. There exists an open neighborhood $U \subset \mathcal{\cS}^*(\mathbb{C})$ of $p$ such that the smooth function
$$
\mathcal{E}(z)=\operatorname{Gr}(m, v)(z)+\log \left|\psi_m(z)\right|^2+\frac{\#\left\{x \in L_\Phi:( x, x)=m\right\}}{4 \pi v} \log |q(z)|^2
$$
on $U \setminus \cS^*(\Phi)(\mathbb{C})$ is bounded, and the differential forms $\partial \mathcal{E}, \bar{\partial} \mathcal{E}$, and $\partial \bar{\partial} \mathcal{E}$ have log-log growth along $\cS^*(\Phi)(\mathbb{C})$. Here $\psi_m(z)=0$ is a local equation for $\mathcal{Z}^*(m)(\mathbb{C})$, and $q(z)=0$ is a local equation for the boundary component $\cS^*(\Phi)(\mathbb{C})$. In particular, $\operatorname{Gr}(m, v)$ is a Green function of $\mathcal Z^{\mathrm{tot}}(m, v)$.
\end{theorem}

By Theorem \ref{thm: boundary behavior of Gr}, we have a well-defined arithmetic Chow divisor for integers $m \ne 0$
\begin{equation}
  \widehat{\cZ}^{\tot}(m,v)\coloneqq (\cZ^{\tot}(m,v),\Gr(m,v)) \in \widehat{\CH}_\R^1(\cS^*).
\end{equation}
It is an arithmetic divisor with the Kudla Green function, denoted by $\widehat{\cZ}_K^{\tot}(m,v)$ in the introduction. We will add subscript $K$ when needing to distinguish it from the arithmetic divisors $\widehat{\mathcal Z}_B^{\tot}(m, v)$ with the Bruinier Green functions

Now we deal with the case $m=0$. Recall that on $\mathcal{S}(\C)$ the two line bundles $\Omega$ and $\omega$ agree, and can be identified the line bundle (still denoted by $\Omega$)  of modular forms of weight $1$ on $\cS(\C)$  by \cite[Proposition 6.2]{BHY}. Moreover, this line bundle is descended from the tautological line bundle on $\mathcal D$, again denoted by $\Omega$. Let $s$ be a section of $\Omega$ and for $z\in \cD$, we let $s_z \in z \subset V_\R$ denote the value of $s$ at  $z$.  Then we can define a metric on $\Omega$ by setting
\begin{equation}\label{eq: metric of Omega}
 \| s_z \|^2 = - 4 \pi v e^\gamma (s_z, s_z).
\end{equation}
Here $\gamma$ is the Euler constant and $v>0$ is an extra parameter (eventually imaginary part of modular variable $\tau \in \mathbb H$).  Compared to the nomalization in \cite[Page 34]{ES}, we add a $v$-factor to count the arithmetic divisor $(0, \log v) $ there.     This gives an element $\widehat{\Omega}^{-1}  \in \widehat{\hbox{Pic}}(\cS)$ that depends on an extra factor $v$. By \cite[Proposition 6.3]{BHY}, this element extends naturally to an element, still denoted by $\widehat{\Omega}^{-1}$ in $  \widehat{\mathrm{Pic}}(\cS^*) \hookrightarrow \widehat{\mathrm{CH}}_\R^1(\cS^*) $.

Following \cite[Page 34]{ES}, we define $\widehat{\cZ}^\tot(0,v)$ as 
\begin{equation}
    \widehat{\cZ}^\tot(0,v)=\widehat{\Omega}^{-1}+(\cB(0,v),\Gr(0,v))\in\widehat{\mathrm{CH}}_\R^1(\mathcal{S}^*),
\end{equation}
where $\cB(0,v)$ is defined in \eqref{eq:Bmv}.

\subsection{Pullback formulas for arithmetic divisors and Proof of Theorem  \ref{theo:Decomposition}}
As in \S \ref{sec:pullback I}, we use the superscript $\dia$ to denote objects on $\mathcal{S}^\dia$ or $\mathcal{S}^{\dia,*}$. The aim of this subsection is to prove   Theorem  \ref{theo:Decomposition}, which is equivalent to the following theorem.

\begin{theorem}\label{thm:pullback final}
    Assume $m\in \Z$ and $v>0$. Then the following equation 
\begin{align}\label{eq: pullback IV}
    \varphi_{\Lambda}^*(\widehat{\cZ}^{\dia,\tot}(m,v))=\sum_{m_1+m_2=m} r_{\Lambda}(m_2)\cdot \widehat{\cZ}^\tot(m_1,v)
\end{align}
holds in $\widehat{\mathrm{CH}}^1_\R(\mathcal{S}^*)$.
\end{theorem}

The right-hand side is an infinite sum and needs a little explanation as $m_1$ can be negative.  It is equal to 
$$
r_\Lambda(m) \widehat{\mathcal Z}(0,\mu)  + (\sum_{\substack{m_1 + m_2 =m\\ m_1>0}} r_\Lambda(m_2) \mathcal Z^{\tot}(m_1, v),   \sum_{\substack{m_1 + m_2 =m\\ m_1\ne 0}} r_\Lambda(m_2) G(m_1, v)).
$$
The sum over cycles is finite as $\mathcal Z^{\tot}(m_1, v)=0$ for $m_1 <0$,  and the sum of Green functions is a convergent sum. 
\begin{proof}    
We regard a Green function as a metric on the corresponding line bundle via \eqref{eq: pic ch} and reduce the improper pullback to a proper pullback following an idea in \cite[\S 7.4]{BHY}.  Let $\widehat{\cL}_m^{\dia,\tot}$ be the metrized  line bundle associated to $\widehat{\cZ}^{\dia,\tot}(m,v)$ with canonical section $\mathbf 1_m^\dia$ such that
$$
-\log \| \mathbf 1_m^\dia(z)\|^2 = G^\dia(m, v)(z) 
$$
for all $z \in  \mathcal D^\dia -Z^{\tot}(m, v)$. We use the notations $\widehat{\mathcal L}_m^{\tot}$ and $\mathbf 1_m$ similarly. Recall that $\widehat{\Omega}^\dia$ (resp. $ \widehat{\Omega}$) is equipped with the natural metric such that
$$
 \| s_z \|^2 = - 4 \pi v e^\gamma (s_z, s_z),
$$
and 
\begin{align}\label{eq: pullback of Omega}
    \varphi_\Lambda^*(\widehat{\Omega}^\dia)  =\widehat{\Omega}.
\end{align}
Here by $\varphi_{\Lambda}^*$ we mean pullback as metrized line bundle, which is always defined.
Finally, let $\widehat{\mathcal L}_0^B$ be the metrized line bundle associated to $(\cB(0,v),\Gr(0,v))$ with a canonical section $\mathbf 1_0^B$ such that 
$$
-\log \| \mathbf 1_0^B(z) \|= G(0, v)(z).
$$
Now by \eqref{eq: pullback of Omega}, the desired identity is equivalent to 
\begin{align}\label{eq: eq 1 thm 6.7}
    \varphi_{\Lambda}^*(\widehat{\cL}_m^{\dia,\mathrm{tot}}\otimes (\widehat{\Omega}^{\dia})^{\otimes r_\Lambda(m)})
    =(\widehat{\mathcal L}_0^B)^{\otimes r_\Lambda(m)} \otimes \bigotimes_{\substack{m_1+m_2=m,\\m_1\neq 0}}  (\widehat{\cL}_{m_1}^{\mathrm{tot}})^{\otimes r_{\Lambda}(m_2)}.
\end{align}
By Theorem \ref{thm:pullback total}, it holds at the line bundle level. So it suffices to prove that it holds for the metrics over $\cS$ (all the metrics on $\cS^*$ are natural extension of those on  $\cS$).

On the right hand side, the natural rational  section 
$$
s_{\mathrm{nice}}= (\mathbf 1_0^B)^{r_{\Lambda}(m)} \otimes   
\bigotimes_{\substack{m_1+m_2=m,\\m_1\neq 0}}  \mathbf 1_{m_1}^{\otimes r_{\Lambda}(m_2)}
$$ has metric 
\begin{equation} \label{eq:nice}
-\log \| s_{\mathrm{nice}}(y)\|^2 = r_\Lambda(m) G(0, v)(y) + \sum_{\substack{ m_1 +m_2 =m, \\ m_1 \ne 0} } r_\Lambda(m_2) \Gr(m_1, v)(y)
\end{equation}
for $y \in  \mathcal D$.

On the left hand side, for $z \in  \mathcal D^\dia -\mathcal D$,
\begin{align*}
&-\log \| \mathbf 1_m^\dia(z)\|^2
=\sum_{\substack{0\ne x_1 + x_2 =x \in L^\dia \\
       x_1 \in L,  x_2 \in  \Lambda \\
       Q(x) =m}}
 \Gr^\dia(x, v) (z)
 \\
&=\sum_{ \substack{0 \ne x \in \Lambda \\ Q(x) =m}}
\beta_1(4 \pi v R^\dia(x, z) )
+ \sum_{\substack{0 \ne  x_1 \in L, Q(x_1) =0   \\ x_2 \in \Lambda,  \, Q(x_2) =m
}} \Gr^\dia(x_1  +x_2, v) (z)
+ \sum_{\substack{m_1 +m_2  =m  \\
       x_1 \in L,  \,  Q(x_1)=m_1 \ne 0  \\
       x_2 \in  \Lambda, \, Q(x_2) =m_2}}
 \Gr^\dia(x_1 +x_2, v) (z).
 \end{align*}
Recall that by \eqref{eq:Taylor expansion} we have
$$
 \beta_1(4 \pi v R^\diamond(x, z) ) =
  -\gamma - \log (4 \pi v R^\diamond(x, z))
    + f_x(z) 
$$
 where  
$$
f_x(z) = -\sum_{k=1}^{\infty}\frac{(-1)^k}{k} \frac{(4\pi v\cdot R^\diamond(x,z))^k}{k!}. 
$$
for all $z \in  \mathcal D^\dia$. Notice that $f_x(z) =0$ for $z \in \mathcal D$. 

For  $0\ne x \in \Lambda$, we have  $(x, x)>0$ and we can define a local  section $s_x$ of  $\Omega^{-1}$ over $\mathcal D^\dia$ so that  
over a point $z \in \cD^\dia$, $s_{x, z}$ is characterized by
$$
s_{x, z}(a) =(\mathrm{pr}_z(x), a) \quad \text{ for } a \in z.
$$
According to \eqref{eq: metric of Omega}, we have
\begin{equation}
    - \log \| s_{x, z}\|^2 = -\gamma - \log (4 \pi v R^\diamond(x, z)).
\end{equation}
So we have a rational section $s$ of $\Omega^{\otimes - r_\Lambda(m)}$ over $S$ with 
$$
s(z) = \sum_{\substack{0\ne x \in \Lambda\\ (x, x)=m} } s_{x, z}
$$
and 
$$
-\log \| s(z) \|^2 = \sum_{\substack{0\ne x \in \Lambda\\ (x, x)=m} } (-\gamma - \log (4 \pi v R^\diamond(x, z)).
$$
As in \cite[\S 7.4]{BHY}, we have an analytic analogue of \eqref{eq: adjunction}. More precisely, for $x\in \Lambda$, we have an isomorphism
\begin{align}\label{eq: analytic adjunction}
    \Oo_{Z(x)}|_{\mathcal{D}}\cong \Omega^{-1}
\end{align}
so that the section $s_x$ of $\Omega^{-1}$ corresponds to the canonical section of $\Oo_{Z(x)}|_{\mathcal{D}}$. 
Moreover, by the same proof of \cite[Theorem 7.12]{BHY}, we have a section $s_m$ of $\Omega^{-1}$ on an \'etale neighborhood $U$ of $\cS^*$ that corresponds to the canonical section of $\Oo_{\cZ(m)}|_U$ under \eqref{eq: adjunction} and agrees with $s=\sum_{\substack{0\ne x \in \Lambda\\ (x, x)=m} } s_{x}$ on $\cD$. 
So we have a `proper' rational  section  $s_{\mathrm{prop}}=\mathbf 1_m^\diamond \otimes s^{-1}$  of  $\mathcal L_m \otimes \Omega^{\otimes  r_\Lambda(m)}$ that corresponds to $s_{\mathrm{nice}}$ under \eqref{eq: eq 1 thm 6.7} and has metric 
\begin{align*}
&-\log\| s_{\mathrm{prop}} (z)\|^2=
  -\log \| 1_m^\diamond(z)\|^2  + \log \| s(z) \|^2 
\\
&= \sum_{ \substack{0 \ne x \in \Lambda \\ Q(x) =m}}
f_x(z)+\sum_{\substack{0 \ne  x_1 \in L, Q(x_1) =0   \\ x_2 \in \Lambda,  \, Q(x_2) =m
}} \Gr^\dia(x_1  +x_2, v) (z)
+ \sum_{\substack{m_1 +m_2  =m  \\
       x_1 \in L,  \,  Q(x_1)=m_1 \ne 0  \\
       x_2 \in  \Lambda, \, Q(x_2) =m_2}}
 \Gr^\dia(x_1 +x_2, v) (z).
\end{align*}
Now for any $y \in \mathcal D$, let $z \in \mathcal D^\dia -\mathcal D$ goes to $y$,  the right hand side has limit, and  since  $f_x(z) =0$ for $z \in \mathcal D$, we have  
\begin{align*}
&-\log\| s_{\mathrm{prop}} (y)\|^2  
\\
&=\sum_{\substack{0 \ne  x_1 \in L, Q(x_1) =0   \\ x_2 \in \Lambda,  \, Q(x_2) =m
}} \Gr^\dia(x_1 , v) (y)
+ \sum_{\substack{m_1 +m_2  =m  \\
       x_1 \in L,  \,  Q(x_1)=m_1 \ne 0  \\
       x_2 \in  \Lambda, \, Q(x_2) =m_2}}
 \Gr^\dia(x_1 , v) (y)
 \\
 &=- \log\| s_{\mathrm{nice}}(y) \|^2
\end{align*}
by (\ref{eq:nice}) as desired.  This proves the theorem.

\end{proof}

\section{Modularity on unitary Shimura curves} \label{sect:Modularity}
\subsection{Eisenstein Series} In this  subsection, we temporarily let 
 $L$ be  a unimodular positive definite Hermitian lattice over $\Oo_{\kay}$ of rank $\ell$.
Let $\psi=\prod_p \psi_p$ be the canonical additive character of $\Q \backslash \A$ with $\psi_\infty(x) = e(x) =e^{ 2 \pi i x}$. Associated to it is a Weil representation $\omega=\omega_{L, \psi}$  of $\hbox{SL}_2(\A) \subset U(1, 1)(\A)$ on $S(V_\A)$ with $V=L_\Q= L \otimes_\Z \Q$, and a $\SL_2(\A)$-equivariant ``Rallis" map ($s_\ell=\ell -1$)
$$
\lambda: \, S(V_\A) \rightarrow I(s_\ell, \chi), \quad \lambda(\phi)(g) =\omega(g) \phi(0).
$$
Here $\chi=\chi_{-D}^\ell$. Let $\phi_L=\cha(\widehat{L}) \otimes \phi_\infty \in S(V_\A)$ with $\phi_\infty(x) =e^{-\pi (x, x)}$, and let $\Phi=\Phi_L=\prod_{p \le \infty} \Phi_p$ be the standard section in $I(s, \chi)$ such that $\Phi_L(g, s_\ell) = \lambda(\phi_L) (g)$. Then it is known that $\Phi_\infty =\gamma(V_\infty) \Phi_\infty^\ell$ is the standard section of weight $\ell$ (up to a scalar multiplication by a local Weil index $\gamma(V_\infty)$): 
$$
\Phi_\infty^\ell(n(b) m(a) k_\theta, s) = \chi_\infty(a)|a|^{s+1} e^{i \ell \theta}, 
$$
where 
$$
n(b) =\kzxz {1} {b} {0} {1},  \quad m(a) =\kzxz {a} {0} {0} {a^{-1}}, \quad \hbox{ and } \quad k_\theta =\kzxz {\cos \theta} {\sin \theta} {-\sin \theta} {\cos \theta}.
$$
Let  (here $g_\tau = n(u)m(\sqrt v)$ for $\tau= u +i v \in \H$)
\begin{equation}
E_L(\tau, s) = v^{-\ell/2} \sum_{\gamma \in B\backslash \SL_2(\Q)} \Phi_L(\gamma g_\tau, s),  \quad  
\end{equation}
be the associated Eisenstein series of weight $\ell$ (level $D$ and character $\chi$), which depends only on the genus $[[L]]$ of $L$. Here 
$$
B=NM=\{ n(b) m(a):\, a \in \Q^\times, b \in \Q\}.
$$
Let $\Gamma_\infty =B \cap \SL_2(\Z)$, then $B \backslash \SL_2(\Q) = \Gamma_\infty \backslash \SL_2(\Z)$.

For $\gamma=\kabcd \in \SL_2(\Q)$, let $\gamma_p$ be its image in $\SL_2(\Q_p)$. We first record the following standard facts as lemmas and a theorem (see for example \cite{YaValue},  \cite{KRYtiny}, \cite{KYEisenstein}). They are needed in the next few subsections. 

\begin{lemma} For $\gamma =\kabcd \in \SL_2(\R)$, we have 
$$
\Phi_\infty^\ell(\gamma g_\tau, s) =v^{\frac{s+1}2} j(\gamma, \tau)^{-\ell} |j(\gamma, \tau)|^{\ell-1 -s}  .
$$
Here $j(\gamma, \tau) =c \tau +d$. In  particular, 
$$
\Phi_\infty^\ell(\gamma g_\tau, s_\ell) = v^{\frac{\ell}2}  j(\gamma, \tau)^{-\ell}.
$$
\end{lemma} 

\begin{lemma} Let $\gamma=\kabcd \in \SL_2(\Z_p)$.

\begin{enumerate}
\item When  $p \nmid D$, $\Phi_p(\gamma, s_\ell) =1$ and $\Phi_p$ is right $\SL_2(\Z_p)$-invariant.

\item When  $p \mid D$ is odd, we have 
$$
\Phi_p(\gamma, s_\ell) = \begin{cases} 
     \chi_{p} (a) &\hbox{if }  p|  c,
     \\
     \gamma(V_p)\chi_{p}(c)  p^{-\ell/2} &\hbox{if } p \nmid c.
     \end{cases}
$$
Moreover, $\Phi_p(g \gamma) =\Phi_p(g) \chi_{p}(a)$ for $p|c$, and  $\gamma(V_p)$ is the local Weil index (a root of unity).
\end{enumerate} 
\end{lemma}

Notice that similar results hold when $s_\ell$ is replaced by general $s$ although we do not need it here. We have now the following theorem by the above lemmas.

\begin{theorem}\label{thm:Eisenstein Series} 
Let the notation be as above. Then  for $\ell \ge 4$ and $s_\ell=\ell -1$, 
\begin{align*}
E_L(\tau, s_\ell) &= \sum_{\gamma \in  \Gamma_\infty \backslash \SL_2(\Z)}
 a(\gamma, \ell)  j(\gamma, \tau)^{-\ell}
\\
 &= \sum_{\delta \in  \Gamma_\infty \backslash \SL_2(\Z)/\Gamma_1(D)} a(\delta, \ell)  \sum_{ \gamma \in \Gamma_\infty \backslash \delta \Gamma_1(N)} j(\gamma, \tau)^{-\ell}.
\end{align*}
Here  $a(\gamma, \ell) =\prod_{p |D}  a_p (\gamma, \ell)$  is right $\Gamma_1(D)$-invariant with 
$$
a_p(\gamma, \ell) 
 = \begin{cases}
  \gamma(V_p)\chi_p(a)   &\hbox{if  }  p |c,
  \\
  \gamma(V_p)\chi_p(c) p^{-\ell/2} &\hbox{if } p\nmid c,  
\end{cases} 
$$ 
if $\gamma=\begin{pmatrix}
    a & b\\
    c & d
\end{pmatrix}$.
\end{theorem}

\subsection{Modularity of arithmetic theta functions with Kudla Green functions} \label{sect:K}

In this subsection, we keep the notation of previous sections but assume  $n=2$. In particular,  $L=\Hom_{\Oo_{\kay}}(\mathfrak a_0, \mathfrak a)$ is a self-dual hermitian lattice of signature  $(1, 1)$. Let 
\begin{equation}
\widehat{\Theta}_K (\tau) = \sum_{m \in \Z} \widehat{\mathcal Z}^{\tot} (m, v) q^m,  \quad v = \Im(\tau)   
\end{equation}
be the arithmetic theta function with Kudla Green functions.
The purpose of this subsection is to prove the following theorem, which is part of Theorem \ref{theo:Modularity}.

\begin{theorem}\label{theo:K} Let the notation be as above. Then $\widehat{\Theta}_K(\tau)$ is a (non-holomorphic) modular form of $\Gamma_0(D) $  with values in $\widehat{\operatorname{CH}}_\C^1(\mathcal S^*)$. By modularity, we mean that for every linear map $\lambda: \widehat{\operatorname{CH}}_\C^1(\mathcal S^*) \rightarrow \C $,  $\lambda (\widehat{\Theta}_K(\tau)) $ is a real analytic modular form of the same type.
\end{theorem}
\begin{proof}
Let $L_0$ be a positive definite self-dual lattice of rank $1$ and $V_0=L_0\otimes \Q$.
Applying Theorem \ref{theo:Decomposition} to each lattice $\Lambda$ in the genus of $L_0^\ell$ and Siegel-Weil Formula for $L_0^\ell$,  we have 
\begin{equation}\label{eq:K}
  \sum_{\Lambda \in [[L_0^\ell]]}  \frac{1}{|\mathrm{Aut}(\Lambda)|} \varphi_\Lambda^*(\widehat{\Theta}_K^\dia(\tau))= \sum_{\Lambda \in [[L_0^\ell]]}  \frac{1}{|\mathrm{Aut}(\Lambda)|}  \theta_{\Lambda}(\tau)\cdot \widehat{\Theta}_K(\tau)
  = E_{L_0^\ell}(\tau,  s_\ell) \widehat{\Theta}_K(\tau) .  
\end{equation}
By  \cite[Theorem B]{BHKRY1} and \cite[Theorem 1.4]{ES},  the left hand side of \eqref{eq:K} is a (non-)holomorphic modular form of  $\Gamma_0(D)$ of  weight $2+\ell$ and character $\chi_{-D}^{\ell+2}$.  So 
$$
\widehat{\Theta}_K(\tau) =\frac{\sum_{\Lambda \in [[L_0^\ell]]}  \frac{1}{|\mathrm{Aut}(\Lambda)|} \varphi_\Lambda^*(\widehat{\Theta}_K^\dia(\tau))}{E_{L_0^\ell}(\tau,  s_\ell)}
$$
is a meromorphic  modular form of  $\Gamma_0(D)$ of  weight $2$ and character $\chi_{-D}^{2}$ with possible poles at the zeros of the Eisenstein series. We now  prove that $\widehat{\Theta}_K(\tau)$ is has no poles by varying the Eisenstein series. By Theorem \ref{thm:Eisenstein Series}, we see that
\[E_{L_0^\ell}(\tau,s_\ell)=\sum_{\gamma \in \Gamma_\infty\backslash \mathrm{SL}_2(\Z)} (\prod_{p|D} b_p(\gamma,\ell)\cdot j(\gamma,\tau)^{-1})^\ell, \]
where  
\[b_p(\gamma,\ell)=\begin{cases}
  \gamma(V_0)\chi_p(a)   &\hbox{if  }  p |c,
  \\
  \gamma(V_0)\chi_p(c) p^{-1/2} &\hbox{if } p\nmid c,  
\end{cases} \]
if $\gamma=\begin{pmatrix}
    a & b\\
    c & d
\end{pmatrix}$.
For every $\tau_0 \in  \H$,  \cite[Lemma 5.6, Chapter I]{freitag1990hilbert} implies that  there is some $\ell>0 $ such that   $E_{L_0^\ell}(\tau_0 ,s_\ell)\ne 0$. So $\widehat{\Theta}_K(\tau)$ is  well-defined at $\tau_0$ and thus well-defined everywhere.  This proves that $\widehat{\Theta}_K(\tau)$ is modular.
\end{proof}

Now, the modularity for  $\widehat{\Theta}_B(\tau)$  (the associated Green functions are Bruinier Green functions as mentioned in the introduction) should follow directly from Theorem  \ref{theo:K} and \cite[Theorem 1.4]{ES}. However, there is a little subtlety involved: the Green function constructed by Ehlen and Sankaran is a little different from that of Bruinier when  $n=2$ as in our case. The next three subsections explains the subtlety and proves Theorem \ref{theo:Modularity}.

\subsection{A little preparation}  \label{sect:ES}
We first  recall  Ehlen and Sankaran's Green functions. 
In this subsection, let $L$ be an even dimensional integral non-degenerate $\Z$-lattice of signature $(2, 2)$ (in our case, we view our unimodular $\Oo_{\kay}$-lattice  $L$ of signature $(1, 1)$ as a quadratic lattice with quadratic form $q(x) =(x,x)$). Let $L'$ be the dual of $L$ with respect to the quadratic form, and $S_L= \C[L'/L]$ with Weil representation $\rho_L$. Let $S_L^\vee $ be its dual with dual Weil representation  $\rho_L^\vee$. Let $\phi_\mu$ ($\mu \in L'/L$) be the standard basis of $S_L$ and $\phi_\mu^\vee$ be the dual basis of $S_L^\vee$.

\begin{definition} (\cite{BF04}) For an  integer  $k \in \Z$,   Let $H_k(\rho_L)$ be the space of twice continuously differentiable functions $F: \H \rightarrow S_L$ such that 
\begin{enumerate}
\item $ F(\gamma)(\tau) =(c \tau +d)^k  \rho(\gamma) F(\tau)$ for $\gamma \in \Gamma=\hbox{SL}_2(\Z)$.

\item (at most exponential growth) there is a constant $C>0$ such that $F(u+ i v) =O(e^{Cv})$
as $v \rightarrow \infty$.

\item ($k$-harmonic) We have  $\Delta_k(F) =0$, where 
$$
\Delta_k = -v^2 (\frac{\partial^2}{\partial u^2} + \frac{\partial^2}{\partial v^2}) + ik v (\frac{\partial}{\partial u}  + i \frac{\partial}{\partial v})
$$
is the hyperbolic Laplace operator in weight $k$.

\item  $\xi_k(F) \in S_{2-k}(\rho_L^\vee)$ is a cuspidal modular form values in $S_L^\vee$. Here  the $\xi$-operator is given by 
$$
\xi_k(F) = 2 i v^k \overline{\frac{\partial F}{\partial \bar\tau}} = v^{k-2} \overline{\mathbf L(F)},
$$
where $\mathbf L = -2 i v^2 \frac{\partial}{\partial \bar\tau}$ is the Maass lowering operator.
\end{enumerate}
   An element in $H_k(\rho_L)$ is called a harmonic Maass modular form (of weight $k$ with  values in $S_L$). It is called a weakly holomomorphic modular form if $\xi_k(F) =0$. We denote the space of weakly holomomorphic forms in $H_k(\rho_L)$ by $M_k^!(\rho_L)$, and the space of modular forms in $H_k(\rho_L)$ by $M_k(\rho_L)$.
\end{definition}

According to Bruinier and Funke (\cite{BF04}),  the $\xi$-operator gives an exact sequence 
\begin{equation}
0 \rightarrow M_k^!(\rho_L) \rightarrow H_k(\rho_L) \rightarrow S_{2-k}(\rho_L^\vee)  \rightarrow 0.
\end{equation}
Every $F \in H_k(\rho_L)$ is smooth and admits a decomposition 
\begin{equation}
F(\tau) = F^+(\tau) + F^-(\tau)
\end{equation}
into its holomorphic and non-holomorphic parts. Here its holomorphic part 
\begin{equation}
    F^+(\tau)=\sum_{m \gg - \infty } c_F^+(m) q^m,  \quad c_F^+(m) \in S_L
\end{equation}
has only finitely many negative terms, and its non-holomorphic part $F^-(\tau)$ is of exponential decay as $v =\Im(\tau)$ goes to infinity.

\begin{lemma} Let $F \in  H_0(\rho_L)$ with $c_F^+(m) =0$ for all $m <0$. Then 
 $F \in M_0(\rho_L) =S_L^{\hbox{SL}_2(\Z)}$ is a holomorphic modular form of weight $0$, i.e., an element in $S_L$ fixed by $\hbox{SL}_2(\Z)$ via $\rho_L$.
\end{lemma}
\begin{proof}
Since  $c_F^+(m) =0$ for  $m <0$, $F^+(\tau)$ and thus $F(\tau)$ is bounded at the cusp $\infty$. So $F(\tau)$ is a bounded  harmonic ($S_L$-valued) function on the compact modular curve $\hbox{SL}_2(\Z) \backslash \H^* =\mathbb P^1(\C)$ and is thus a constant.
\end{proof}

The following is basically  a special case of \cite[Lemma 2.4]{ES}.

\begin{lemma} \label{lem:F} For $m \in \Q$ and $\mu \in  L'/L$ with $Q(\mu) \equiv m \mod \Z$, there is a unique $F_{m, \mu} \in H_0(\rho_L)$ such that 
\begin{enumerate}
    \item  The holomorphic part of $F_{m, \mu}$ has the form
$$
F_{m, \mu}^+ (\tau) =q^{-m} \tilde\phi_\mu  + \sum_{n \ge 0} c_{m, \mu}^+ (n) q^n,  \quad c_{m, \mu}^+(n) \in S_L,
$$
where $\tilde\phi_\mu = \frac{1}2 (\phi_\mu + \phi_{-\mu})$.  We denote  $c_m^+(0)$ for the $\phi_0$-component of $c_{m,0}^+(0)$ for later use.

  \item The identity
  \begin{equation} \label{eq:n=0}
\sum_{\nu \in L'/L}  a_f(0, \nu)  c_{m, \mu}^+(0, \nu) =a_f(m, \mu)
\end{equation}
holds for every $f \in M_{0}(\rho_L^\vee)$.  Here $a_f(n, \nu)$ is the $(n, \nu)$-th coefficient of $f$, and $c_{m, \mu}^+(n, \nu)$ is the $(n, \nu)$-th coefficient of $F_{m, \mu}^+$.
\end{enumerate}
    
\end{lemma}
\begin{proof}
   The lemma  follows from \cite[Lemma 2.4]{ES} together with the following explicit construction of a splitting map $\eta$ in \cite[(2.8)]{ES}. First notice that $M_0(\rho_L)=S_L^{\SL_2(\Z)}$ is the maximal subspace of $S_L$ on which $\SL_2(\Z)$ acts trivially. Secondly,  we have a natural embedding
$$
M_0(\rho_L^\vee) \rightarrow \hbox{Sing}_2(\rho_L^\vee),  \quad f=\sum c_f(m) q^m  \mapsto P(f) =\sum_{m \le 0} c_f(m) q^m,
$$
where
$$
\hbox{Sing}_2(\rho_L^\vee) =\{ P =\sum_{m \le 0} a_P(m) q^m:\,  a_P(m) \in S_L^\vee \}.
$$
Now the natural non-degenerate bilinear form
$$
S_L \times S_L^\vee\rightarrow  \C ,  \quad  \langle \sum a(\mu) \phi_\mu,  \sum b(\mu) \phi_\mu^\vee \rangle =\sum a(\mu) b(\mu)
$$
gives rise to the splitting map
$$
\eta:  M_0(\rho_L)^\vee  \cong M_0(\rho_L^\vee) \rightarrow \hbox{Sing}_2(\rho_L^\vee).
$$
It is clear from the definition that \cite[Lemma 2.4  Condition(ii)]{ES} is the same as  (\ref{eq:n=0}).
\end{proof}

\begin{corollary} Let the notation be as in  Lemma \ref{lem:F}. Then $F_{m, \mu}=0$ for $m <0$, and $F_{0, \mu} \in M_0(\rho_L)$. 
\end{corollary}

Following  Ehlen and Sankaran \cite{ES}, we can view $F_m$ as a harmonic 
 modular form of weight $0$ valued in $S_L \otimes S_L^\vee$ via 
$$
F_m =\sum_{\mu \in L'/L} F_{m, \mu} \phi_\mu^\vee.
$$

Let $V=L_\Q =L \otimes_\Z \Q$, and let 
$\mathcal D^o$ be the oriented  negative $2$-planes in $V_\R$. For $z \in \mathcal D^o$, decompose $$V_\R =z \oplus z^\perp,  \quad x = x_z + x_{z^\perp}.$$ For $h \in \SO(V)(\A_f)$ and $z \in \mathcal D^o$, one has the Siegel-theta function 
\begin{equation}
\theta(\tau, z, h; \phi) =v  \sum_{x \in  V} \phi(h^{-1}x)  \phi_\infty(\sqrt v x, z)
\end{equation}
which is a modular form of $\tau=u+iv$ of weight $0$, and $\SO(V)(\Q)$-invariant. Here 
$$
\phi_\infty(x, z) = e^{-\pi (x, x)_z} \in S(V_\R).
$$
Denote 
\begin{equation}
\theta_L(\tau, z, h) =\sum_{\mu \in L'/L} \theta(\tau, z, h; \phi_\mu) \phi_\mu^\vee.
\end{equation}
which is a (non-holomorphic) modular form of $\SL_2(\Z)$ of weight $0$ valued in $S_L^\vee$.
The Bruinier Green functions (the Ehlen and Sankaran version)  are defined as 
\begin{align}\label{eq:ESGreen}
G_{ES}(m, \mu) &= \int^{reg} \langle F_{m, \mu}, \theta_L \rangle d\mu(\tau)
\\
 &= \operatorname{CT}_{s=0} \lim_{T \rightarrow \infty} \int_{\mathcal F_T}  \langle  F_{m, \mu}, \theta_L \rangle v^{-s} d\mu(\tau).   \notag
\end{align}
Here 
$
\langle F_{m, \mu}, \theta_L \rangle 
$
comes from  the natural paring between $S_L$ and $S_L^\vee$, 
$$
\mathcal F_T =\{ \tau =u+ i v \in \H:\,  -\frac{1}{2} \le u \le \frac{1}{2},  0 <v <T, |\tau| \ge 1 \}
$$ is the truncated fundamental domain for $\SL_2(\Z)$, and finally $\operatorname{CT}_{s=0}(f(s))$ is the constant term of the Laurent series of $f(s)$ at $s=0$. 

\subsection{Arithmetic theta function with Bruinier Green functions I---Ehlen and Sankaran's normalization}  \label{sect:B}

Now we come back to the notation and assumptions of Section \ref{sect:K}. The harmonic functions $F_m$ in Section \ref{sect:ES} are now associated to the lattice $L$ with quadratic form $Q(x) =(x, x)$.

For a cusp $\Phi$ associated to the genus of  $L$  and $m \ge 0$, let  
\begin{equation} \label{eq:cm0}
\eta_\Phi(m) =-2[  \sigma_1(m) + c_{m}^+(0)].
\end{equation}
Here  $c_m^+(0)$ is defined in Lemma \ref{lem:F} and  $\sigma_1(0) =-\frac{1}{24}$. Note that $\eta_\Phi$ does not depend on the choice of the cusp $\Phi$. 

Let 
\begin{equation}
    \mathcal Z_{ES}^{\tot}(m) =\mathcal Z^*(m) +  \sum_{\Phi \in  \mathrm{Cusp}([[L]]) } \eta_{\Phi}(m) [\mathcal S^*(\Phi)]  \in {\CH}_\C^1(\mathcal S^*),
\end{equation}
for $m \in \Z_{>0}$, and 
\begin{equation}
    \mathcal Z_{ES}^{\tot}(0) = \widehat{\Omega}^{-1} + (0, -\log D)  +  \sum_{\Phi\in\mathrm{Cusp}([[L]]) } \eta_{\Phi}(0) [\mathcal S^*(\Phi)]  \in {\CH}_\C^1(\mathcal S^*).
\end{equation}
Define 
\begin{equation}
\widehat{\cZ}_{ES}^{\mathrm{tot}}(m)
=(\mathcal Z_{ES}^{\tot}(m),  G_{ES}(m))  \in  \widehat{\CH}_\C^1(\mathcal S^*).
\end{equation}
Here the Ehlen-Sankaran Green function  $G_{ES}(m) =G_{ES}(m, 0)$ is the Green function  (\ref{eq:ESGreen})
restricted to $\mathcal D \times U(V)(\A_f)$, via the embedding 
$$
\mathcal D \times U(V)(\A_f)  \rightarrow \mathcal D^o \times \SO(V) (\A_f), 
$$ associated to  the natural embedding $U(V) \subset \SO(V)$.

Finally define the arithmetic theta function with Ehlen-Sankran Green functions
\begin{equation}
\widehat{\Theta}_{ES} (\tau) = \sum_{m \in D^{-1} \Z_{\ge 0} } \widehat{\mathcal Z}_{ES}^{\tot}(m) q^m.
\end{equation}

\begin{theorem} \label{theo:B} Let the notations be as above. Then $\widehat{\Theta}_{ES}(\tau) $ is a (holomorphic) modular form of $\Gamma_0(D)$ of weight $2$ (trivial character) with values in 
   $\widehat{\CH}_\C^1(\mathcal S^*)$. 
\end{theorem}
\begin{proof} By \cite[Theorem 1.4]{ES} (more precisely its $\phi_0$-component), $\widehat{\Theta}_K(\tau) - \widehat{\Theta}_{ES}(\tau) $
is modular. Now Theorem \ref{theo:K} implies that  $\widehat{\Theta}_{ES}$ is modular.
\end{proof}

Since the Eisenstein series $E_2(\tau)= \sum_{m=0} \sigma_1(m) q^m + \frac{1}{8 \pi v} $ is a modular form of $\SL_2(\Z)$ of weight $2$, we have the following corollary.

\begin{corollary} \label{cor:cm+} Let the notations be as above.  Then 
$$
\sum_{m \ge 0} c_m^+(0) q^m
$$
is a modular form for $\Gamma_0(D)$ of weight $2$.
\end{corollary}
\begin{proof} Looking at the generic fiber of $\widehat{\Theta}_{ES}$, we see that 
$$
\Theta_{ES}(\tau) = \left( E_2(\tau) + \sum_{m \in \Z_{\ge 0}} c_m^+(0) q^m  \right) \left(\sum_{\Phi \in  \mathrm{Cusp}([[L]]) }   [ S^*(\Phi)] \right)  + \Theta^{\mathrm{geo}}(\tau)
$$
is a modular form with values in ${\CH}_\C^1( S^*)$ where $S^*$ is the generic fiber of $\cS^*$. Here 
$$
\Theta^{\mathrm{geo}}(\tau) = -\Omega -\frac{1}{8 \pi v} \sum_{\Phi \in  \mathrm{Cusp}([[L]]) } [ S^*(\Phi)]   +\sum_{m \in \Z_{> 0}} Z^*(m) q^m.
$$
Embedding $U(V)$ to $\hbox{SO}(V)$ and applying  \cite[Theorem 1.2]{EGT}, we see that 
$\Theta^{\mathrm{geo}}(\tau)$ is a modular form for $\Gamma_0(D)$ of weight $2$ (trivial character) with values in $\CH_\R^1(S^*)$. Notice  $\left(\sum_{\Phi \in  \mathrm{Cusp}([[L]]) }   [ S^*(\Phi)] \right)$ is non-trivial as its degree is bigger than $0$.  Now the corollary is clear.
\end{proof}

\subsection{Arithmetic theta function with Bruinier Green functions II-Bruinier's normalization} \label{sect:B2}

Similar to Lemma \ref{lem:F} and  \cite[Proposition 3.11]{BF04}),  for each  $m \ge 0$, there is a unique  harmonic  Maass (scalar valued) modular form $f_m$  for $\Gamma_0(D)$ of weight $0$  for each  $m \ge 0$ such that 
\begin{enumerate}

 \item 
\begin{equation}
 f_m^+(\tau) =q^{-m} + \sum_{n \ge  0} c_m(n)q^n,
 \end{equation}

\item  the constant term of $f_m^+$ is zero, 

\item At another cusp $P=\gamma (\infty) \ne \infty$ for $\Gamma_0(D)$, 
$$
f_m^+(\gamma \tau) =\sum_{n \ge 0} c_{m, \gamma} (n) q^n.
$$
\end{enumerate}
Indeed, the existence is essentially given by \cite[Proposition 3.11]{BF04}) without Condition (2).  Two different ones differ by  a constant. Condition (2) makes it unique. 
Lemma \ref{lem:F} (2) is another way to normalize the weight $0$ harmonic Maass forms $f_{m, \mu}$.

For such an $f_m$, let 
$$
\tilde f_m = \sum_{\gamma \in \Gamma_0(D) \backslash \SL_2(\Z)}
 f(\gamma \tau) \rho_L(\gamma^{-1}) \phi_0
 $$
be the associated harmonic Maass modular form for $\SL_2(\Z)$ of weight $0$ and representation $\rho_L$. By \cite[Proposition 6.1.2]{BHKRY1}, we also have  
$$
\tilde f_m ^+  =q^{-m} \phi_0  + \sum_{n \ge 0} \tilde c_m(n) q^n, \quad  \tilde c_m(n) \in S_L.
$$
Let 
\begin{equation} \label{eq:BGreen}
G_B(m) = \int^{reg} \tilde f_m \theta_L d\mu(\tau)
\end{equation} 
be the regularized theta lifting of $\tilde f_m$ as in (\ref{eq:ESGreen}). 
According to \cite[Corollary 4.12]{BHY}, $G_B(m)$ is a Green function for 
$$  Z^*(m)  -2 \sigma_1(m) \sum_{\Phi \in \hbox{Cusp}[[L]]} [  S^*(\Phi)]  
$$   
Similar to \cite[Section 7.3]{BHKRY1}, we define  an arithmetic divisor 
\begin{equation}
\widehat{\mathcal Z}_B^{\tot}(m) = (\mathcal Z^*(m)  -2 \sigma_1(m) \sum_{\Phi \in \hbox{Cusp}[[L]]} [\mathcal S^*(\Phi)],  G_B(m))  +   \begin{cases}
   0 &\ff \,  m >0,
  \\
   \widehat{\Omega}^{-1} + (0, -\log D) &\ff \, m=0.
   \end{cases}
\end{equation}
We again form the arithmetic theta function
\begin{equation}
    \widehat{\Theta}_B(\tau) 
     = \sum_{m \ge 0} \widehat{\mathcal Z}_B^{\tot}(m) q^m. 
\end{equation}

\begin{theorem} The arithmetic theta function $\widehat{\Theta}_B(\tau)$ is a modular form for $\Gamma_0(D)$ of weight $2$ with values in $\widehat{\CH}_\C^1(\mathcal S^*)$.
\end{theorem}
\begin{proof} It suffices to prove 
\begin{equation} \label{eq:Difference}
\widehat{\Theta}_B(\tau) - \widehat{\Theta}_{ES}(\tau) 
=\sum_{m \ge 0} 
\left(c_m^+(0) \sum_{\Phi} \mathcal S^*(\Phi),  \Phi_m \right) q^m
\end{equation}
is modular. Here  $\Phi_m $ is the regularized theta lifting of  $\tilde f_m -F_{m,0} \in M_0(\rho_L)$, and is a Green function of  $c_m^+(0) \sum_\Phi \mathcal S^*(\Phi)$. By Borcherd's well-known theorem theorem on Borcherds product (\cite{Bor98}),  $\Phi_m= - \log \|\Psi_m \|^2 $, where   
$\Psi_m$ is the Borcherds lifting of  $\tilde f_m -F_{m,0} \in M_0(\rho_L)$ and has Borcherds product expansion. This implies that $\Psi_m$ is a rational section of $\omega_\Q^{c c_m^+(0)}$ for some positive rational number $c>0$.  The $q$-expansion principal implies that it extends to a rational section of $\omega^{c c_m^+(0)}$, and the right hand side of (\ref{eq:Difference}) becomes 
$$
c \sum c_m^+(0) q^m \hat{\omega} 
$$
which is modular by Corollary \ref{cor:cm+}. 
\end{proof}

\bibliographystyle{alpha}
\bibliography{reference}
\end{document}